\documentclass[12pt,fleqn,twoside]{article}
\usepackage{amssymb,latexsym,amsmath,amsfonts}
\usepackage{graphicx}

    \advance\textheight by \topskip

    \numberwithin{equation}{section}

    \def\ds{\displaystyle}
    
    \def\Re{{\rm Re \,  }}
    \def\Im{{\rm Im \, }}
    \def\Ai{{\rm Ai  }}
    \def\Bi{{\rm Bi  }}

    \DeclareMathOperator*{\supp}{supp}

    \DeclareMathOperator*{\Tr}{Tr}

    \newtheorem{theorem}{Theorem}[section]
    \newtheorem{lemma}[theorem]{Lemma}
    \newtheorem{corollary}[theorem]{Corollary}
    \newtheorem{proposition}[theorem]{Proposition}
    
    \newtheorem{Definition}[theorem]{Definition}
    
    \newtheorem{Remark}[theorem]{Remark}
    \newenvironment{remark}{\begin{Remark}\rm}{\end{Remark}}
    \newtheorem{Example}[theorem]{Example}

    \newenvironment{proof}%
    {\rm \trivlist \item[\hskip \labelsep{\bf Proof. }]}%
    {\hspace*{\fill}$\Box$\endtrivlist}
    \newenvironment{varproof}%
    {\rm \trivlist \item[\hskip \labelsep{\bf Proof}]}%
    {\hspace*{\fill}$\Box$\endtrivlist}

\begin{document}

\title{Critical edge behavior in unitary random matrix ensembles and the
    thirty fourth Painlev\'e transcendent}
\author{A.R. Its \\
{\em Department of Mathematical Sciences} \\
{\em Indiana University -- Purdue University Indianapolis} \\
{\em Indianapolis IN 46202-3216, U.S.A.} \\
itsa@math.iupui.edu
 \\[15pt]
A.B.J. Kuijlaars and J. \"{O}stensson \\
{\em Department of Mathematics} \\
{\em Katholieke Universiteit Leuven} \\
{\em Celestijnenlaan 200B} \\
{\em 3001 Leuven, Belgium} \\
arno.kuijlaars@wis.kuleuven.be \\
ostensson@wis.kuleuven.be}

\maketitle

\begin{abstract}
We describe a new universality class for unitary invariant
random matrix ensembles. It arises in the double scaling
limit of ensembles of random $n \times n$ Hermitian matrices
$Z_{n,N}^{-1} |\det M|^{2\alpha} e^{-N \Tr V(M)} dM$
with $\alpha > -1/2$, where the factor $|\det M|^{2\alpha}$ induces critical
eigenvalue behavior near the origin.
Under the assumption that the limiting mean eigenvalue density
associated with $V$ is regular, and that the origin is a right endpoint
of its support, we compute the limiting eigenvalue correlation kernel
in the double scaling limit as $n, N \to
\infty$ such that $n^{2/3}(n/N-1) = O(1)$.
We use the Deift-Zhou steepest descent method for the
Riemann-Hilbert problem for polynomials on the line orthogonal
with respect to the weight $|x|^{2\alpha} e^{-NV(x)}$. Our main
attention is on the construction of a local parametrix near the
origin by means of the $\psi$-functions associated with a distinguished
solution of the Painlev\'e XXXIV equation. This solution is related to
a particular solution of the Painlev\'e II equation, which however is different
from the usual Hastings-McLeod solution.
\end{abstract}

\vspace*{1cm}
{\em 2000 Mathematics Subject Classification}:  15A52, 33E17, 34M55

\pagestyle{myheadings}
\thispagestyle{plain}
\markboth{A.R. ITS, A.B.J. KUIJLAARS, and J. \"OSTENSSON}{CRITICAL EDGE BEHAVIOR IN RANDOM MATRIX ENSEMBLES}

\newpage
\section{Introduction and statement of results}

\subsection{Unitary random matrix models}
For $n \in \mathbb N$, $N > 0$, and $\alpha > -1/2$, we consider the unitary
random matrix ensemble
\begin{equation}
\label{randommatrixmodel}
    Z_{n,N}^{-1} |\det M|^{2\alpha} e^{-N \Tr V(M)} \;dM,
\end{equation}
on the space $\mathcal{M}(n)$ of $n\times n$ Hermitian matrices $M$,
where $V$ is real analytic and satisfies
\begin{equation} \label{Vgrowth}
    \lim_{x \to \pm \infty} \frac{V(x)}{\log(x^2+1)} = +\infty.
    \end{equation}
This is a unitary random matrix ensemble in the sense that it is
invariant under conjugation, $M \mapsto U M U^{-1}$, by unitary
matrices $U$. As is well-known \cite{Deift,Mehta}, it induces the
following probability density on the $n$ eigenvalues $x_1, \ldots, x_n$ of $M$
\begin{equation}
P^{(n,N)}(x_1,\ldots,x_n) = \widehat{Z}_{n,N}^{-1} \prod_{j=1}^{n}
|x_j|^{2\alpha} e^{-NV(x_j)} \prod_{i < j} |x_i - x_j|^2.
\end{equation}
The eigenvalue
distribution is determinantal with kernel $K_{n,N}$ built out of the
polynomials $p_{j,N}(x) = \kappa_{j,N}\,x^j + \cdots$, $\kappa_{j,N} > 0$,
orthonormal with
respect to the weight $|x|^{2\alpha} e^{-NV(x)}$ on $\mathbb R$.
Indeed, as shown by Dyson, Gaudin, and Mehta, see e.g.~\cite{Deift,Dyson,Mehta},
for any $m = 1, \ldots, n-1$, the $m$-point correlation function
\begin{equation}
R_m^{(n,N)}(x_1, \ldots, x_m) \equiv \frac{n!}{(n-m)!}
\int_{-\infty}^{\infty} \cdots \int_{-\infty}^{\infty}
P^{(n,N)}(x_1, \ldots, x_n)\,dx_{m+1} \cdots dx_n
\end{equation}
is given by
\begin{equation}
R_m^{(n,N)}(x_1,\ldots,x_m) = \det\left(K_{n,N}(x_i,x_j)\right)_{1 \leq i,j \leq m},
\end{equation}
where
\begin{equation} \label{Christoffel-Darboux-kernel}
    K_{n,N}(x,y) = |x|^{\alpha} |y|^{\alpha} e^{-\frac{1}{2}N(V(x) + V(y))}
    \sum_{j=0}^{n-1} p_{j,N}(x)\,p_{j,N}(y).
\end{equation}

In the limit $n, N \to \infty$, $n/N \to 1$, the global eigenvalue
regime is determined by $V$ as follows. The equilibrium measure
$\mu_V$ for $V$ is the unique minimizer of
\begin{equation} \label{energyfunctional}
    I_V(\mu) = \iint \log \frac{1}{|x-y|} d\mu(x) d\mu(y) +
    \int V(x) d\mu(x) \end{equation}
taken over all Borel probability measures $\mu$ on $\mathbb{R}$.
Since $V$ is real analytic we have that $\mu_V$
is supported on a finite union of disjoint intervals \cite{DKM}, and it has
a density $\rho_V$ such that
\[ \lim_{n, N \to \infty, n/N \to 1} \frac{1}{n} K_{n,N}(x,x)
    = \rho_V(x), \qquad x \neq 0. \]
The limiting mean eigenvalue density is independent
of $\alpha$.

The factor $|\det M|^{2\alpha}$  changes the local eigenvalue behavior
near $0$. This is reflected in the local scaling limits of $K_{n,N}$
around $0$ that do depend on $\alpha$. If $0$ is in the bulk of the
spectrum and $\rho_V(0) > 0$, then instead of the usual sine kernel
we get a Bessel kernel depending on $\alpha$ \cite{KV}. If $0$ is in
the bulk and $\rho_V(0) = \rho_V'(0) = 0$, $\rho_V''(0) > 0$, then the
local scaling limits of the kernel near $0$ are associated with the
Hastings-McLeod solution of the Painlev\'e II equation $q'' = sq +
2q^3 - \alpha$ \cite{CKV}.

In this paper we study the effect of $\alpha$ in case $0$ is an
endpoint of the spectrum which is such that the density $\rho_V$ vanishes
like a square root at $0$.
For $\alpha = 0$ the scaling limit is the well-known Airy kernel,
see the papers \cite{BB,For,Moore,TW} and also \cite{BI1,DG}, and so we are
asking the question: What is the $\alpha$-generalization of the Airy kernel?

For $\alpha > -1/2$, we have found a new one-parameter family of
limiting kernels as stated in Theorem \ref{theorem1} below.
In Theorem \ref{theorem1} we also assume that
the eigenvalue density $\rho_V$ is regular, which means the following.
\begin{itemize}
\item The function $x \mapsto 2 \int \log |x-s| \rho_V(s) ds - V(x)$ defined
for $x \in \mathbb R$, assumes its maximum value only on the support of $\rho_V$.
\item The density $\rho_V$ is positive on the interior of its support.
\item The density $\rho_V$ vanishes like a square root at each of the endpoints
of its support.
\end{itemize}

\begin{theorem} \label{theorem1}
For every $\alpha > -1/2$, there exists a one-parameter family of kernels
$K_{\alpha}^{edge}(x,y;s)$ such that the following holds.
Let $V$ be a real analytic external field on $\mathbb R$ such that its
mean limiting eigenvalue density $\rho_V$ is regular.
Suppose that $0$ is a right endpoint of
the support of $\rho_V$ so that for some constant $c_1 = c_{1,V} >0$
\begin{equation} \label{c1def}
    \rho_V(x) \sim \frac{c_1}{\pi} |x|^{1/2} \qquad \mbox{ as } x \to 0-.
    \end{equation}
Then there exists a second constant $c_2 = c_{2,V} > 0$ such that
\begin{equation} \label{KnNlimit}
    \lim_{n,N \to \infty} \frac{1}{(c_1n)^{2/3}}
    K_{n,N}\left(\frac{x}{(c_1 n)^{2/3}}, \frac{y}{(c_1n)^{2/3}}\right) =
        K_{\alpha}^{edge}(x,y;s) \end{equation}
whenever $n, N \to \infty$ such that
\begin{equation} \label{nNlimit}
    \lim_{n \to \infty} n^{2/3} \left(\frac{n}{N} - 1\right) = L \in \mathbb R
    \end{equation}
and $s = - c_{2,V} L$.
\end{theorem}

For $\alpha = 0$, the limiting kernels reduce to the kernel
\begin{equation} \label{K0}
K_0^{edge}(x,y;s) = \frac{\Ai(x+s) \Ai'(y+s)-\Ai'(x+s)\Ai(y+s)}{x-y},
\end{equation}
which is the (shifted) Airy kernel from random matrix theory mentioned above,
see also Subsection \ref{conclusion1} below.
For $\alpha \neq 0$, a new type of special functions is needed to describe
the limiting kernel $K_{\alpha}^{edge}(x,y;s)$. This description is given in
the next subsections.

\subsection{The model RH problem} \label{modelRHP1}
We describe $K_{\alpha}^{edge}(x,y;s)$
through the solution of a special Riemann-Hilbert (RH) problem,
that we will refer to as the model RH problem.

The model RH problem is posed on a contour $\Sigma$ in an auxiliary
$\zeta$-plane, consisting of four rays $\Sigma_1 = \{ \arg \zeta =
0\}$, $\Sigma_2 = \{ \arg \zeta = 2\pi/3\}$, $\Sigma_3 = \{ \arg
\zeta = \pi\}$, and $\Sigma_4 = \{ \arg \zeta = -2\pi/3\}$ with
orientation as shown in Figure \ref{figure1}. As usual in RH
problems, the orientation defines a $+$ and a $-$ side on each part
of the contour, where the $+$-side is on the left when traversing
the contour according to its orientation. For a function $f$ on
$\mathbb C \setminus \Sigma$, we use $f_{\pm}$ to denote its
limiting values on $\Sigma$ taken from the $\pm$-side, provided such
limiting values exist.
The contour $\Sigma$ divides the
complex plane into four sectors $\Omega_j$ also shown in the figure.

\begin{figure}[t]
\ifx\JPicScale\undefined\def\JPicScale{1}\fi
\unitlength \JPicScale mm
\begin{picture}(80,50)(-40,15)
\linethickness{0.3mm} \put(10,40){\line(1,0){70}}
\linethickness{0.3mm}
\multiput(20,15)(0.12,0.12){208}{\line(1,0){0.12}}
\put(65,42.5){\makebox(0,0)[cc]{$\Sigma_1$}}

\put(35,55){\makebox(0,0)[cc]{$\Sigma_2$}}

\put(25,42.5){\makebox(0,0)[cc]{$\Sigma_3$}}

\put(35,25){\makebox(0,0)[cc]{$\Sigma_4$}}

\put(55,55){\makebox(0,0)[cc]{$\Omega_1$}}

\put(20,50){\makebox(0,0)[cc]{$\Omega_2$}}

\put(20,30){\makebox(0,0)[cc]{$\Omega_3$}}

\put(55,25){\makebox(0,0)[cc]{$\Omega_4$}}

\put(45,37.5){\makebox(0,0)[cc]{$0$}}


\thicklines \put(64,40){\line(1,0){1}}
\put(65,40){\vector(1,0){0.12}} \thicklines
\multiput(20,65)(0.12,-0.12){208}{\line(1,0){0.12}} \thicklines
\put(24,40){\line(1,0){1}} \put(25,40){\vector(1,0){0.12}}
\thicklines \multiput(32,53)(0.12,-0.12){8}{\line(1,0){0.12}}
\put(33,52){\vector(1,-1){0.12}} \thicklines
\multiput(32,27)(0.12,0.12){8}{\line(1,0){0.12}}
\put(33,28){\vector(1,1){0.12}}




\linethickness{0.3mm}
\multiput(48.06,40.5)(0.06,-0.5){1}{\line(0,-1){0.5}}
\multiput(47.93,40.99)(0.13,-0.49){1}{\line(0,-1){0.49}}
\multiput(47.73,41.46)(0.1,-0.23){2}{\line(0,-1){0.23}}
\multiput(47.46,41.89)(0.14,-0.21){2}{\line(0,-1){0.21}}
\multiput(47.13,42.27)(0.11,-0.13){3}{\line(0,-1){0.13}}
\multiput(46.74,42.6)(0.13,-0.11){3}{\line(1,0){0.13}}
\multiput(46.31,42.88)(0.21,-0.14){2}{\line(1,0){0.21}}
\multiput(45.85,43.08)(0.23,-0.1){2}{\line(1,0){0.23}}
\multiput(45.36,43.22)(0.49,-0.14){1}{\line(1,0){0.49}}
\multiput(44.86,43.28)(0.5,-0.06){1}{\line(1,0){0.5}}
\multiput(44.35,43.27)(0.51,0.01){1}{\line(1,0){0.51}}
\multiput(43.85,43.18)(0.5,0.09){1}{\line(1,0){0.5}}
\multiput(43.37,43.02)(0.48,0.16){1}{\line(1,0){0.48}}
\multiput(42.91,42.79)(0.23,0.11){2}{\line(1,0){0.23}}
\multiput(42.5,42.5)(0.21,0.15){2}{\line(1,0){0.21}}

\put(47.75,45.5){\makebox(0,0)[cc]{$2\pi / 3$}}
\end{picture}
\caption{\label{figure1}Contour for the model RH problem.}
\end{figure}
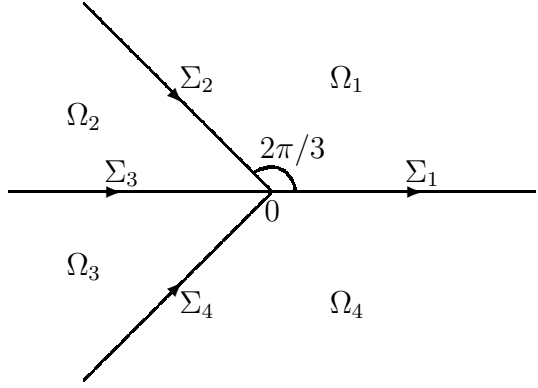

The model RH problem reads as follows.
\paragraph{Riemann-Hilbert problem for $\Psi_{\alpha}$}
\begin{enumerate}
\item[\rm (a)] $\Psi_{\alpha} : \mathbb{C} \setminus \Sigma  \to
    \mathbb C^{2\times 2}$ is analytic.
\item[\rm (b)] $\Psi_{\alpha,+}(\zeta) = \Psi_{\alpha,-}(\zeta)
    \begin{pmatrix} 1 & 1 \\ 0 & 1 \end{pmatrix}$,
    for $\zeta \in \Sigma_1$,

    $\Psi_{\alpha,+}(\zeta) = \Psi_{\alpha,-}(\zeta)
    \begin{pmatrix} 1 & 0 \\ e^{2 \alpha \pi i} & 1 \end{pmatrix}$,
    for $\zeta \in \Sigma_2$,

    $\Psi_{\alpha,+}(\zeta) = \Psi_{\alpha,-}(\zeta)
    \begin{pmatrix} 0 & 1 \\ -1 & 0  \end{pmatrix}$,
    for $\zeta \in \Sigma_3$,

    $\Psi_{\alpha,+}(\zeta) = \Psi_{\alpha,-}(\zeta)
    \begin{pmatrix} 1 & 0 \\ e^{-2\alpha \pi i} & 1 \end{pmatrix}$,
    for $\zeta \in \Sigma_4$.
\item[\rm (c)] $\Psi_{\alpha}(\zeta) = \zeta^{-\sigma_3/4}
    \frac{1}{\sqrt{2}} \begin{pmatrix} 1 & i \\ i & 1 \end{pmatrix}
      (I + O(1/\zeta^{1/2})) e^{-(\frac{2}{3} \zeta^{3/2} + s \zeta^{1/2})\sigma_3}$
    as $\zeta \to \infty$.
    Here $\sigma_3 = \left(\begin{smallmatrix} 1 & 0 \\ 0 & -1 \end{smallmatrix}\right)$
    is the third Pauli matrix.
\item[\rm (d)] $\Psi_{\alpha}(\zeta) = O\begin{pmatrix} |\zeta|^{\alpha} & |\zeta|^{\alpha} \\
    |\zeta|^{\alpha} & |\zeta|^{\alpha}
\end{pmatrix}$ as $\zeta \to 0$, if $\alpha < 0$; and

    $\Psi_{\alpha}(\zeta) = \left\{ \begin{array}{ll}
    O\begin{pmatrix} |\zeta|^{\alpha} & |\zeta|^{-\alpha} \\
    |\zeta|^{\alpha} & |\zeta|^{-\alpha} \end{pmatrix}
    & \mbox{as $\zeta \to 0$ with $\zeta \in \Omega_1 \cup \Omega_4$}, \\[10pt]
    O\begin{pmatrix} |\zeta|^{-\alpha} & |\zeta|^{-\alpha} \\
    |\zeta|^{-\alpha} & |\zeta|^{-\alpha} \end{pmatrix}
    & \mbox{as $\zeta \to 0$ with $\zeta \in \Omega_2 \cup \Omega_3$},
    \end{array} \right.$
    if $\alpha \geq 0$.
\end{enumerate}
Here, and in what follows, the $O$-terms are taken entrywise.
Note that the RH problem depends on a parameter $s$ through the asymptotic condition
at infinity. If we want to emphasize the dependence on $s$ we will write
$\Psi_{\alpha}(\zeta;s)$ instead of $\Psi_{\alpha}(\zeta)$.

The model RH problem is not uniquely solvable.
Indeed, if $\Psi_{\alpha}$ is a solution, then
$\left(\begin{smallmatrix} 1 & 0 \\ \eta & 1 \end{smallmatrix} \right) \Psi_{\alpha}$
is also a solution for any $\eta = \eta(s)$, and it turns out that
this is the only freedom we have (see Proposition \ref{nonunique}).

\begin{theorem} \label{theorem2}
The model RH problem is solvable for every $s \in \mathbb R$.
Let $\Psi_{\alpha}$ be a solution of the model RH problem
and put
\begin{equation} \label{psi12def}
    \begin{pmatrix} \psi_1(x;s) \\[5pt] \psi_2(x;s) \end{pmatrix}
    = \left\{ \begin{array}{ll}
        \Psi_{\alpha,+}(x;s) \begin{pmatrix} 1 \\ 0 \end{pmatrix},
        & \quad \mbox{ for } x > 0, \\
        \Psi_{\alpha,+}(x;s) e^{-\alpha \pi i \sigma_3}
        \begin{pmatrix} 1 \\ 1 \end{pmatrix},
        & \quad \mbox{ for } x < 0.
        \end{array} \right.
        \end{equation}
Then the limiting kernel $K_{\alpha}^{edge}(x,y;s)$
can be written in the ``integrable form''
\begin{equation}
\label{Kintform}
K_{\alpha}^{edge}(x,y;s)
     = \frac{\psi_2(x;s) \psi_1(y;s) - \psi_1(x;s) \psi_2(y;s)}{2\pi i (x-y)}.
\end{equation}
\end{theorem}

The function $\psi_2$ depends
on the particular choice of solution $\Psi_{\alpha}$ to the model RH problem.
Indeed, for any $\eta$ we have that the mapping $\Psi_{\alpha} \mapsto
\left(\begin{smallmatrix} 1 & 0 \\ \eta & 1 \end{smallmatrix} \right) \Psi_{\alpha}$ leaves
$\psi_1$ invariant and changes $\psi_2$ to $\psi_2 + \eta \psi_1$.
However, this does not change the expression \eqref{Kintform}
for the kernel $K_{\alpha}^{edge}(x,y;s)$.

It follows from \eqref{psi12def} and part (c) of the model RH problem that
$\psi_1$ and $\psi_2$ have the  asymptotic behavior
\begin{alignat}{2}
\label{asympt1}
    \psi_1(x;s) & =
    \frac{1}{\sqrt{2} x^{1/4}} e^{-\frac{2}{3} x^{3/2} - s x^{1/2}}( 1 + O(x^{-1/2})), \\
    \psi_2(x;s) & =
    \frac{i x^{1/4}}{\sqrt{2}} e^{-\frac{2}{3} x^{3/2} - s x^{1/2}} (1 + O(x^{-1/2})),
\end{alignat}
as $x \to +\infty$, and
\begin{alignat}{2}
    \psi_1(x;s) & =
    \sqrt{2} |x|^{-1/4} \cos\left(\frac{2}{3} |x|^{3/2} - s |x|^{1/2} - \alpha \pi  - \pi/4\right) + O(x^{-3/4}), \\
\label{asympt4}
    \psi_2(x;s) & =
    -i \sqrt{2} |x|^{1/4} \sin\left(\frac{2}{3} |x|^{3/2} - s|x|^{1/2} - \alpha \pi  - \pi/4\right) + O(x^{-1/4}),
\end{alignat}
as $x \to -\infty$.

\begin{remark}
The kernel $K_{\alpha}^{edge}(x,y;s)$ describes an edge effect for
the random matrix ensemble \eqref{randommatrixmodel}. If we assume
that $0$ is the rightmost point in the support of $\rho_V$, and if
given $M$ we let $\lambda_{\max}(M)$ denote its largest eigenvalue,
then it follows under the assumptions of Theorem \ref{theorem1}, in
particular the limit assumption \eqref{nNlimit}, that
\begin{equation} \label{largesteigenvalue}
\lim_{n,N \to \infty} \mathbb P
    \left( (c_1n)^{2/3} \lambda_{\max} \leq t \right)
        = \det \left(1 - \left.K_{\alpha,s}\right|_{(t,\infty)} \right),
\end{equation}
where $\left.K_{\alpha,s}\right|_{(t,\infty)}$ is the trace class operator in
$L^2(t,\infty)$ with kernel $K_{\alpha}^{edge}(x,y;s)$.
To prove \eqref{largesteigenvalue} one must show that the
operator with kernel $\frac{1}{(c_1n)^{2/3}} K_{n,N}\left(\frac{x}{(c_1n)^{2/3}},
\frac{y}{(c_1n)^{2/3}}\right)$ converges in the trace class norm on
$L^2(t,\infty)$ to the operator with kernel $K_{\alpha}^{edge}(x,y;s)$.
This requires good estimates on the rate of convergence in
\eqref{KnNlimit}, which can be established as in \cite{DG}.

For $\alpha = 0$, the kernel is the (shifted) Airy kernel, and the
Fredholm determinant \eqref{largesteigenvalue} has an equivalent
expression in terms of a special solution of the Painlev\'e II
equation. The resulting distribution is the famous Tracy-Widom
distribution \cite{TW,TW2}. It would be very interesting to find an
analogous expression for general $\alpha$. The connection to the
model RH problem given in Theorem \ref{theorem2} can be used in
obtaining such an expression, following the approach of \cite{BD}
and \cite{HI}.
We are planning to address this question in a future publication.
\end{remark}

\subsection{Connection with the Painlev\'e XXXIV equation}

The model RH problem is related to a special solution of
the equation number XXXIV from the list of Painlev\'e and Gambier \cite{Ince},
\begin{equation}
\label{painleve34}
    u'' = 4 u^2 + 2s u + \frac{(u')^2 - (2\alpha)^2}{2u}.
\end{equation}
All solutions of \eqref{painleve34} are meromorphic in the complex plane.

\begin{theorem} \label{theorem3}
Let $\Psi_{\alpha}(\zeta;s)$ be a solution of the model RH problem.
Then
\begin{equation} \label{usolution}
    u(s) = -\frac{s}{2} - i \frac{d}{ds}
     \lim_{\zeta \to \infty} \left[ \zeta
    \left(\Psi_{\alpha}(\zeta;s) e^{\left(\frac{2}{3}\zeta^{3/2} + s \zeta^{1/2}\right)\sigma_3}
        \frac{1}{\sqrt{2}} \begin{pmatrix} 1 & - i \\ - i & 1 \end{pmatrix}
        \zeta^{\sigma_3/4} \right)_{12} \right]
\end{equation}
exists and satisfies \eqref{painleve34}.
The function \eqref{usolution} is a global solution
of \eqref{painleve34} (i.e., it does not have poles on the
real line), and it does not depend
on the particular solution of the model RH problem.
\end{theorem}

The connection with the Painlev\'e XXXIV equation leads
to the following characterization of
$\psi_1$ and $\psi_2$.

\begin{theorem} \label{theorem4}
Let $u$ be the solution of Painlev\'e XXXIV
given by \eqref{usolution}.
Then there exists a solution $\Psi_{\alpha}$ of the model RH problem
so that the functions $\psi_1$ and $\psi_2$ defined by \eqref{psi12def} satisfy the following
system of linear differential equations
\begin{equation}
\label{psisystem}
    \frac{d}{dx} \begin{pmatrix} \psi_1(x;s) \\ \psi_2(x;s) \end{pmatrix} =
    \begin{pmatrix} u'/(2x) & i-iu/x \\
    -i(x+s+u + ((u')^2-(2\alpha)^2)/(4ux)) & - u'/(2x)
    \end{pmatrix}
    \begin{pmatrix} \psi_1(x;s) \\ \psi_2(x;s) \end{pmatrix}
\end{equation}
and have asymptotics \eqref{asympt1}--\eqref{asympt4}.
\end{theorem}

In fact we will prove that for $\zeta \in \mathbb C \setminus \Sigma$,
\begin{equation}
\label{psisystem2}
    \frac{d}{d\zeta} \Psi_{\alpha}(\zeta;s) =
    \begin{pmatrix} u'/(2\zeta) & i-iu/\zeta \\
    -i(\zeta+s+u + ((u')^2-(2\alpha)^2)/(4u\zeta)) & - u'/(2\zeta)
    \end{pmatrix}
        \Psi_{\alpha}(\zeta;s)
\end{equation}
from which \eqref{psisystem} readily follows in view of \eqref{psi12def}.
We emphasize that \eqref{psisystem} and \eqref{psisystem2} hold
for one particular solution of the model RH problem.
Any other solution $\left(\begin{smallmatrix} 1 & 0 \\ \eta & 0 \end{smallmatrix}\right)
\Psi_{\alpha}(\zeta;s)$ also satisfies a system of linear
differential equations, but with matrix
\begin{equation}
    \begin{pmatrix} 1 & 0 \\ \eta & 1 \end{pmatrix}
    \begin{pmatrix} u'/(2\zeta) & i-iu/\zeta \\
    -i(\zeta+s+u + ((u')^2-(2\alpha)^2)/(4u\zeta)) & - u'/(2\zeta)
    \end{pmatrix}
    \begin{pmatrix} 1 & 0 \\ -\eta & 1 \end{pmatrix}.
\end{equation}

In order to make Theorem \ref{theorem4} a genuine, i.e., independent of
the $\Psi_{\alpha}$ RH problem, characterization of $\psi_1$ and $\psi_2$,
we need an independent of formula \eqref{usolution} characterization of the solution
$u(s)$ of equation \eqref{painleve34}. This can be achieved by indicating the
asymptotic behavior of $u(s)$ as $s \to \infty$, cf.\ the characterization of the
Hastings-McLeod solution of Painlev\'e II equation \cite{HMcL}. We discuss this
issue in detail in the last section of the paper, see in particular Proposition
\ref{uasymptotics} and the end of Remark \ref{rem:uasymptotics} where the possible asymptotic characterizations
of the solution $u(s)$ are given.

\subsection{Overview of the rest of the paper}
In Section \ref{section2} we give the proofs of Theorem \ref{theorem1} and
Theorem \ref{theorem2}. We start by presenting the RH problem for orthogonal
polynomials on the line \cite{FIK}. The eigenvalue
correlation kernel $K_{n,N}$ can be explicitly expressed in terms of the
solution of this RH problem \cite{Deift,DKMVZ1}. As in earlier papers,
see e.g.\ \cite{BI1,BI2,CK,CKV,DKMVZ2,DKMVZ1}, we apply the
Deift-Zhou steepest descent method for RH problems, see \cite{DeiftZhou}.
For the local analysis near $0$, we need the model RH problem
for $\Psi_{\alpha}(\zeta;s)$ as introduced in Subsection \ref{modelRHP1}.
We show, following the methodology of \cite{FZ},
that the model RH problem has a solution for every $s \in \mathbb R$.
Then we follow the usual steps in the steepest descent analysis for
RH problems,
which lead us to the proofs of Theorem \ref{theorem1} and \ref{theorem2}.

Section \ref{section3} is devoted to the proofs of Theorem \ref{theorem3}
and Theorem \ref{theorem4}. We start by discussing the RH problem,
in the form due to Flaschka and Newell,
associated with the Painlev\'e II equation $q'' = sq+2q^3-\nu$.
Following \cite{BBIK}, we show that for a special
choice of monodromy data the Flaschka-Newell RH problem is related
to the model RH problem. The parameters in the Painlev\'e equations
are related by $\nu = 2\alpha + 1/2$. The monodromy data corresponds to
a solution of Painlev\'e II which is different from the Hastings-McLeod
solution that has appeared more often in random matrix theory \cite{BI2,CK,CKV,HMcL,TW}.
The known results (asymptotics, Lax pair etc.) for the RH problem for
Painlev\'e II are then transferred to the model RH problem, and then used
to complete the proofs of Theorems \ref{theorem3} and \ref{theorem4}.
In particular it gives rise to the special solution $u$ of the Painlev\'e
XXXIV equation defined by \eqref{usolution}.

In Section \ref{section4} we make some concluding remarks.
For the important special cases $\alpha = 0$ and $\alpha = 1$, we show how
the model RH problem can be explicitly solved in terms of Airy functions,
and how the limiting kernel $K_{\alpha}^{edge}(x,y;s)$ as well as the
special Painlev\'e XXXIV solution $u$ can be
explicitly computed in both cases. Our final remarks concern the characterization
of $u$, in the case of general $\alpha$, through its asymptotic behavior at infinity.

\section{Proof of Theorem \ref{theorem1} and \ref{theorem2}}
\label{section2}
\subsection{The Riemann-Hilbert problem for orthogonal polynomials}
\label{FIKRHproblem}
The RH problem for orthogonal polynomials on the line,
for our particular weight, is the following (cf.\ \cite{FIK}).
\paragraph{Riemann-Hilbert problem for $Y$}
\begin{itemize}
\item $Y : \mathbb C \setminus \mathbb R \to
    \mathbb C^{2\times 2}$ is analytic.
\item $Y_+(x) = Y_-(x)
    \begin{pmatrix} 1 & |x|^{2\alpha} e^{-NV(x)} \\ 0 & 1 \end{pmatrix}$
    for $x \in \mathbb R \setminus \{0\}$, with $\mathbb R$ oriented from left to right.
\item $Y(z) = (I + O(1/z)) \begin{pmatrix} z^n & 0 \\ 0 & z^{-n} \end{pmatrix}$
    as $z \to \infty$.
\item If $\alpha < 0$, then $Y(z) =
O\left(\begin{smallmatrix} 1 & |z|^{2\alpha} \\ 1 & |z|^{2\alpha}
\end{smallmatrix}\right)$ as $z \to 0$. If $\alpha \geq 0$, then
$Y(z) = O\left(\begin{smallmatrix} 1 & 1 \\ 1 & 1
\end{smallmatrix}\right)$ as $z \to 0$.
\end{itemize}

The RH problem has the unique solution
\begin{equation} \label{RHsolution}
Y(z) = \begin{pmatrix}
    \ds \frac{1}{\kappa_{n,N}}\,p_{n,N}(z) &
    \ds \frac{1}{2\pi i\kappa_{n,N}}\,  \int_{\mathbb R} \frac{p_{n,N}(s) |s|^{2\alpha} e^{-NV(s)}}{s-z} ds \\[10pt]
    -2\pi i\,\kappa_{n-1,N}\,p_{n-1,N}(z) &
    \ds -\kappa_{n-1,N} \int_{\mathbb R} \frac{p_{n-1,N}(s) |s|^{2\alpha} e^{-NV(s)}}{s-z} ds
  \end{pmatrix},
\end{equation}
where $p_{j,N}(x) = \kappa_{j,N}\,x^j + \cdots$ is the orthonormal polynomial with respect to the weight
$|x|^{2\alpha} e^{-NV(x)}$.
By \eqref{Christoffel-Darboux-kernel} and the Christoffel-Darboux formula for
orthogonal polynomials, we have
\begin{equation}
    K_{n,N}(x,y) = |x|^{\alpha} |y|^{\alpha} e^{-\frac{1}{2}N(V(x)+V(y))}
\frac{\kappa_{n-1,N}}{\kappa_{n,N}}\,\frac{p_{n,N}(x)\,p_{n-1,N}(y)
- p_{n-1,N}(x)\,p_{n,N}(y)}{x-y}.
\end{equation}
Thus, using \eqref{RHsolution} and the fact that $\det Y \equiv 1$, we may express the eigenvalue correlation kernel
directly in terms of $Y$:
\begin{equation}
\label{CD-kernel}
K_{n,N}(x,y) = \frac{1}{2\pi i(x-y)}
    |x|^{\alpha} |y|^{\alpha} e^{-\frac{1}{2}N(V(x)+ V(y))}
    \begin{pmatrix} 0 & 1 \end{pmatrix}
        Y_+^{-1}(y) Y_+(x) \begin{pmatrix} 1 \\ 0 \end{pmatrix}.
\end{equation}

The main idea for the proof of Theorems \ref{theorem1} and
\ref{theorem2} is to apply the
powerful steepest descent analysis for RH problems of Deift and Zhou
\cite{DeiftZhou} to the RH problem satisfied by $Y$. In the case at
hand it consists of constructing a sequence of invertible
transformations $Y \mapsto T \mapsto S \mapsto R$, where the
matrix-valued function $R$ is close to the identity. By unfolding the above
transformations asymptotics for $Y$ and thus, in view of \eqref{CD-kernel}, for
$K_{n,N}$ in various regimes may be derived. Our main attention will
be devoted to the local behavior of $Y$ near $0$. Around $0$ we
construct a local parametrix with the help of the model RH problem,
which we next discuss in more detail.

\subsection{The model RH problem} \label{modelRHP2}

The model RH problem is not uniquely solvable.
\begin{proposition} \label{nonunique}
    Let $\Psi_{\alpha}$ be a solution of the model RH problem.
    Then the following hold.
    \begin{enumerate}
    \item[\rm (a)] $\det \Psi_{\alpha} \equiv 1$.
    \item[\rm (b)] For any $\eta \in \mathbb R$ (which may depend on $s$),
    we have that $\left(\begin{smallmatrix} 1 & 0 \\ \eta & 1 \end{smallmatrix}\right) \Psi_{\alpha}$
    also solves the model RH problem.
    \item[\rm (c)] Any two solutions are related as in part {\rm (b)}, i.e., if
    $\Psi_{\alpha}^{(1)}$ and $\Psi_{\alpha}^{(2)}$ are two solutions of the model RH problem,
    then $\Psi_{\alpha}^{(2)} = \left(\begin{smallmatrix} 1 & 0 \\ \eta & 1
    \end{smallmatrix}\right) \Psi_{\alpha}^{(1)}$ for some $\eta = \eta(s)$.
    \end{enumerate}
\end{proposition}

\begin{proof}
(a) We have that $\det \Psi_{\alpha}$ is analytic in $\mathbb{C}
\setminus \{0\}$, since all jump matrices have determinant one. In
case $\alpha < 0$ we get from  condition (d) of the RH problem
that $\det \Psi_{\alpha}(\zeta) = O(|\zeta|^{2\alpha})$ as $\zeta \to 0$.
Since $2\alpha > -1$ it follows that the singularity at the origin
is removable.
In case $\alpha \geq 0$ we find from condition (d) of the RH problem
that $\det \Psi_{\alpha}(\zeta) = O(1)$ as $\zeta \to 0$ in
$\Omega_1 \cup \Omega_4$. Thus the singularity at the origin
cannot be a pole.  Since $\det \Psi_{\alpha} =
O(|\zeta|^{-2\alpha})$ as $\zeta \to 0$, it cannot
be an essential singularity either and therefore the singularity
at the origin is removable also in this case.
Hence $\det \Psi_{\alpha}$ is entire.
From condition (c) of the RH problem it follows that
$\det \Psi_{\alpha}(\zeta) \to 1$ as $\zeta \to \infty$, and
so part (a) of the proposition follows from Liouville's theorem.

(b) It is clear that $\left(\begin{smallmatrix} 1 & 0 \\ \eta & 1 \end{smallmatrix}\right)
\Psi_{\alpha}$ satisfies the conditions
(a), (b), and (d) of the model RH problem. To establish (c) it is enough to observe that
\begin{alignat*}{2} \begin{pmatrix} 1 & 0 \\ \eta & 1 \end{pmatrix} \zeta^{-\sigma_3/4} \frac{1}{\sqrt{2}}
\begin{pmatrix} 1 & i \\ i & 1 \end{pmatrix} &= \zeta^{-\sigma_3/4} \frac{1}{\sqrt{2}}
\begin{pmatrix} 1 & i \\ i & 1 \end{pmatrix} \left(I + \frac{\eta}{2 \zeta^{1/2}}
\begin{pmatrix} -i & 1 \\ 1 & i \end{pmatrix} \right)\\
&= \zeta^{-\sigma_3/4} \frac{1}{\sqrt{2}}
\begin{pmatrix} 1 & i \\ i & 1 \end{pmatrix} (I + O(1/\zeta^{1/2}))
\end{alignat*}
as $\zeta \to \infty$.

(c) In view of part (a) we know that $\Psi_{\alpha}^{(1)}$ is invertible.
Then $\Psi_{\alpha}^{(2)} (\Psi_{\alpha}^{(1)})^{-1}$ is analytic
in $\mathbb C \setminus \{0\}$ and, by arguments similar to those in the
proof of part (a), the singularity at the origin is removable.
As $\zeta \to \infty$ we get from condition (c) of the model RH problem that
\begin{alignat*}{2}
\Psi_{\alpha}^{(2)}(\zeta) (\Psi_{\alpha}^{(1)}(\zeta))^{-1} &= \zeta^{-\sigma_3/4} \frac{1}{\sqrt{2}}
\begin{pmatrix} 1 & i \\ i & 1 \end{pmatrix} (I + O(\zeta^{-1/2})) \frac{1}{\sqrt{2}} \begin{pmatrix} 1 & -i \\ -i & 1 \end{pmatrix}
\zeta^{\sigma_3/4}\\
&= I + O \begin{pmatrix} \zeta^{-1/2} & \zeta^{-1} \\ 1 & \zeta^{-1/2} \end{pmatrix}.
\end{alignat*}
The statement now follows from Liouville's theorem.
\end{proof}

In the following we will need more information about the behavior at the origin of functions
satisfying properties (a), (b), and (d) of the model RH problem.
The following result is similar to Proposition 2.3 in \cite{CKV}.

\begin{proposition} \label{Psiat0Prop}
Let $\Psi$ satisfy conditions {\rm (a), (b)}, and {\rm (d)} of the
RH problem for $\Psi_{\alpha}$.
Then, with all branches being principal, the following hold.
\begin{itemize}
\item If $\alpha - \frac{1}{2} \notin \mathbb{N}_0$, there exist an analytic matrix-valued
function $E$ and constant matrices $A_j$ such that
\begin{equation}
\label{Psiat0Eq1} \Psi(\zeta) =
E(\zeta)\,\zeta^{\alpha \sigma_3}\,A_j, \quad
\text{ for } \zeta \in \Omega_j.
\end{equation}
Letting $v_j$ denote the jump matrix for
$\Psi$ on $\Sigma_j$, we have
\begin{equation}
\label{Amatrices1}
A_1 = A_4\,v_1, \quad A_1 =
A_2\,v_2, \quad A_3 = A_4\,v_4,
\end{equation}
and
\begin{equation}
\label{Amatrices2}
A_2 =
\begin{pmatrix} \ds \frac{1}{2\cos \alpha \pi} & \ds \frac{1}{2\cos \alpha \pi} \\[10pt]
\ds -e^{\alpha \pi i} & \ds e^{-\alpha \pi i} \end{pmatrix}.
\end{equation}
\item If $\alpha - \frac{1}{2} \in \mathbb{N}_0$, then
$\Psi$ has logarithmic behavior at the origin: There exist
an analytic matrix-valued function $E$ and constant matrices $A_j$ such that
\begin{equation}
\label{Psiat0Eq2} \Psi(\zeta) =
E(\zeta)\,
\begin{pmatrix} \ds \zeta^{\alpha} & \ds \frac{1}{\pi} \zeta^{\alpha} \log \zeta \\[10pt]
\ds 0 & \ds\zeta^{-\alpha}\end{pmatrix} A_j, \quad \text{ for } \zeta \in
\Omega_j.
\end{equation}
Letting $v_j$ denote the jump matrix for $\Psi$ on
$\Sigma_j$, we now have
\begin{equation}
\label{Amatrices3} A_1 = A_4\,v_1, \quad A_1 =
A_2\,v_2, \quad A_3 = A_4\,v_4,
\end{equation}
and
\begin{equation}
\label{Amatrices4} A_2 =
\begin{pmatrix} \ds 0 & \ds e^{3\pi i/4} \\[10pt] \ds e^{\pi i/4} & \ds e^{\pi i/4} \end{pmatrix}.
\end{equation}
\item In all cases it holds that $\det A_j = 1$ and
\begin{equation}
\label{OneElementis0}
(A_1)_{21} = (A_4)_{21} = 0.
\end{equation}
\end{itemize}
\end{proposition}
\begin{proof}
The statement \eqref{OneElementis0} is an immediate consequence of
the explicit formulas for the $A_j$'s.

Consider the case $\alpha - \frac{1}{2} \notin \mathbb{N}_0$. Define
$E$ by \eqref{Psiat0Eq1}, i.e., let
\begin{equation}
\label{Eat0def}
E(\zeta) =
\Psi(\zeta)\,A_j^{-1}\,\zeta^{-\alpha \sigma_3}, \quad
\text{ for } \zeta \in \Omega_j,
\end{equation}
with $A_j$ as in \eqref{Amatrices1}, \eqref{Amatrices2}. Then $E$ is analytic in $\mathbb{C} \setminus \Sigma$.
We now show that $E$ is indeed entire.
The relations \eqref{Amatrices1} and the condition (b) of the model RH problem imply that $E$ is analytic
also on $\Sigma_1 \cup \Sigma_2 \cup \Sigma_4$. Moreover, on $\Sigma_3$
\begin{equation*}
E_-^{-1}(\zeta) E_+(\zeta) = \zeta_-^{\alpha \sigma_3}\,A_3\,v_3\,A_2^{-1}\,\zeta_+^{-\alpha \sigma_3}.
\end{equation*}
Now, by \eqref{Amatrices1}, \eqref{Amatrices2}, and straightforward computation
\begin{equation}
\label{cyclicidentity1}
A_3\,v_3\,A_2^{-1} = A_2\,v_2\,v_1^{-1}\,v_4\,v_3\,A_2^{-1} = e^{2\alpha \pi i\sigma_3}
= \zeta_-^{-\alpha \sigma_3}\,\zeta_+^{\alpha \sigma_3}.
\end{equation}
Thus, $E$ is analytic also on $\Sigma_3$, and therefore in $\mathbb{C} \setminus \{0\}$.

We next show that the singularity at $0$ is removable. If $\alpha <
0$, we see from \eqref{Eat0def} and the condition (d) of the model RH problem, that as $\zeta
\to 0$
\begin{equation*}
E(\zeta) = O\begin{pmatrix} |\zeta|^{\alpha} &
|\zeta|^{\alpha} \\ |\zeta|^{\alpha} & |\zeta|^{\alpha}
\end{pmatrix} O\begin{pmatrix} 1 & 1 \\ 1 & 1 \end{pmatrix}
O\begin{pmatrix} |\zeta|^{-\alpha} & 0 \\ 0 & |\zeta|^{\alpha}
\end{pmatrix} = O\begin{pmatrix} 1 & |\zeta|^{2\alpha} \\ 1 & |\zeta|^{2\alpha} \end{pmatrix},
\end{equation*}
so (since $2\alpha > -1$) the isolated singularity at $0$ is indeed
removable. If $\alpha \geq 0$ and $\zeta \to 0$ in $\Omega_1$ we find in the same way (also
using $(A_1)_{21} = 0$) that
\begin{equation*}
E(\zeta) = O\begin{pmatrix} |\zeta|^{\alpha} &
|\zeta|^{-\alpha} \\ |\zeta|^{\alpha} & |\zeta|^{-\alpha}
\end{pmatrix} O\begin{pmatrix} 1 & 1 \\ 0 & 1 \end{pmatrix}
O\begin{pmatrix} |\zeta|^{-\alpha} & 0 \\ 0 & |\zeta|^{\alpha}
\end{pmatrix} = O\begin{pmatrix} 1 & 1 \\ 1 & 1 \end{pmatrix},
\end{equation*}
so that $E$ is bounded near $0$ in $\Omega_1$
and thus $0$ cannot be a pole. Since $0$ cannot be an essential singularity
either, we conclude that the singularity is indeed removable.

In case $\alpha - \frac{1}{2} \in \mathbb{N}_0$ the proof
is almost identical, only now the equation \eqref{cyclicidentity1} is replaced by
\begin{alignat}{2} \nonumber
    A_3\,v_3\,A_2^{-1} & =
    A_2\,v_2\,v_1^{-1}\,v_4\,v_3\,A_2^{-1} =
    \begin{pmatrix} -1 & -2 i \\ 0 & -1\end{pmatrix} \\
    \label{cyclicidentity2}
    & = \begin{pmatrix} \zeta^{-\alpha} & - \frac{1}{\pi} \zeta^{\alpha} \log \zeta \\
        0 & \zeta^{\alpha} \end{pmatrix}_-
      \begin{pmatrix} \zeta^{\alpha} &  \frac{1}{\pi} \zeta^{\alpha} \log \zeta \\
        0 & \zeta^{-\alpha} \end{pmatrix}_+.
\end{alignat}
\end{proof}

\subsection{Existence of solution to the model RH problem}

We will need that for $s \in \mathbb R$ the model RH problem indeed
has a solution. To prove existence of a solution to the model RH
problem it suffices to prove existence of a (unique) solution
$\Psi_{\alpha}^{(spec)}$ to the special RH problem obtained when the asymptotics (c) at
infinity is replaced by the following stronger condition
\begin{equation} \label{Psispecial}
    \Psi_{\alpha}^{(spec)}(\zeta) = (I + O(1/\zeta)) \zeta^{-\sigma_3/4}
    \frac{1}{\sqrt{2}} \begin{pmatrix} 1 & i \\ i & 1 \end{pmatrix}
      e^{-(\frac{2}{3} \zeta^{3/2} + s \zeta^{1/2})\sigma_3}, \end{equation}
as $\zeta \to \infty$.

A key element in the proof of unique solvability of the RH problem
for $\Psi_{\alpha}^{(spec)}$ is the following vanishing lemma (cf.\ \cite{FZ}).

\begin{proposition}\label{vanishinglemma} \textbf{(vanishing lemma)}
Let $\alpha > -1/2$, $s \in \mathbb{R}$, and put $\theta(\zeta) = \theta(\zeta;s) = \frac{2}{3} \zeta^{3/2} + s \zeta^{1/2}$.
Suppose that $F_{\alpha}$ satisfies the conditions {\rm (a), (b)}, and {\rm (d)} in the RH
problem for $\Psi_{\alpha}$ but, instead of condition {\rm (c)},
has the following behavior at infinity:
\begin{equation}
\label{Psiatinfinity}
F_{\alpha}(\zeta) = O(1/\zeta) \zeta^{-\sigma_3/4}
    \frac{1}{\sqrt{2}} \begin{pmatrix} 1 & i \\ i & 1 \end{pmatrix} e^{-\theta(\zeta) \sigma_3},
\end{equation}
as $\zeta \to \infty$. Then $F_{\alpha} \equiv 0$.
\end{proposition}
\begin{proof}
The ideas of the proof are similar in spirit to those in Deift et
al.\ \cite{DKMVZ2}. Let $G_{\alpha}$ be defined as follows:
\begin{equation}
\label{DefinitionGalpha}
G_{\alpha}(\zeta) = \left\{
\begin{array}{ll}
  F_{\alpha}(\zeta) e^{\theta(\zeta) \sigma_3} \begin{pmatrix} 0 & -1 \\ 1 & 0\end{pmatrix},
  &\text{ for } \zeta \in \Omega_1, \\[10pt]
  F_{\alpha}(\zeta) e^{\theta(\zeta) \sigma_3}
  \begin{pmatrix} 1 & 0 \\ e^{2 \alpha \pi i} e^{2\theta(\zeta)}& 1\end{pmatrix}
  \begin{pmatrix} 0 & -1 \\ 1 & 0\end{pmatrix}, &\text{ for } \zeta \in \Omega_2, \\[10pt]
  F_{\alpha}(\zeta) e^{\theta(\zeta) \sigma_3} \begin{pmatrix} 1 & 0 \\ -e^{-2 \alpha \pi i} e^{2\theta(\zeta)} & 1\end{pmatrix},
  &\text{ for } \zeta \in \Omega_3, \\[10pt]
  F_{\alpha}(\zeta) e^{\theta(\zeta) \sigma_3}, &\text{ for } \zeta \in \Omega_4.
\end{array} \right.
\end{equation}

Then $G_{\alpha}$ satisfies the following RH problem.
\paragraph{Riemann-Hilbert problem for $G_{\alpha}$}
\begin{enumerate}
\item[\rm (a)] $G_{\alpha} : \mathbb C \setminus \mathbb R \to \mathbb{C}^{2 \times 2}$ is analytic.
\item[\rm (b)] $G_{\alpha,+}(\zeta) = G_{\alpha,-}(\zeta) v_{G_{\alpha}}(\zeta)$
for $\zeta \in \mathbb R \setminus \{0\}$, where
\begin{equation}
\label{GalphaJump} v_{G_{\alpha}}(\zeta) = \left\{
\begin{array}{ll}
  \begin{pmatrix} e^{-2\theta(\zeta)} & -1 \\ 1 & 0 \end{pmatrix},
  &\text{ for } \zeta > 0,  \\[10pt]
  \begin{pmatrix} 1 & -e^{2 \alpha \pi i} e^{2\theta_+(\zeta)}\\
  e^{- 2\alpha \pi i} e^{2\theta_-(\zeta)} & 0\end{pmatrix}, &\text{ for } \zeta < 0.
\end{array}
\right.
\end{equation}
\item[\rm (c)] $G_{\alpha}(\zeta) = O(\zeta^{-3/4}) \quad \text{ as } \zeta \to \infty.$
\item[\rm (d)] $G_{\alpha}$ has the following behavior near the origin: If $\alpha < 0$,
\begin{equation}
\label{Galphaat01}
G_{\alpha}(\zeta) = O\begin{pmatrix} |\zeta|^{\alpha} & |\zeta|^{\alpha} \\
|\zeta|^{\alpha} & |\zeta|^{\alpha} \end{pmatrix},\quad \text{ as } \zeta \to 0,
\end{equation}
and if $\alpha \geq 0$,
\begin{equation}
\label{Galphaat02}
G_{\alpha}(\zeta) = \left\{ \begin{array}{ll}
    O\begin{pmatrix} |\zeta|^{-\alpha} & |\zeta|^{\alpha} \\
    |\zeta|^{-\alpha} & |\zeta|^{\alpha} \end{pmatrix}
    & \text{ as } \zeta \to 0, \Im \zeta > 0, \\[10pt]
    O\begin{pmatrix} |\zeta|^{\alpha} & |\zeta|^{-\alpha} \\
    |\zeta|^{\alpha} & |\zeta|^{-\alpha} \end{pmatrix}
    & \text{ as } \zeta \to 0, \Im \zeta < 0.
    \end{array} \right.
\end{equation}
\end{enumerate}

The jumps in (b) follow from straightforward computations which uses that
$\theta_+(\zeta) +\theta_-(\zeta) = 0$ for $\zeta < 0$.
The behavior (c) of $G_{\alpha}$ at infinity (uniformly in each sector) follows directly
from \eqref{Psiatinfinity}, \eqref{DefinitionGalpha}, and the fact that $\Re \theta(\zeta) < 0$
for $\zeta \in \Omega_2 \cup \Omega_3$.
The behavior \eqref{Galphaat01} at the origin is immediate from the condition (d) of the RH problem for $F_{\alpha}$,
and so is the behavior \eqref{Galphaat02} if $\zeta \to 0$ with $\zeta \in \Omega_1 \cup \Omega_4$.
To prove \eqref{Galphaat02} if $\zeta \to 0$ with $\zeta \in \Omega_2 \cup \Omega_3$,
we need Proposition \ref{Psiat0Prop}.
Consider first the case $\alpha - \frac{1}{2} \notin \mathbb{N}_0$ and $\zeta \in \Omega_2$.
Then we have, using \eqref{DefinitionGalpha}, \eqref{Psiat0Eq1}, \eqref{Amatrices1},
and \eqref{OneElementis0}
\begin{alignat*}{2}
G_{\alpha}(\zeta) &= F_{\alpha}(\zeta) e^{\theta(\zeta) \sigma_3}
  \begin{pmatrix} 1 & 0 \\ e^{2 \alpha \pi i} e^{2\theta(\zeta)}& 1\end{pmatrix}
  \begin{pmatrix} 0 & -1 \\ 1 & 0\end{pmatrix}\\
  &= E(\zeta) \zeta^{\alpha \sigma_3} A_2
  \begin{pmatrix} 1 & 0 \\ e^{2 \alpha \pi i} & 1\end{pmatrix} e^{\theta(\zeta) \sigma_3}
  \begin{pmatrix} 0 & -1 \\ 1 & 0\end{pmatrix}\\
  &= E(\zeta) \zeta^{\alpha \sigma_3} A_2
  \begin{pmatrix} 1 & 0 \\ e^{2 \alpha \pi i} & 1\end{pmatrix} \begin{pmatrix} 0 & -1 \\ 1 & 0 \end{pmatrix} e^{-\theta(\zeta) \sigma_3}\\
  &= E(\zeta) \zeta^{\alpha \sigma_3} A_1 \begin{pmatrix} 0 & -1 \\ 1 & 0 \end{pmatrix} e^{-\theta(\zeta) \sigma_3} =
  E(\zeta) \zeta^{\alpha \sigma_3} \begin{pmatrix} * & * \\ 0 & *\end{pmatrix} \begin{pmatrix} 0 & -1 \\ 1 & 0 \end{pmatrix} e^{-\theta(\zeta) \sigma_3},
\end{alignat*}
where $*$ denotes an unspecified constant.
Using the boundedness of $E$ and $\theta$ at the origin, we find \eqref{Galphaat02} as $\zeta \to 0$ in
the sector $\Omega_2$. The case $\zeta \in \Omega_3$ is treated similarly.
Using \eqref{Psiat0Eq2}, \eqref{Amatrices3} instead of \eqref{Psiat0Eq1}, \eqref{Amatrices1},
the same argument works in case $\alpha - \frac{1}{2} \in \mathbb{N}_0$. Note
that in spite of the logarithmic entry in \eqref{Amatrices3}, there are no
logarithmic entries in \eqref{Galphaat02}.

Introduce the auxiliary matrix-valued function
\begin{equation} \label{Halpha1}
H_{\alpha}(\zeta) =
G_{\alpha}(\zeta)\,(G_{\alpha}(\bar{\zeta}))^{*}, \qquad \zeta \in \mathbb C \setminus \mathbb R.
\end{equation}
Then $H_{\alpha}$ is analytic and
\begin{equation} \label{Halpha2}
    H_{\alpha}(\zeta) = O(\zeta^{-3/2}), \qquad \text{as }
    \zeta \to \infty.
\end{equation}
From the condition (d) in the RH problem for $G_{\alpha}$
it follows that $H_{\alpha}$ has the following behavior near the origin:
\begin{equation} \label{Halpha3}
H_{\alpha}(\zeta) = \left\{ \begin{array}{ll}
    O\begin{pmatrix} |\zeta|^{2\alpha} & |\zeta|^{2\alpha} \\
    |\zeta|^{2\alpha} & |\zeta|^{2\alpha} \end{pmatrix}
    & \text{ as } \zeta \to 0, \text{ in case } \alpha < 0, \\[10pt]
    O\begin{pmatrix} 1 & 1 \\
    1 & 1 \end{pmatrix}
    & \text{ as } \zeta \to 0, \text{ in case } \alpha \geq 0.
    \end{array} \right.
\end{equation}
Since $\alpha > -1/2$, we see from \eqref{Halpha2} and \eqref{Halpha3} that each
entry of $H_{\alpha,+}$ is integrable
over the real line, and by Cauchy's theorem and \eqref{Halpha2}
\begin{equation} \label{Halpha4}
    \int_{\mathbb{R}} H_{\alpha,+}(\zeta)\,d\zeta = 0.
\end{equation}
That is, by \eqref{Halpha1},
\begin{equation} \label{VanishingIdentity1}
\int_{\mathbb{R}}
G_{\alpha,+}(\zeta)\,(G_{\alpha,-}(\zeta))^{*}\,d\zeta = 0.
\end{equation}
Adding \eqref{VanishingIdentity1} to its Hermitian conjugate and using \eqref{GalphaJump}
we obtain
\begin{align} \nonumber
    0 & = \int_{\mathbb{R}} G_{\alpha,-}(\zeta)
        \left[v_{G_{\alpha}}(\zeta) + (v_{G_{\alpha}}(\zeta))^*\right](G_{\alpha,-}(\zeta))^{*}\,d\zeta \\
    & = \int\limits_{-\infty}^0 G_{\alpha,-}(\zeta)
        \begin{pmatrix} 2 & 0 \\ 0 & 0 \end{pmatrix}
        (G_{\alpha,-}(\zeta))^{*}\,d\zeta +
        \int\limits_{0}^{+\infty} G_{\alpha,-}(\zeta)
        \begin{pmatrix} 2e^{-2\theta(\zeta)} & 0 \\ 0 & 0 \end{pmatrix}
        (G_{\alpha,-}(\zeta))^{*}\,d\zeta.
        \label{VanishingIdentity2}
\end{align}
Here we also used that $\theta_+(\zeta) = -\theta_-(\zeta) \in i\mathbb{R}$ for $\zeta < 0$,
which holds because $s$ is real. The identity \eqref{VanishingIdentity2} implies
that the first column of $G_{\alpha,-}$ vanishes identically on $\mathbb{R}$.
Thus, in view of the form of the jump matrix in \eqref{GalphaJump}, the second column of
$G_{\alpha,+}$ vanishes identically on $\mathbb{R}$ as well. It follows
that the first column of $G_{\alpha}$ vanishes identically in the lower half-plane, and
the second column vanishes identically in the upper half-plane.

To prove that the full matrix $G_{\alpha}$ vanishes identically in both
half-planes, we shall use a
Phragmen-Lindel\"{o}f type theorem due to Carlson \cite{Carlson,ReedSimon}.
Define for $j=1,2$,
\begin{equation}
g_j(\zeta) = \left\{
\begin{array}{ll}
  (G_{\alpha})_{j1}(\zeta), &\text{ for } \Im \zeta > 0,  \\[5pt]
  (G_{\alpha})_{j2}(\zeta), &\text{ for } \Im \zeta < 0.
\end{array}
\right.
\end{equation}
The conditions of the RH problem for $G_{\alpha}$ yield that
both $g_1$ and $g_2$ have analytic continuation across $(0, \infty)$
and that they are both solutions of the following scalar RH problem.
\paragraph{Riemann-Hilbert problem for $g$}
\begin{itemize}
\item $g : \mathbb{C} \setminus (-\infty, 0] \to \mathbb C$ is analytic with jump
\begin{equation} \label{gjump}
g_+(\zeta) = g_-(\zeta)\,e^{-2\alpha \pi i} e^{2\theta_-(\zeta)}, \qquad \text{ for } \zeta \in (-\infty, 0).
\end{equation}
\item $g(\zeta) = O(\zeta^{-3/4})$ as  $\zeta \to \infty$.
\item $g(\zeta) = O(|\zeta|^{-|\alpha|})$ as $\zeta \to 0$.
\end{itemize}
We are going to prove that this RH problem has only the trivial solution.

Let $g$ be any solution and define $\hat{g}$ by
\begin{equation} \label{hatgdefinition}
    \hat{g}(\zeta) = \left\{ \begin{array}{ll}
        g(\zeta^2), & \text{ if } \Re \zeta > 0, \\[5pt]
        g(\zeta^2) e^{-2\alpha \pi i} e^{-2 (\frac{2}{3} \zeta^3 + s \zeta)}, & \text{ if } \Re \zeta < 0, \, \Im \zeta > 0, \\[5pt]
        g(\zeta^2) e^{2\alpha \pi i} e^{-2(\frac{2}{3} \zeta^3 + s \zeta)}, & \text{ if } \Re \zeta < 0, \, \Im \zeta < 0.
        \end{array} \right.
\end{equation}
The jump property \eqref{gjump} ensures that $\hat{g}$ is analytic across the imaginary axis.

Now define
\begin{equation}
h(\zeta) = \left( \frac{\zeta}{1+\zeta} \right)^{\frac{8}{3} |\alpha|}
    \hat{g}(\zeta^{4/3}), \qquad \text{for } \Re \zeta \geq 0,
    \end{equation}
with (as usual) the principal branches of the fractional powers.
Then it can be checked that $h$ is analytic in $\Re \zeta > 0$,
bounded for $\Re \zeta \geq 0$, and satisfies
\[ |h(\zeta)| \leq C e^{-c |\zeta|^4}, \qquad \text{ if } \zeta \in i \mathbb R, \]
for some positive constants $c$ and $C$. Hence, by Carlson's theorem,
$h \equiv 0$ in $\Re \zeta \geq 0$. Therefore $g \equiv 0$, and so $g_1$ and $g_2$ are both identically zero.
It follows that the full matrix $G_{\alpha}$ vanishes identically in both half-planes.

Thus $F_{\alpha} \equiv 0$ by \eqref{DefinitionGalpha}, and this
completes the proof of the proposition.
\end{proof}

We now show how (unique) solvability of the RH problem for
$\Psi_{\alpha}^{(spec)}$ can be deduced from the above vanishing
lemma.
\begin{proposition}
\label{ExistenceProp}
The RH problem for $\Psi_{\alpha}^{(spec)}$ has a unique solution
for every $s \in \mathbb R$.
\end{proposition}

\begin{proof}
The idea of the proof is this: Given a solution
$\Psi_{\alpha}^{(spec)}$ to the above RH problem, we show how to
construct a solution $m_{\alpha}$ to a certain normalized RH problem
(i.e., $m_{\alpha}(\zeta) \to I$ as $\zeta \to \infty$)
characterized by a jump matrix $v$ on a contour $\widetilde{\Sigma}$,
and conversely. To prove the proposition
it therefore suffices to prove (unique) solvability of the
normalized RH problem. This can be done by
utilizing the basic relationship between normalized RH problems and
singular integral equations. We recall briefly, in our setting, some standard
facts regarding this relationship. For further details, and proofs, the reader is
referred to the papers \cite{DeiftZhou2,DeiftZhou3,DeiftZhou4,Zhou},
and to the appendix of \cite{KMM}.

Let $C$ denote the Cauchy operator
\begin{equation}
    C h(\zeta) = \frac{1}{2\pi i} \int_{\widetilde{\Sigma}} \frac{h(s)}{s-\zeta} ds,
    \qquad h \in L^2(\widetilde{\Sigma}),
    \quad \zeta \in \mathbb C \setminus \widetilde{\Sigma},
\end{equation}
and denote by $C_{\pm}h(\zeta)$, $\zeta \in \widetilde{\Sigma}$, the limits of $Ch(\zeta')$ as
$\zeta' \to \zeta$ from the $(\pm)$-side of $\widetilde{\Sigma}$.
The operators $C_{\pm}$ are bounded on $L^2(\widetilde{\Sigma})$.
Let
\begin{equation}
\label{factorization1}
v(\zeta) = (v_-(\zeta))^{-1} v_+(\zeta), \qquad \zeta \in \widetilde{\Sigma},
\end{equation}
be a pointwise factorization of $v(\zeta)$ with $v_{\pm}(\zeta) \in \text{GL}(2,\mathbb C)$,
and define $\omega_{\pm}$ through
\begin{equation}
\label{factorization2}
v_\pm(\zeta) = I \pm \omega_{\pm}(\zeta), \qquad \zeta \in \widetilde{\Sigma}.
\end{equation}
Our choice of factorization will imply that $\omega_{\pm} \in L^{2}(\widetilde{\Sigma}) \cap L^{\infty}(\widetilde{\Sigma})$.
The singular integral operator $C_{\omega} : L^2(\widetilde{\Sigma}) \to L^2(\widetilde{\Sigma})$,
defined by
\begin{equation}
\label{singint}
C_{\omega} h = C_+(h \omega_-) + C_-(h \omega_+), \qquad h \in L^2(\widetilde{\Sigma}),
\end{equation}
is then bounded on $L^2(\widetilde{\Sigma})$. Moreover, it makes sense to study the singular integral
equation
\begin{equation}
\label{singinteq}
(1 - C_{\omega}) \mu = I
\end{equation}
for $\mu \in I + L^2(\widetilde{\Sigma})$. For if we write $\mu = I + h$, then \eqref{singinteq} takes the
form
\begin{equation}
(1 - C_{\omega})h = C_{\omega}I \in L^2(\widetilde{\Sigma}).
\end{equation}
Suppose that $\mu \in I + L^2(\widetilde{\Sigma})$ is a solution of \eqref{singinteq}. Then, indeed
\begin{equation}
m_{\alpha}(\zeta) = I + C(\mu (\omega_+ +\omega_-))(\zeta), \qquad \zeta \in \mathbb C \setminus \widetilde{\Sigma},
\end{equation}
solves the normalized RH problem.
Thus, if we can prove that the operator $1 - C_{\omega}$ is a bijection in $L^2(\widetilde{\Sigma})$, then
solvability of the RH problem for $m_{\alpha}$, and hence of that for $\Psi_{\alpha}^{(spec)}$, has been
established. Bijectivity of $1 - C_{\omega}$ in $L^2(\widetilde{\Sigma})$ is proved in two steps.
We first show that, for an appropriate choice of $\omega = (\omega_-, \omega_+)$
in the above factorization,
$1 - C_{\omega}$ is Fredholm in $L^2(\widetilde{\Sigma})$ with index $0$.
Second, we show that the kernel of $1 - C_{\omega}$ is trivial. Now, it is a standard fact that
$\text{ker } (1 - C_{\omega}) = \{0\}$ if and only if the associated homogeneous RH problem (for say
$m_{\alpha}^{0}$) has only the trivial solution.
But the explicit relation between $\Psi_{\alpha}^{(spec)}$ and $m_{\alpha}$ also establishes a relation
between solutions $F_{\alpha}$ and $m_{\alpha}^{0}$ of the associated homogeneous RH problems.
In view of Proposition \ref{vanishinglemma}, which states that $F_{\alpha} \equiv 0$,
the second step has thus already been accomplished.

We now establish the above mentioned relation between
$\Psi_{\alpha}^{(spec)}$ and $m_{\alpha}$, derive the RH problem satisfied
by $m_{\alpha}$, and finally show that a factorization of $v$ may be chosen
so that $1 - C_{\omega}$ is Fredholm with index $0$, cf.\ \cite{FZ}.

Let $\mathbb{D} = \{\zeta \in \mathbb{C} \mid |\zeta| < 1\}$. Set
    $\theta(\zeta) = \frac{2}{3} \zeta^{3/2} + s \zeta^{1/2}$ and
\begin{equation}
\label{malphajump} m_{\alpha}(\zeta) = \left\{
\begin{array}{ll}
  \Psi_{\alpha}^{(spec)}(\zeta) A_j^{-1}
  \begin{pmatrix} \zeta^{-\alpha} & -\frac{\kappa_{\alpha}}{\pi} \zeta^{\alpha} \log \zeta \\ 0 & \zeta^{\alpha} \end{pmatrix},&
  \text{ for } \zeta \in \Omega_j \cap \mathbb{D}, \\[10pt]
  \Psi_{\alpha}^{(spec)}(\zeta) e^{\theta(\zeta) \sigma_3}
\frac{1}{\sqrt{2}} \begin{pmatrix} 1 & -i \\ -i & 1 \end{pmatrix}
\zeta^{\sigma_3/4}, & \text{ for } \zeta \in \Omega_j \cap
\overline{\mathbb{D}}^c,
\end{array}
\right.
\end{equation}
with $\{A_j\}_{j=1}^4$ being the matrices in Proposition \ref{Psiat0Prop},
and where $\kappa_{\alpha} = 1$ if $\alpha-1/2 \in \mathbb{N}_{0}$ and $0$ otherwise.
By Proposition \ref{Psiat0Prop} it follows that $m_{\alpha}$ is analytic in $\mathbb{D}$.
Let $\widetilde{\Sigma} = \Sigma \cup \partial \mathbb{D}$ and orient the components of
$\widetilde{\Sigma}$ as in Figure \ref{figure2}. This makes
$\widetilde{\Sigma}$ a complete contour, meaning that $\mathbb{C}
\setminus \widetilde{\Sigma}$ can be expressed as the union of two
disjoint sets, $\mathbb{C} \setminus \widetilde{\Sigma} = \Omega_+
\cup \Omega_-$, $\Omega_+ \cap \Omega_- = \emptyset $, such that
$\widetilde{\Sigma}$ is the positively oriented boundary of
$\Omega_+$ and the negatively oriented boundary of $\Omega_-$. Let
$\widetilde{\Sigma}_j = \Omega_j \cap \partial \mathbb{D}$.

\begin{figure}[t]
\input{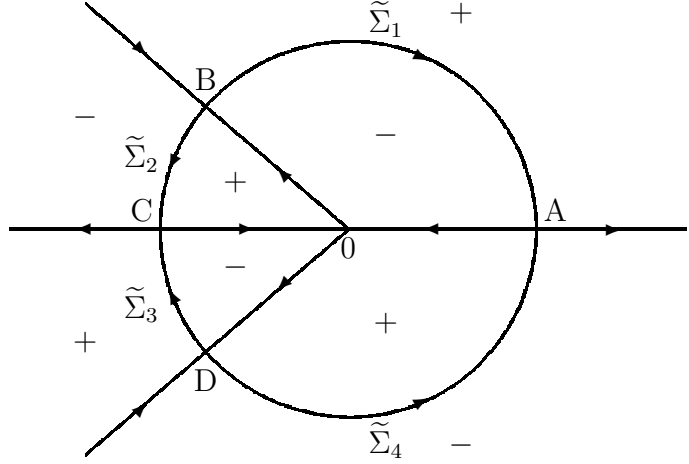}
\caption{\label{figure2} Contour for the RH problem for $m_{\alpha}$.}
\end{figure}

Computations show that $m_{\alpha}$ satisfies the
following normalized RH problem. As in Proposition \ref{Psiat0Prop}
we use $v_j$ to denote the jump matrix on $\Sigma_j$ in the model RH problem.

\paragraph{Riemann-Hilbert problem for $m_{\alpha}$}
\begin{itemize}
\item $m_{\alpha} : \mathbb{C} \setminus \widetilde{\Sigma} \to
    \mathbb C^{2\times 2}$ is analytic.
\item $m_{\alpha,+}(\zeta) = m_{\alpha,-}(\zeta) v(\zeta)$ for $\zeta \in \widetilde{\Sigma}$,
where\\
\[v(\zeta) = \left\{
    \begin{array}{ll}
    I, &\text{ for } \zeta \in \widetilde{\Sigma} \cap \mathbb{D},\\[10pt]
    \zeta^{-\sigma_3/4} \frac{1}{\sqrt{2}}
    \begin{pmatrix} 1 & i \\ i & 1 \end{pmatrix} e^{-\theta \sigma_3}
    v_j e^{\theta \sigma_3} \frac{1}{\sqrt{2}}
    \begin{pmatrix} 1 & -i \\ -i & 1 \end{pmatrix} \zeta^{\sigma_3/4},
    &\text{ for } \zeta \in \Sigma_j \cap \overline{\mathbb{D}}^c, j \in \{1,2,4\},\\[10pt]
    I, &\text{ for } \zeta \in \Sigma_3 \cap \overline{\mathbb{D}}^c,\\[10pt]
    \begin{pmatrix} \zeta^{\alpha} & \frac{\kappa_{\alpha}}{\pi} \zeta^{\alpha} \log \zeta \\
        0 & \zeta^{-\alpha} \end{pmatrix}
        A_j\,e^{\theta \sigma_3} \frac{1}{\sqrt{2}}
    \begin{pmatrix} 1 & -i \\ -i & 1 \end{pmatrix} \zeta^{\sigma_3/4},
    &\text{ for } \zeta \in \widetilde{\Sigma}_j, j \in \{1,3\},\\[10pt]
    \zeta^{-\sigma_3/4} \frac{1}{\sqrt{2}}
    \begin{pmatrix} 1 & i \\ i & 1 \end{pmatrix} e^{-\theta \sigma_3} A_j^{-1}
    \begin{pmatrix} \zeta^{-\alpha} & -\frac{\kappa_{\alpha}}{\pi} \zeta^{\alpha} \log \zeta \\
        0 & \zeta^{\alpha} \end{pmatrix},&
        \text{ for } \zeta \in \widetilde{\Sigma}_j, j \in \{2,4\}.
\end{array}
\right.\]
\item $m_{\alpha}(\zeta) = I + O(1/\zeta) \text{ as } \zeta \to \infty$.
\end{itemize}
The analyticity of $m_{\alpha}$ on $\Sigma_3 \cap
\overline{\mathbb{D}}^c$ follows since $\theta_+(\zeta)
+ \theta_-(\zeta) = 0$ for $\zeta < 0$.

It is important to note that $v(\zeta) - I$ decays
exponentially as $\zeta \rightarrow \infty$ along $\widetilde{\Sigma}$.
Next observe that, at any of the points $0,A,B,C,D$ of self-intersection
of $\widetilde{\Sigma}$ (see Figure \ref{figure2}), precisely four
contours come together. At a fixed point of self-intersection, say $P$,
order the contours that meet at $P$ counterclockwise, starting from
any contour that is oriented outwards from $P$.
Denoting the limiting value of the jump matrices over the $j$th contour
at $P$ by $v^{(j)}(P)$, we then have the cyclic relation
\begin{equation} \label{cyclic}
v^{(1)}(P) \left(v^{(2)}(P)\right)^{-1} v^{(3)}(P) \left(v^{(4)}(P)\right)^{-1} = I.
\end{equation}
This is trivial in case $P = 0$, and follows by direct computation
in the other cases.
We remark that the cyclic relation \eqref{cyclic} at $C$ is a consequence
of the relation \eqref{cyclicidentity1} in the case $\alpha - 1/2 \not\in \mathbb N_0$,
and of \eqref{cyclicidentity2} in the case $\alpha - 1/2 \in \mathbb N_0$
(see the proof of Proposition \ref{Psiat0Prop}).

Outside small neighborhoods of the points of self-intersection we choose the trivial
factorization $v_+ = v$, $v_- = I$ in \eqref{factorization1},
so that $\omega_+ = v - I$, $\omega_- = 0$ by \eqref{factorization2}.
Using the cyclic relations \eqref{cyclic}, we are then able to choose a factorization of $v$
in the remaining neighborhoods in such a way that $\omega_+$
is continuous along the boundary of each connected component of $\Omega_+$,
and similarly, $\omega_-$ is continuous along the boundary of each connected
component of $\Omega_-$.

The exponential decay of $v(\zeta) - I$ as $\zeta \rightarrow \infty$ ensures that
$\omega_{\pm} \in L^{2}(\widetilde{\Sigma}) \cap L^{\infty}(\widetilde{\Sigma})$.
From this it follows that $1 - C_{\omega}$ is Fredholm in $L^2(\widetilde{\Sigma})$.
Indeed, set
\begin{equation}
\widetilde{\omega}_- = I - v_-^{-1}, \quad
\widetilde{\omega}_+ = v_+^{-1} - I.
\end{equation}
The choice of $\widetilde{\omega} =
(\widetilde{\omega}_-, \widetilde{\omega}_+)$ is motivated by the relations
\begin{equation}
\label{omega-omegatilderelations}
\widetilde{\omega}_-\,\omega_- = \widetilde{\omega}_- + \omega_-, \quad
\widetilde{\omega}_+\,\omega_+ = -(\widetilde{\omega}_+ + \omega_+).
\end{equation}
A direct calculation, using $C_+ - C_- = 1$ and
\eqref{omega-omegatilderelations}, shows that
\begin{equation}
(1 - C_{\omega})(1 - C_{\widetilde{\omega}}) = 1 + T,
\end{equation}
where
\begin{equation}
Tf =
C_+((C_-[f(\widetilde{\omega}_+ + \widetilde{\omega}_-)])\,\omega_-)
+
C_-((C_+[f(\widetilde{\omega}_+ + \widetilde{\omega}_-)])\,\omega_+)
\end{equation}
for $f \in L^2(\widetilde{\Sigma})$. Standard computations, using continuity of the
functions $\omega_+$ resp. $\omega_-$ along the boundary of each connected component of $\Omega_+$ resp.
$\Omega_-$, show that $T$ is compact in $L^2(\widetilde{\Sigma})$. Similar computations show that
$(1 - C_{\widetilde{\omega}})(1 - C_{\omega}) = 1 + S$, with $S$ compact in $L^2(\widetilde{\Sigma})$.
So $1 - C_{\widetilde{\omega}}$ is a pseudoinverse for $1 - C_{\omega}$, which is
therefore Fredholm in $L^2(\widetilde{\Sigma})$.

It follows from general theory that the index of the operator $1 - C_{\omega}$ equals the
winding number of $\det v$ along $\widetilde{\Sigma}$, the latter being defined in the natural way.
Now, since $\det v \equiv 1$, this is trivially zero. This completes the proof
of Proposition \ref{ExistenceProp}.
\end{proof}

\begin{remark} The RH problem for $\Psi_{\alpha}^{(spec)}$ is indeed solvable for all
$s \in \mathbb{C} \setminus \mathcal{D}$, where $\mathcal{D}$ is a discrete set
in $\mathbb{C}$ (disjoint from $\mathbb{R}$ according to Proposition \ref{ExistenceProp}),
and the solution $\Psi_{\alpha}^{(spec)}$ is meromorphic in $s$ with poles in $\mathcal D$.
To see this, we first observe that the factorization \eqref{factorization1}, \eqref{factorization2}
can be done so that $\omega_{\pm}$ are both analytic in $s$. It follows that
$s \mapsto 1 - C_{\omega}$ is an analytic map taking values in the Fredholm operators
on $L^2(\widetilde{\Sigma})$.
Since we know that $1 - C_{\omega}$ is invertible for $s \in \mathbb R$, we then get,
by a version of the analytic Fredholm theorem \cite{Zhou},
that $\mu$ defined by \eqref{singinteq} is meromorphic.
Thus $m_{\alpha}$ and hence $\Psi_{\alpha}^{(spec)}$ is meromorphic in $s$.
\end{remark}

\subsection{Some preliminaries on equilibrium measures}
Before we embark on the steepest descent analysis for the RH problem
of Subsection \ref{FIKRHproblem}, we recall certain  properties of
equilibrium measures, see \cite{Deift,SaTo}.
We use the following notation:
\begin{equation}
    t = \frac{n}{N}, \quad V_t(x) = \frac{1}{t}\,V(x).
\end{equation}
As explained in the Introduction, we are interested in the case
where $n/N \to 1$ as $n,N \to \infty$, which means that
we are interested in $t$ close to $1$. For every $t$ we consider
the energy functional $I_{V_t}(\mu)$ as in \eqref{energyfunctional},
and its minimizer $\mu_t$.

The equilibrium measure $d\mu_t = \rho_t\,dx$ is characterized by
the following Euler-Lagrange variational conditions: There is a
constant $l_t \in \mathbb{R}$ such that
\begin{alignat}{2}
\label{variational1}
2 \int \log|x - s|\rho_t(s)\,ds - V_t(x) + l_t &= 0, \quad x \in \supp \mu_t,\\
\label{variational2}
2 \int \log|x - s|\rho_t(s)\,ds - V_t(x) + l_t & \leq 0, \quad x \in \mathbb{R} \setminus \supp \mu_t.
\end{alignat}

For $t=1$, we have that the support of $\mu_V$ consists of a finite union of
disjoint intervals, see \cite{DKM}, say
\[ \supp \mu_V = \bigcup_{j=1}^k [a_j,b_j] \]
with $a_1 < b_1 < a_2 < \cdots < a_k < b_k$.
Due to the assumption that the density $\rho_V$ of $\mu_V$ is regular,
we have the following proposition.

\begin{proposition}
\label{supp}
For every $t$ in an interval around $1$, we have that the density
$\rho_t$ of $\mu_t$ is regular, and that $\supp \mu_t$ consists of $k$
intervals, say
\[ \supp \mu_t = \bigcup_{j=1}^k [a_j(t),b_j(t)] \]
with $a_1(t) < b_1(t) < a_2(t) < \cdots < a_k(t) < b_k(t)$.
In this interval around $1$, the functions $t \mapsto a_j(t)$
and  $t \mapsto b_j(t)$ are real analytic with $a_j'(t) < 0$ and $b_j'(t) > 0$.
\end{proposition}
\begin{proof}
See Theorem 1.3 (iii) and Lemma 8.1 of \cite{KM}.
\end{proof}

For the rest of the proof of Theorem \ref{theorem1} we shall assume
that $\supp \mu_V$ consists of one interval.
In the general case (when $\supp \mu_V$ consists of $k \geq 2$ intervals)
one proceeds analogously, but the parametrix away from the
end points given in Subsection \ref{ParametrixAtInfinity} must then instead
be constructed with the help of the $\theta$-function of $B$-periods for
the two-sheeted Riemann surface $y^2 = \Pi_{j=1}^k [(z-a_j)(z-b_j)]$
obtained by gluing together two copies of the slit plane
$\mathbb C \setminus \bigcup_{j=1}^k [a_j,b_j]$ in the standard way \cite{DKMVZ2,KV}.
Since the formulas will be more complicated in the multi-interval case, but
do not contribute to the main issue of the present paper,
we chose to give the proof in full for the one-interval case only.

\subsection{Steepest descent analysis}

\subsubsection{Preliminaries}
We assume from now on that $k=1$, so that $\supp \mu_V$ consists
of one interval  which we take as
\[ \supp(\mu_V) = [a,0], \qquad a < 0. \]
Then there is $\delta_1 > 0$ such that $\mu_t$ is supported
on one interval $[a_t,b_t]$ for every $t \in (1-\delta_1,1+\delta_1)$,
and its density $\rho_t$ is regular.
Hence $\rho_t$ is positive on $(a_t,b_t)$ and vanishes
like a square root at the end points, and it takes the form \cite{DKMVZ2}
\begin{equation} \label{htdef}
    \rho_t(x) = \frac{1}{2\pi} \sqrt{(b_t-x)(x-a_t)} \, h_t(x), \qquad
    \text{for } x\in [a_t,b_t],
\end{equation}
where $h_t$ is positive on $[a_t,b_t]$, and analytic in the domain
of analyticity of $V$. In addition, $h_t$ depends analytically on
$t \in (1-\delta_1,1+\delta_1)$.

We are going to use the equilibrium measure $\mu_t$ in the
first transformation of the RH problem.
We remark that in \cite{CK,CKV,CV,DK} a modified equilibrium measure
was used in the steepest descent analysis of a RH problem at a critical
point.
It is likely that we could have modified the equilibrium measure
in the present situation as well, but the approach with the
unmodified $\mu_t$ also works, as we will see, and we chose to use it
in this paper.

In the one-interval case one can show by explicit computation that
\begin{equation} \label{edgespeed1}
    \frac{d}{dt} a_t = -\frac{4}{t(b_t - a_t)h_t(a_t)},
    \qquad
    \frac{d}{dt} b_t = \frac{4}{t(b_t - a_t)h_t(b_t)},
\end{equation}
which indeed shows that $\frac{d}{dt} a_t < 0$ and
$\frac{d}{dt} b_t > 0$.
It follows that $b_t > 0$ for $t \in (1,1+\delta_1)$ and
$b_t < 0$ for $t \in (1-\delta_1,1)$.
In both cases we have $a_t < 0$.

We introduce two functions $\varphi_t$ and $\tilde{\varphi}_t$ as
follows. For $z \in \mathbb C \setminus (-\infty, b_t]$ lying in the domain
of analyticity of $V$ (which we may restrict to be simply connected,
without loss of generality), we put
\begin{equation} \label{varphi}
    \varphi_t(z) = \frac{1}{2}\,\int_{b_t}^{z}
    ((s-b_t)(s-a_t))^{1/2} h_t(s)\,ds,
\end{equation}
and for $z \in \mathbb C \setminus [a_t,\infty)$ also in the domain of
analyticity of $V$,
\begin{equation} \label{varphitilde}
    \tilde{\varphi}_t(z) = \frac{1}{2}\,\int_{a_t}^{z}
    ((s-b_t)(s-a_t))^{1/2}h_t(s)\,ds.
\end{equation}
It follows from \eqref{varphi} that
\begin{equation}
\varphi_t(z) = \frac{1}{3} \sqrt{-a_t}\,h_t(b_t) (z-b_t)^{3/2} \chi_t(z),
\end{equation}
where $\chi_t$ is analytic in a neighborhood of $b_t$ and $\chi_t(b_t) = 1$.
Taking
\begin{equation} \label{ftdef}
    f_t(z) = \left( \frac{3}{2} \varphi_t(z) \right)^{2/3} =
    \left(\frac{1}{2} \sqrt{-a_t}\,h_t(b_t)\right)^{2/3} (z - b_t)\,\chi_t^{2/3}(z),
\end{equation}
we see that $f_t$ is analytic in a neighborhood of $b_t$ with $f_t(b_t) = 0$,
\begin{equation} \label{ftprime}
f_t'(b_t) =
\left(\frac{1}{2}\,\sqrt{-a_t}\,h_t(b_t)\right)^{2/3} \neq 0,
\end{equation}
and $f_t(z)$ real for real values of $z$. Hence, in particular,
\begin{equation} \label{ftprop}
    f_t(0) > 0,  \text{ if } t < 1, \qquad
    f_1(0) = 0, \quad \text{ and } \quad
    f_t(0) < 0, \text{ if } t > 1.
\end{equation}
Moreover, $f_t \to f_1$ as $t \to 1$,
uniformly in a neighborhood of $0$.
We choose a small disc $U^{(0)}$ around $0$ and $\delta_2 > 0$
sufficiently small, so that $f_t$ is a conformal map from $U^{(0)}$ onto
a convex neighborhood of $0$ for every $t \in (1-\delta_2,1+\delta_2)$.

Similarly, there exists a disc $U^{(a)}$ centered at $a < 0$,
and a $\delta_3 > 0$, so that
\begin{equation} \label{tildeftdef}
    \tilde{f}_t(z) = \left( \frac{3}{2} \tilde{\varphi}_t(z)\right)^{2/3}
    \end{equation}
is a conformal map from $U^{(a)}$ onto a convex neighborhood of $0$
for every $t \in (1-\delta_3,1+\delta_3)$.

We let $\delta_0 = \min(\delta_1,\delta_2,\delta_3)$ and
we fix $t \in (1-\delta_0,1+\delta_0)$. In what follows
we also take the neighborhoods
$U^{(0)}$ and $U^{(a)}$ as above.

\subsubsection{First transformation $Y \mapsto T$}
We introduce the so-called $g$-function:
\begin{equation} \label{gtdef}
g_t(z) = \int \log(z - s)\,d\mu_t(s) = \int \log(z -
s)\,\rho_t(s)\,ds, \quad z \in \mathbb{C} \setminus (-\infty, b_t],
\end{equation}
where $\log$ denotes the principal branch. Then $g_t$ is analytic
in $\mathbb{C} \setminus (-\infty, b_t]$.
Define $T$ by
\begin{equation} \label{Tdef}
    T(z) = e^{\frac{1}{2} n l_t \sigma_3}\,Y(z)\,e^{-\frac{1}{2} n l_t \sigma_3}\,e^{-n g_t(z)
     \sigma_3}, \quad z \in \mathbb{C} \setminus \mathbb{R},
\end{equation}
where $l_t$ is the constant from
\eqref{variational1}--\eqref{variational2}. By a straightforward
calculation it then follows that $T$ has the following jump matrix
$v_T$ on $\mathbb{R}$ (oriented from left to right):
\begin{equation}
v_T(x) = \begin{pmatrix}
        e^{-n(g_{t,+}(x) - g_{t,-}(x))} & |x|^{2\alpha}\,e^{n(g_{t,+}(x) + g_{t,-}(x) - V_t(x) + l_t)}\\
        0 & e^{n(g_{t,+}(x) - g_{t,-}(x))}
      \end{pmatrix}.
\end{equation}
Because of the identities, see \cite{Deift,DKMVZ2},
\begin{alignat}{2} \label{g-identity1}
    g_{t,+}(x) + g_{t,-}(x) - V_t(x) + l_t & = -2\varphi_t(x), \quad \text{ for } x > b_t, \\
    \label{g-identity2}
    g_{t,+}(x) + g_{t,-}(x) - V_t(x) + l_t & = -2\tilde{\varphi}_t(x), \quad \text{ for } x < a_t,
\end{alignat}
we see that the RH problem for $T$ is the following.

\paragraph{Riemann-Hilbert problem for $T$}
\begin{itemize}
\item $T : \mathbb C \setminus \mathbb R \to
    \mathbb C^{2\times 2}$ is analytic.
\item $T_+(x) = T_-(x)\,v_T(x)$ for $x \in \mathbb R$, with
\[ v_T(x) =\left\{ \begin{array}{ll}
             \begin{pmatrix}
                 1 & |x|^{2\alpha}\,e^{-2 n \tilde{\varphi}_t(x)} \\
                 0 & 1
             \end{pmatrix},&\text{ for } x < a_t,
             \\[10pt]
             \begin{pmatrix}
                 e^{2 n\varphi_{t,+}(x)} & |x|^{2\alpha} \\
                 0 & e^{2 n \varphi_{t,-}(x)}
             \end{pmatrix},&\text{ for } x \in (a_t, b_t),
             \\[10pt]
             \begin{pmatrix}
                 1 & |x|^{2\alpha}\,e^{-2 n \varphi_t(x)}\\
                 0 & 1
             \end{pmatrix},&\text{ for } x > b_t.
\end{array}
\right. \]
\item $T(z) = I + O(1/z)$ as $z \to \infty$.
\item If $\alpha < 0$, then $T(z) = O\left(\begin{smallmatrix} 1 & |z|^{2\alpha} \\ 1 & |z|^{2\alpha}
\end{smallmatrix}\right)$ as $z \to 0$. If $\alpha \geq 0$,
then $T(z) = O\left(\begin{smallmatrix} 1 & 1 \\ 1 & 1
\end{smallmatrix}\right)$ as $z \to 0$.
\end{itemize}

\subsubsection{Second transformation $T \mapsto S$}
The opening of lenses is based on the following factorization of
$v_T$ on $(a_t, b_t)$:
\begin{alignat*}{2}
v_T(x) &= \begin{pmatrix}
                 e^{2 n \varphi_{t,+}(x)} & |x|^{2\alpha} \\
                 0 & e^{2 n \varphi_{t,-}(x)}
               \end{pmatrix}\\[10pt]
           &=\begin{pmatrix}
                 1 & 0 \\
                 |x|^{-2\alpha}\,e^{2 n \varphi_{t,-}(x)} & 1
               \end{pmatrix}
             \begin{pmatrix}
                 0 & |x|^{2\alpha} \\
                 -|x|^{-2\alpha} & 0
               \end{pmatrix}
             \begin{pmatrix}
                 1 & 0 \\
                 |x|^{-2\alpha}\,e^{2 n \varphi_{t,+}(x)} & 1
             \end{pmatrix}.
\end{alignat*}

Introduce a lens around the segment $[a_t, 0]$ as in Figure \ref{figure3}
(recall that $a_t < 0$).
In the disc $U^{(0)}$ around $0$ we take the lens such that
$z \mapsto \zeta = f_t(z) - f_t(0)$, see \eqref{ftdef}, maps the parts of the
upper and lower lips of
the lens that are in $U^{(0)}$ into the rays $\arg \zeta = 2\pi/3$ and $\arg \zeta = -2\pi/3$,
respectively.
Similarly, in the disc $U^{(a)}$ we choose the lens so that
$z \mapsto \zeta = \tilde{f}_t(z)$, see \eqref{tildeftdef},
maps the parts of the upper and lower lips of the lens that are in $U^{(a)}$ into the rays
$\arg \zeta = \pi/3$, and $\arg \zeta = -\pi/3$, respectively.
The remaining parts of the lips of the lens are arbitrary. However,
they should be contained in the domain of analyticity of $V$, and we
take them so that
\[ \Re \varphi_t(z) < -c < 0 \]
for $z$ on the lips of the lens outside $U^{(0)}$ and $U^{(a)}$,
with $c > 0$ independent of $t$.

It is important to note that the lens is around $[a_t,0]$, and not
around $[a_t,b_t]$.

\begin{figure}[t]
\ifx\JPicScale\undefined\def\JPicScale{1}\fi
\unitlength \JPicScale mm
\begin{picture}(80,20)(-30,20)
\linethickness{0.3mm} \put(10,30){\line(1,0){70}}

\linethickness{0.3mm}
\qbezier(25.00,30.00)(45.00,50.00)(65.00,30.00)

\linethickness{0.3mm}
\qbezier(25.00,30.00)(45.00,10.00)(65.00,30.00)

\thicklines \put(70,30){\line(1,0){2.5}}
\put(72.5,30){\vector(1,0){0.12}} \thicklines
\put(15,30){\line(1,0){2.5}} \put(17.5,30){\vector(1,0){0.12}}
\thicklines \put(43.75,30){\line(1,0){2.5}}
\put(46.25,30){\vector(1,0){0.12}} \thicklines
\put(45,40){\line(1,0){1.25}} \put(46.25,40){\vector(1,0){0.12}}
\thicklines \put(45,20){\line(1,0){1.25}}
\put(46.25,20){\vector(1,0){0.12}}
\put(23.5,27.75){\makebox(0,0)[cc]{$a_t$}}

\put(66.25,27.75){\makebox(0,0)[cc]{$0$}}

\end{picture}
\caption{\label{figure3} Opening of a lens around $[a_t,0]$.}
\end{figure}
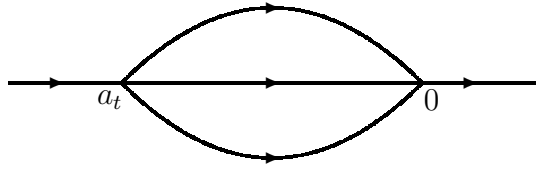

Define $S$ by
\begin{equation} \label{Sdef}
S(z) = \left\{
\begin{array}{ll}
  T(z),&\text{ for } z \text{ outside the lens},
        \\
  T(z) \begin{pmatrix}
           1 & 0 \\
           (-z)^{-2\alpha}\,e^{2 n \varphi_t(z)} & 1
       \end{pmatrix}^{-1},&\text{ for } z \text{ in the upper part of the lens},
        \\[10pt]
  T(z) \begin{pmatrix}
           1 & 0 \\
           (-z)^{-2\alpha}\,e^{2 n \varphi_t(z)} & 1
       \end{pmatrix},&\text{ for } z \text{ in the lower part of the lens}.
\end{array}
        \right.
\end{equation}
Here the map $z \mapsto (-z)^{-2\alpha}$ is defined with a cut
along the positive real axis.
Then, from \eqref{Sdef} and the RH problem for $T$, we find that
$S$ is the unique solution of the following RH problem.
\paragraph{Riemann-Hilbert problem for $S$}
\begin{itemize}
\item $S : \mathbb C \setminus \Sigma_S \to
    \mathbb C^{2\times 2}$ is analytic, where $\Sigma_S$ consists of
    the real line and the upper and lower lips of the lens, with orientation as
    in Figure \ref{figure3}.
\item $S_+(z) = S_-(z)\,v_S(z)$ for $z \in \Sigma_S$, where $v_S$ is given as follows.
    For $t < 1$, so that $b_t < 0$, we have
    \[ v_S(z) =\left\{ \begin{array}{ll}
             \begin{pmatrix}
                 1 & |z|^{2\alpha}\,e^{-2 n \tilde{\varphi}_t(z)} \\
                 0 & 1
             \end{pmatrix},&\text{ for } z \in (-\infty, a_t),
              \\[10pt]
             \begin{pmatrix}
                 0 & |z|^{2\alpha} \\
                 -|z|^{-2\alpha} & 0
             \end{pmatrix},&\text{ for } z \in (a_t, b_t),
              \\[10pt]
             \begin{pmatrix}
                 0 & |z|^{2\alpha}\,e^{-2 n \varphi_t(z)} \\
                 -|z|^{-2\alpha}\,e^{2 n \varphi_t(z)} & 0
             \end{pmatrix},&\text{ for } z \in (b_t, 0),
              \\[10pt]
             \begin{pmatrix}
                 1 & |z|^{2\alpha}\,e^{-2 n \varphi_t(z)}\\
                 0 & 1
             \end{pmatrix},&\text{ for } z \in (0,\infty),
             \\[10pt]
             \begin{pmatrix}
                 1 & 0 \\
                 (-z)^{-2\alpha}\,e^{2 n \varphi_t(z)}& 1
             \end{pmatrix},&\text{ for } z \text{ on both lips of the lens},
\end{array}
\right. \]
while, for $t \geq 1$, so that $b_t \geq 0$, we have
\[ v_S(z) =\left\{ \begin{array}{ll}
             \begin{pmatrix}
                 1 & |z|^{2\alpha}\,e^{-2 n \tilde{\varphi}_t(z)} \\
                 0 & 1
            \end{pmatrix},&\text{ for } z \in (-\infty, a_t),
              \\[10pt]
             \begin{pmatrix}
                 0 & |z|^{2\alpha} \\
                 -|z|^{-2\alpha} & 0
              \end{pmatrix},&\text{ for } z \in (a_t,0),
              \\[10pt]
             \begin{pmatrix}
                 e^{2 n\varphi_{t,+}(z)} & |z|^{2\alpha} \\
                 0 & e^{2 n \varphi_{t,-}(z)}
             \end{pmatrix},&\text{ for } z \in (0, b_t),
              \\[10pt]
             \begin{pmatrix}
                 1 & |z|^{2\alpha}\,e^{-2 n \varphi_t(z)}\\
                 0 & 1
             \end{pmatrix},&\text{ for } z \in (b_t,\infty),
              \\[10pt]
             \begin{pmatrix}
                 1 & 0 \\
                 (-z)^{-2\alpha}\,e^{2 n \varphi_t(z)}& 1
             \end{pmatrix},&\text{ for } z \text{ on both lips of the lens}.
\end{array}
\right. \]
\item $S(z) = I + O(1/z)$ as $z \to \infty$.
\item If $\alpha < 0$, then $S(z) = O\left(\begin{smallmatrix} 1 & |z|^{2\alpha} \\ 1 & |z|^{2\alpha}
\end{smallmatrix}\right)$ as $z \to 0$.
If $\alpha \geq 0$, then $S(z) = O\left(\begin{smallmatrix} 1 & 1 \\ 1 & 1
\end{smallmatrix}\right)$ as $z \to 0$ from outside the lens and $S(z) = O
\left(\begin{smallmatrix}
    |z|^{-2\alpha} & 1 \\
    |z|^{-2\alpha} & 1
\end{smallmatrix}\right)$ as $z \to 0$ from inside the lens.
\end{itemize}

The next step is to approximate $S$ by a parametrix $P$, consisting of three
parts $P^{(\infty)}$, $P^{(a)}$, and $P^{(0)}$:
\begin{equation} \label{Pdef}
P(z) = \left\{
\begin{array}{ll}
  P^{(0)}(z),&\text{ for } z \in U^{(0)} \setminus \Sigma_S,\\
  P^{(a)}(z),&\text{ for } z \in U^{(a)} \setminus \Sigma_S, \\
  P^{(\infty)}(z),&\text{ for } z \in \mathbb{C} \setminus
     (\overline{U^{(0)} \cup U^{(a)} \cup (a_t,0)}),
\end{array}
\right.
\end{equation}
where $U^{(a)}$ and $U^{(0)}$ are small discs centered at $a$ and $0$, respectively,
that have been introduced before.
The parametrices $P^{(\infty)}$, $P^{(a)}$ and $P^{(0)}$ are constructed below.

\subsubsection{The parametrix $P^{(\infty)}$}
\label{ParametrixAtInfinity}

The parametrix $P^{(\infty)}$ is a solution of the following RH problem.

\paragraph{Riemann-Hilbert problem for $P^{(\infty)}$}
\begin{itemize}
\item $P^{(\infty)} : \mathbb C \setminus [a_t,0] \to
    \mathbb C^{2\times 2}$ is analytic.
\item $P^{(\infty)}_+(x) = P^{(\infty)}_-(x)\, \left(
    \begin{array}{cc}
      0 & |x|^{2\alpha} \\
      -|x|^{-2\alpha} & 0 \\
    \end{array}
  \right)$  for $x \in (a_t, 0)$, oriented from left to right.
\item $P^{(\infty)}(z) = I + O(1/z)$ as $z \to \infty$.
\end{itemize}

The RH problem for $P^{(\infty)}$ can be explicitly solved as in \cite{CKV}.
Take
\begin{equation}
D(z) = z^{\alpha}\,\phi \left(\frac{2z -
a_t}{-a_t}\right)^{-\alpha}, \quad \text{for } z \in \mathbb{C}
\setminus [a_t, 0],
\end{equation}
where $\phi(z) = z + (z-1)^{1/2}\,(z+1)^{1/2}$ is the conformal map
from $\mathbb{C} \setminus [-1,1]$ onto the exterior of the unit
circle. Then $D_+(x)\, D_-(x) = |x|^{2\alpha}$ for $x \in (a_t, 0)$.
It follows that $D(\infty)^{-\sigma_3}\,P^{(\infty)}(z)\,D(z)^{\sigma_3}$
satisfies the normalized RH problem with jump matrix
$\left(\begin{smallmatrix} 0 & 1 \\ -1 & 0 \end{smallmatrix}\right)$ on $(a_t, 0)$
(oriented from left to right), whose solution is well-known, see e.g.\ \cite{Deift,DKMVZ2},
and it leads to
\begin{equation}
    P^{(\infty)}(z) = D(\infty)^{\sigma_3}
    \left(
      \begin{array}{cc}
         \frac{1}{2} \left(\beta_t(z) + \beta_t(z)^{-1}\right) & \frac{1}{2i} \left(\beta_t(z) - \beta_t(z)^{-1}\right) \\[5pt]
         -\frac{1}{2i} \left(\beta_t(z) - \beta_t(z)^{-1}\right) &  \frac{1}{2}\left(\beta_t(z) + \beta_t(z)^{-1}\right)
      \end{array}
    \right)
    D(z)^{-\sigma_3},
\end{equation}
for $z \in \mathbb C \setminus [a_t,0]$, where
\begin{equation} \beta_t(z) = \left(\frac{z}{z - a_t}\right)^{1/4},
\quad \text{for } z \in \mathbb{C} \setminus [a_t,0].
\end{equation}

\subsubsection{The parametrix $P^{(a)}$}

The parametrix $P^{(a)}$ is defined in the disc $U^{(a)}$ around $a$,
where $P^{(a)}$ satisfies the following RH problem.

\paragraph{Riemann-Hilbert problem for $P^{(a)}$}
\begin{itemize}
\item $P^{(a)} : U^{(a)} \setminus \Sigma_S \to
    \mathbb C^{2\times 2}$ is analytic.
\item $P^{(a)}_+(z) = P^{(a)}_-(z)\,v_S(z)$ for $z \in U^{(a)} \cap \Sigma_S$.
\item $P^{(a)}(z)\,\left(P^{(\infty)}(z)\right)^{-1} = I +
O(n^{-1})$, as $n \to \infty$, uniformly for $z \in \partial U^{(a)} \setminus \Sigma_S$.
\end{itemize}
We seek $P^{(a)}$ in the form
\begin{equation*} \label{Pahat}
    P^{(a)}(z) = \widehat{P}^{(a)}(z) \,e^{n \tilde{\varphi}_t(z) \sigma_3}
    \,(-z)^{-\alpha \sigma_3},
    \quad \text{ for } z \in U^{(a)} \setminus \Sigma_S,
\end{equation*}
where $(-z)^{-\alpha}$ is defined with a branch cut along $[0,\infty)$.
Then $\widehat{P}^{(a)}$  satisfies a RH problem with constant jumps
and can be constructed in terms of the Airy function in a standard way;
for more details see the presentation in \cite{Deift}.

\subsubsection{The parametrix $P^{(0)}$}

The parametrix $P^{(0)}$, defined in the disk $U^{(0)}$ around $0$,
should satisfy the following RH problem.

\paragraph{Riemann-Hilbert problem for $P^{(0)}$}
\begin{itemize}
\item $P^{(0)} : \overline{U^{(0)}} \setminus \Sigma_S \to
    \mathbb C^{2\times 2}$ is continuous and analytic on $U^{(0)} \setminus \Sigma_S$.
\item $P^{(0)}_+(z) = P^{(0)}_-(z)\,v_S(z)$ for $z \in \Sigma_S \cap U^{(0)}$
(with the same orientation as $\Sigma_S$).
\item $P^{(0)}(z)\,\left(P^{(\infty)}(z)\right)^{-1} = I +
O(n^{-1/3})$, as $n \to \infty$, $t \to 1$ such that $n^{2/3}(t - 1) = O(1)$,
uniformly for $z \in \partial U^{(0)} \setminus \Sigma_S$.
\item $P^{(0)}$ has the same behavior near $0$ as $S$ has (see the RH problem for $S$).
\end{itemize}
A parametrix $P^{(0)}$ with these properties can be constructed
using a solution $\Psi_{\alpha}$ of the model
RH problem of Subsection \ref{modelRHP1}. The construction is done
in three steps.

\paragraph{Step 1: Transformation to constant jumps.}

We seek $P^{(0)}$ in the form
\begin{equation} \label{P0hat}
    P^{(0)}(z) = \widehat{P}^{(0)}(z) \,e^{n \varphi_t(z) \sigma_3}
    \,z^{-\alpha \sigma_3},
    \quad \text{ for } z \in \overline{U^{(0)}} \setminus \Sigma_S,
\end{equation}
where as usual $z^{-\alpha}$ denotes the principal branch. It then follows
from the RH problem for $P^{(0)}$ that
 $\widehat{P}^{(0)}$ should satisfy the following RH problem.

\paragraph{Riemann-Hilbert problem for $\widehat{P}^{(0)}$}
\begin{itemize}
\item $\widehat{P}^{(0)} : \overline{U^{(0)}} \setminus \Sigma_S \to
    \mathbb C^{2\times 2}$ is continuous and analytic on
    $U^{(0)} \setminus \Sigma_S$.
\item For $z \in \Sigma_S \cap U^{(0)}$, we have
\[ \widehat{P}^{(0)}_+(z) = \widehat{P}^{(0)}_-(z) \times
    \left\{ \begin{array}{ll}
    \begin{pmatrix}
        1 & 1 \\
        0 & 1
    \end{pmatrix},&\text{ for } z \in (0,\infty) \cap U^{(0)},\\[10pt]
    \begin{pmatrix}
        1 & 0 \\
        e^{2 \alpha \pi i} & 1
    \end{pmatrix},&\text{ for } z \text{ in } U^{(0)} \text{ on the upper lip of the lens},
    \\[10pt]
    \begin{pmatrix}
        0 & 1 \\
        -1 & 0
    \end{pmatrix},&\text{ for } z \in (-\infty,0) \cap U^{(0)},
    \\[10pt]
    \begin{pmatrix}
        1 & 0 \\
        e^{-2 \alpha \pi i} & 1
    \end{pmatrix},&\text{ for } z \text{ in } U^{(0)} \text{ on the lower lip of the lens}.
\end{array}
\right. \]
uniformly for $z \in \partial U^{(0)} \setminus \Sigma_S$.
\item If $\alpha < 0$, then
    \[ \widehat{P}^{(0)}(z) =  O \begin{pmatrix}
      |z|^{\alpha} & |z|^{\alpha} \\
      |z|^{\alpha} & |z|^{\alpha}
      \end{pmatrix} \quad \text{ as $z \to 0$}, \]
while if $\alpha \geq 0$ we have that
\begin{align*} \widehat{P}^{(0)}(z) & = O \begin{pmatrix}
    |z|^{\alpha} & |z|^{-\alpha} \\
    |z|^{\alpha} & |z|^{-\alpha}
  \end{pmatrix} \quad \text{ as   $z \to 0$ from outside the lens, and} \\
    \widehat{P}^{(0)}(z) & = O \begin{pmatrix}
    |z|^{-\alpha} & |z|^{-\alpha} \\
    |z|^{-\alpha} & |z|^{-\alpha}
  \end{pmatrix} \quad \text{ as  $z \to 0$ from inside the lens}.
  \end{align*}
\end{itemize}
Note that the jump matrices of $\widehat{P}^{(0)}$ do not depend on $t$.

The reader may note the similarities between the
above RH problem for $\widehat{P}^{(0)}$ and the RH problem for
$\Psi_{\alpha}$ from Subsection \ref{modelRHP1}. In the next step we show how
we can use $\Psi_{\alpha}$ to construct a solution of the RH problem for
$\widehat{P}^{(0)}$.

\paragraph{Step 2: The construction of $\widehat P^{(0)}$ in terms of $\Psi_{\alpha}$.}

Recall that $\Sigma_S$ in $U^{(0)}$ was taken such that $z \mapsto
f_t(z) - f_t(0)$ maps $\Sigma_S \cap U^{(0)}$ onto a subset of $\Sigma$,
where  $\Sigma$ is the contour in the RH problem
for $\Psi_{\alpha}$, see Subsection \ref{modelRHP1}.

We choose any solution $\Psi_{\alpha}$ of the model RH problem and
we define $\widehat{P}^{(0)}$ by
\begin{equation} \label{P0hat2}
    \widehat{P}^{(0)}(z) =
    E(z)\,\Psi_{\alpha}\left(
    n^{\frac{2}{3}}(f_t(z)-f_t(0));
    n^{\frac{2}{3}} f_t(0) \right),
    \quad \text{ for } z \in \overline{U^{(0)}} \setminus \Sigma_S,
\end{equation}
where $E = E_{n,N}$ is analytic in $U^{(0)}$.
Taking $P^{(0)}$ as in \eqref{P0hat} with $\widehat P^{(0)}$ as
in \eqref{P0hat2} we find that all the conditions of the RH problem
for $P^{(0)}$ are satisfied, except for the matching condition
\begin{equation} \label{matching}
    P^{(0)}(z)\,\left(P^{(\infty)}(z)\right)^{-1} = I + O(n^{-1/3}),
    \end{equation}
as $n \to \infty$, $t \to 1$ such that $n^{2/3}(t - 1) = O(1)$.

\paragraph{Step 3: Matching condition.}

To be able to satisfy \eqref{matching} we have to take $E$
in the following way
\begin{equation} \label{Edef}
    E(z) = P^{(\infty)}(z)\,z^{\alpha  \sigma_3}\,
    \frac{1}{\sqrt{2}} \begin{pmatrix}
    1 & -i \\
    -i & 1
  \end{pmatrix}
  \left(n^{\frac{2}{3}}(f_t(z) - f_t(0)) \right)^{\sigma_3/4}, \quad
    \text{ for } z \in U^{(0)} \setminus [a_t,0],
\end{equation}
where both branches are taken as principal. Clearly then
$E$ is analytic in $U^{(0)} \setminus [a_t,0]$.
It turns out that $E$ has analytic continuation to $U^{(0)}$.
This follows by direct calculation, but it relies on the fact that
we chose $[a_t,0]$ as the jump contour for $P^{(\infty)}$.

With the  choice \eqref{Edef} for $E$, we now show that \eqref{matching}
is satisfied as well. By \eqref{P0hat},
\eqref{P0hat2}, we have for $z \in
\partial U^{(0)} \setminus \Sigma_S$,
\begin{alignat*}{2}
    P^{(0)}(z) & =
    E(z)\,\Psi_{\alpha}\left(n^{\frac{2}{3}}(f_t(z) - f_t(0));
    n^{\frac{2}{3}}f_t(0)\right) e^{n\varphi_t\,\sigma_3}\,z^{-\alpha
    \sigma_3}
\end{alignat*}
and we are interested in the behavior as $n \to
\infty$, $t \to 1$ such that $n^{2/3}(t -1) = O(1)$.

We show first that $n^{2/3} f_t(0)$ remains bounded.
\begin{lemma} \label{matchinglemma1}
    Suppose $n \to \infty$, $t \to 1$ such that $n^{2/3} (t-1) = O(1)$.
    Then $n^{2/3} f_t(0)$ remains bounded.
    More precisely, if $n^{2/3} (t-1) \to L \in \mathbb R$,
    then
    \begin{equation} \label{c2Vprop}
         n^{2/3} f_t(0) \to - c_{2,V} L = s,
        \end{equation}
    where
    \begin{equation} \label{c2Vdef}
        c_{2,V} = (c_{1,V})^{2/3} \frac{db_t}{dt} \bigg|_{t=1}
    \end{equation}
    and $c_{1,V}$ is the constant in \eqref{c1def}.
\end{lemma}
\begin{proof}
It follows from \eqref{ftdef},  that
\begin{align*}
    f_t(0) & = \left(\frac{1}{2} \sqrt{-a_t} h_t(b_t)\right)^{2/3} (-b_t) \chi_t^{2/3}(0) \\
    & = - \left( \frac{1}{2} \sqrt{-a} h_1(0) \right)^{2/3} (t-1) \frac{db_t}{dt} \bigg|_{t=1}
        + O((t-1)^2) \qquad \mbox{ as } t \to 1.
\end{align*}
By \eqref{c1def} and \eqref{htdef}, we have
\begin{equation} \label{c1Vformula}
    c_{1,V} = \frac{1}{2} \sqrt{-a} h_1(0),
    \end{equation}
so that \eqref{c2Vprop}--\eqref{c2Vdef} indeed follows if $n^{2/3}
(t-1) \to L$.
\end{proof}
If we use the formula \eqref{edgespeed1}
for the $t$-derivative of $b_t$ at $t=1$, then we find from
\eqref{c2Vdef} that
\begin{equation} \label{c2Vformula}
    c_{2,V} = 2 (-a)^{-1/2}\,c_{1,V}^{-1/3}.
    \end{equation}

Now we continue with the proof of \eqref{matching}.

\begin{lemma} \label{matchinglemma2}
Suppose that $n \to \infty$, $t \to 1$ such that $n^{2/3} (t-1) = O(1)$.
Then \eqref{matching} holds.
\end{lemma}
\begin{proof}
In the proof all $O$-terms are for $n \to \infty$, $t \to 1$
such that $n^{2/3} (t-1)$ is bounded.

By Lemma \ref{matchinglemma1} the values $n^{2/3} f_t(0)$ remain bounded.
Since the asymptotic condition (c) in the RH problem for $\Psi_{\alpha}$
is valid uniformly for $s$ in bounded subsets of $\mathbb R$, we find
by \eqref{P0hat}, \eqref{P0hat2}, and \eqref{Edef}
\begin{alignat}{2} \nonumber
    P^{(0)}(z) & =
    E(z)\,\left(n^{\frac{2}{3}}(f_t(z)-f_t(0))\right)^{-\sigma_3/4}\,\frac{1}{\sqrt{2}}\,
    \begin{pmatrix}
    1 & i \\
    i & 1
  \end{pmatrix}
  \left(I + O(n^{-1/3})\right) \\ \nonumber
  & \qquad \times \exp\left(- \theta(n^{2/3}(f_t(z)-f_t(0)); n^{2/3} f_t(0)) \sigma_3\right)
    e^{n \varphi_t \sigma_3} z^{-\alpha \sigma_3} \\[10pt]
    & = \label{P0expansion}
    P^{(\infty)}(z) (I + O(n^{-1/3}))
    \exp\left(- \left(\theta(n^{2/3}(f_t(z)-f_t(0)); n^{2/3} f_t(0)) - n \varphi_t \right) \sigma_3 \right)
\end{alignat}
uniformly for $z \in \partial U^{(0)}$.
As before we denote $\theta(\zeta;s) = \frac{2}{3} \zeta^{3/2} + s \zeta^{1/2}$.

The next step is to evaluate the expression in the exponential factor. We have
\begin{alignat*}{2}
    & \theta(n^{2/3}(f_t(z)-f_t(0)); n^{2/3} f_t(0))
    -n \varphi_t  \\
    & =
    \frac{2}{3} n \left[(f_t(z)-f_t(0))^{3/2} - (f_t(z))^{3/2}\right] +
    n f_t(0) (f_t(z)-f_t(0))^{1/2}.
\end{alignat*}
We will show that this is $O(n^{-1/3})$ uniformly for $z \in \partial U^{(0)}$.
To that end, it is enough to show that
\begin{alignat}{2} \nonumber
   F(t,z) & := (f_t(z)-f_t(0))^{3/2} - (f_t(z))^{3/2} + \frac{3}{2}
    f_t(0)\,(f_t(z)-f_t(0))^{1/2} \\[5pt]
    & = O((t-1)^2) \qquad \text{as } t \to 1, \label{thetaprop}
    \end{alignat}
uniformly for $z \in \partial U^{(0)}$.

By \eqref{ftprop}, we have
\begin{equation}
    F(1,z) = 0.
\end{equation}
Moreover,
\begin{alignat*}{2}
\frac{\partial}{\partial t} F(t,z) &=
    \frac{3}{2}\,(f_t(z)-f_t(0))^{\frac{1}{2}}\,\frac{\partial }{\partial t}\,(f_t(z)-f_t(0)) -
    \frac{3}{2}(f_t(z))^{\frac{1}{2}}\,\frac{\partial }{\partial t}\,f_t(z)\\
    & \qquad + \frac{3}{2}\,\left(\frac{d }{d t}\,f_t(0)\right)
    (f_t(z)-f_t(0))^{\frac{1}{2}} \\
    & \qquad +
\frac{3}{2}\,f_t(0)\,\frac{1}{2}\,(f_t(z)-f_t(0))^{-\frac{1}{2}}\,\frac{\partial }{\partial t}\,(f_t(z)-f_t(0)).
\end{alignat*}
Let $t = 1$ and again use \eqref{ftprop} and \eqref{ftdef}. Due to cancellations one finds
\begin{equation}
    \frac{\partial}{\partial t}\,F(1,z) = 0.
\end{equation}
Since, in addition, $F(t,z)$ is analytic in both variables
and bounded with respect to $z$ in $\partial U^{(0)}$, it follows
from a Taylor expansion that $F(t,z) = O((t - 1)^2)$, as claimed
in \eqref{thetaprop}.
Thus
\[ \theta(n^{2/3}(f_t(z)-f_t(0)); n^{2/3} f_t(0)) -n \varphi_t
    = O(n^{-1/3}), \]
so that \eqref{P0expansion} leads to
\[ P^{(0)}(z) = P^{(\infty)}(z) \left(I + O(n^{-1/3})\right), \]
uniformly for $z \in \partial U^{(0)}$.
Then \eqref{matching} follows   since $P^{(\infty)}(z)$ and its inverse are
bounded in $n$ and $t$, uniformly for $z \in \partial U^{(0)}$.
\end{proof}

This completes the construction of the parametrix $P^{(0)}$.

\begin{remark}
The local parametrix $P^{(0)}$ is constructed with the help
of a solution $\Psi_{\alpha}$ of the model RH problem.
Since the solution $\Psi_{\alpha}$ is not unique
(see Proposition \ref{nonunique}),
the local parametrix is not unique. In what follows we can
take any $P^{(0)}$ and it will not affect the final results
(Theorems \ref{theorem1} and \ref{theorem2}).
\end{remark}

\subsubsection{Third transformation $S \mapsto R$}

\begin{figure}[t]
\ifx\JPicScale\undefined\def\JPicScale{1}\fi
\unitlength \JPicScale mm
\begin{picture}(80,15)(-30,20)

\linethickness{0.3mm}
\put(10,30){\line(1,0){15}}

\linethickness{0.30mm}
\put(30.00,30.00){\circle{10.00}}

\linethickness{0.30mm}
\put(60.00,30.00){\circle{10.00}}

\linethickness{0.3mm}
\put(65,30){\line(1,0){15}}

\linethickness{0.3mm}
\multiput(56.5,34.07)(0.13,-0.11){3}{\line(1,0){0.13}}
\multiput(56.11,34.37)(0.13,-0.1){3}{\line(1,0){0.13}}
\multiput(55.71,34.67)(0.2,-0.15){2}{\line(1,0){0.2}}
\multiput(55.3,34.95)(0.2,-0.14){2}{\line(1,0){0.2}}
\multiput(54.88,35.23)(0.21,-0.14){2}{\line(1,0){0.21}}
\multiput(54.46,35.49)(0.21,-0.13){2}{\line(1,0){0.21}}
\multiput(54.03,35.74)(0.21,-0.13){2}{\line(1,0){0.21}}
\multiput(53.6,35.98)(0.22,-0.12){2}{\line(1,0){0.22}}
\multiput(53.15,36.21)(0.22,-0.11){2}{\line(1,0){0.22}}
\multiput(52.71,36.42)(0.22,-0.11){2}{\line(1,0){0.22}}
\multiput(52.25,36.62)(0.23,-0.1){2}{\line(1,0){0.23}}
\multiput(51.79,36.81)(0.23,-0.1){2}{\line(1,0){0.23}}
\multiput(51.33,36.99)(0.46,-0.18){1}{\line(1,0){0.46}}
\multiput(50.86,37.16)(0.47,-0.17){1}{\line(1,0){0.47}}
\multiput(50.39,37.31)(0.47,-0.15){1}{\line(1,0){0.47}}
\multiput(49.91,37.45)(0.48,-0.14){1}{\line(1,0){0.48}}
\multiput(49.43,37.58)(0.48,-0.13){1}{\line(1,0){0.48}}
\multiput(48.95,37.69)(0.48,-0.11){1}{\line(1,0){0.48}}
\multiput(48.46,37.79)(0.49,-0.1){1}{\line(1,0){0.49}}
\multiput(47.97,37.88)(0.49,-0.09){1}{\line(1,0){0.49}}
\multiput(47.48,37.95)(0.49,-0.07){1}{\line(1,0){0.49}}
\multiput(46.98,38.01)(0.49,-0.06){1}{\line(1,0){0.49}}
\multiput(46.49,38.06)(0.49,-0.05){1}{\line(1,0){0.49}}
\multiput(45.99,38.09)(0.5,-0.03){1}{\line(1,0){0.5}}
\multiput(45.5,38.11)(0.5,-0.02){1}{\line(1,0){0.5}}
\multiput(45,38.12)(0.5,-0.01){1}{\line(1,0){0.5}}
\multiput(44.5,38.11)(0.5,0.01){1}{\line(1,0){0.5}}
\multiput(44.01,38.09)(0.5,0.02){1}{\line(1,0){0.5}}
\multiput(43.51,38.06)(0.5,0.03){1}{\line(1,0){0.5}}
\multiput(43.02,38.01)(0.49,0.05){1}{\line(1,0){0.49}}
\multiput(42.52,37.95)(0.49,0.06){1}{\line(1,0){0.49}}
\multiput(42.03,37.88)(0.49,0.07){1}{\line(1,0){0.49}}
\multiput(41.54,37.79)(0.49,0.09){1}{\line(1,0){0.49}}
\multiput(41.05,37.69)(0.49,0.1){1}{\line(1,0){0.49}}
\multiput(40.57,37.58)(0.48,0.11){1}{\line(1,0){0.48}}
\multiput(40.09,37.45)(0.48,0.13){1}{\line(1,0){0.48}}
\multiput(39.61,37.31)(0.48,0.14){1}{\line(1,0){0.48}}
\multiput(39.14,37.16)(0.47,0.15){1}{\line(1,0){0.47}}
\multiput(38.67,36.99)(0.47,0.17){1}{\line(1,0){0.47}}
\multiput(38.21,36.81)(0.46,0.18){1}{\line(1,0){0.46}}
\multiput(37.75,36.62)(0.23,0.1){2}{\line(1,0){0.23}}
\multiput(37.29,36.42)(0.23,0.1){2}{\line(1,0){0.23}}
\multiput(36.85,36.21)(0.22,0.11){2}{\line(1,0){0.22}}
\multiput(36.4,35.98)(0.22,0.11){2}{\line(1,0){0.22}}
\multiput(35.97,35.74)(0.22,0.12){2}{\line(1,0){0.22}}
\multiput(35.54,35.49)(0.21,0.13){2}{\line(1,0){0.21}}
\multiput(35.12,35.23)(0.21,0.13){2}{\line(1,0){0.21}}
\multiput(34.7,34.95)(0.21,0.14){2}{\line(1,0){0.21}}
\multiput(34.29,34.67)(0.2,0.14){2}{\line(1,0){0.2}}
\multiput(33.89,34.37)(0.2,0.15){2}{\line(1,0){0.2}}
\multiput(33.5,34.07)(0.13,0.1){3}{\line(1,0){0.13}}
\multiput(33.12,33.75)(0.13,0.11){3}{\line(1,0){0.13}}

\linethickness{0.3mm}
\multiput(33.12,26.25)(0.13,-0.11){3}{\line(1,0){0.13}}
\multiput(33.5,25.93)(0.13,-0.1){3}{\line(1,0){0.13}}
\multiput(33.89,25.63)(0.2,-0.15){2}{\line(1,0){0.2}}
\multiput(34.29,25.33)(0.2,-0.14){2}{\line(1,0){0.2}}
\multiput(34.7,25.05)(0.21,-0.14){2}{\line(1,0){0.21}}
\multiput(35.12,24.77)(0.21,-0.13){2}{\line(1,0){0.21}}
\multiput(35.54,24.51)(0.21,-0.13){2}{\line(1,0){0.21}}
\multiput(35.97,24.26)(0.22,-0.12){2}{\line(1,0){0.22}}
\multiput(36.4,24.02)(0.22,-0.11){2}{\line(1,0){0.22}}
\multiput(36.85,23.79)(0.22,-0.11){2}{\line(1,0){0.22}}
\multiput(37.29,23.58)(0.23,-0.1){2}{\line(1,0){0.23}}
\multiput(37.75,23.38)(0.23,-0.1){2}{\line(1,0){0.23}}
\multiput(38.21,23.19)(0.46,-0.18){1}{\line(1,0){0.46}}
\multiput(38.67,23.01)(0.47,-0.17){1}{\line(1,0){0.47}}
\multiput(39.14,22.84)(0.47,-0.15){1}{\line(1,0){0.47}}
\multiput(39.61,22.69)(0.48,-0.14){1}{\line(1,0){0.48}}
\multiput(40.09,22.55)(0.48,-0.13){1}{\line(1,0){0.48}}
\multiput(40.57,22.42)(0.48,-0.11){1}{\line(1,0){0.48}}
\multiput(41.05,22.31)(0.49,-0.1){1}{\line(1,0){0.49}}
\multiput(41.54,22.21)(0.49,-0.09){1}{\line(1,0){0.49}}
\multiput(42.03,22.12)(0.49,-0.07){1}{\line(1,0){0.49}}
\multiput(42.52,22.05)(0.49,-0.06){1}{\line(1,0){0.49}}
\multiput(43.02,21.99)(0.49,-0.05){1}{\line(1,0){0.49}}
\multiput(43.51,21.94)(0.5,-0.03){1}{\line(1,0){0.5}}
\multiput(44.01,21.91)(0.5,-0.02){1}{\line(1,0){0.5}}
\multiput(44.5,21.89)(0.5,-0.01){1}{\line(1,0){0.5}}
\multiput(45,21.88)(0.5,0.01){1}{\line(1,0){0.5}}
\multiput(45.5,21.89)(0.5,0.02){1}{\line(1,0){0.5}}
\multiput(45.99,21.91)(0.5,0.03){1}{\line(1,0){0.5}}
\multiput(46.49,21.94)(0.49,0.05){1}{\line(1,0){0.49}}
\multiput(46.98,21.99)(0.49,0.06){1}{\line(1,0){0.49}}
\multiput(47.48,22.05)(0.49,0.07){1}{\line(1,0){0.49}}
\multiput(47.97,22.12)(0.49,0.09){1}{\line(1,0){0.49}}
\multiput(48.46,22.21)(0.49,0.1){1}{\line(1,0){0.49}}
\multiput(48.95,22.31)(0.48,0.11){1}{\line(1,0){0.48}}
\multiput(49.43,22.42)(0.48,0.13){1}{\line(1,0){0.48}}
\multiput(49.91,22.55)(0.48,0.14){1}{\line(1,0){0.48}}
\multiput(50.39,22.69)(0.47,0.15){1}{\line(1,0){0.47}}
\multiput(50.86,22.84)(0.47,0.17){1}{\line(1,0){0.47}}
\multiput(51.33,23.01)(0.46,0.18){1}{\line(1,0){0.46}}
\multiput(51.79,23.19)(0.23,0.1){2}{\line(1,0){0.23}}
\multiput(52.25,23.38)(0.23,0.1){2}{\line(1,0){0.23}}
\multiput(52.71,23.58)(0.22,0.11){2}{\line(1,0){0.22}}
\multiput(53.15,23.79)(0.22,0.11){2}{\line(1,0){0.22}}
\multiput(53.6,24.02)(0.22,0.12){2}{\line(1,0){0.22}}
\multiput(54.03,24.26)(0.21,0.13){2}{\line(1,0){0.21}}
\multiput(54.46,24.51)(0.21,0.13){2}{\line(1,0){0.21}}
\multiput(54.88,24.77)(0.21,0.14){2}{\line(1,0){0.21}}
\multiput(55.3,25.05)(0.2,0.14){2}{\line(1,0){0.2}}
\multiput(55.71,25.33)(0.2,0.15){2}{\line(1,0){0.2}}
\multiput(56.11,25.63)(0.13,0.1){3}{\line(1,0){0.13}}
\multiput(56.5,25.93)(0.13,0.11){3}{\line(1,0){0.13}}

\thicklines \put(71.25,30){\line(1,0){1.25}}
\put(72.5,30){\vector(1,0){0.12}} \thicklines
\put(17.5,30){\line(1,0){1.25}} \put(18.75,30){\vector(1,0){0.12}}
\thicklines \put(29.38,35){\line(1,0){0.62}}
\put(29.38,35){\vector(-1,0){0.12}} \thicklines
\put(59.38,35){\line(1,0){0.62}} \put(59.38,35){\vector(-1,0){0.12}}
\thicklines \put(45,38.12){\line(1,0){0.62}}
\put(45.62,38.12){\vector(1,0){0.12}} \linethickness{0.3mm}
\put(45,21.88){\line(1,0){0.62}}
\put(45.62,21.88){\vector(1,0){0.12}}

\put(30,27.75){\makebox(0,0)[cc]{$a$}}

\put(60,27.75){\makebox(0,0)[cc]{$0$}}

\put(60,30){\makebox(0,0)[cc]{.}}

\put(30,30){\makebox(0,0)[cc]{.}}

\end{picture}
\caption{\label{figure4} Contour for the RH problem for $R$.}
\end{figure}
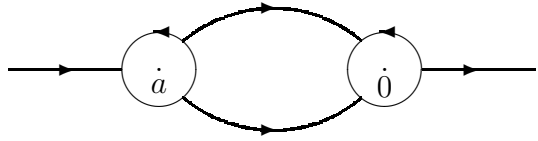

Having $P^{(\infty)}$, $P^{(a)}$, and $P^{(0)}$, we take $P$ as
in \eqref{Pdef}, and then we define
\begin{equation}
\label{Rdef}
R(z) = S(z)\,P^{-1}(z), \quad \text{ for } z \in
\mathbb{C} \setminus (\partial U^{(0)} \cup \partial U^{(a)} \cup
\Sigma_S).
\end{equation}
Since $S$ and $P$ have the same jump matrices on $U^{(0)} \cap
\Sigma_S$, $U^{(a)} \cap \Sigma_S$ and $(a,0) \setminus (U^{(0)}
\cup U^{(a)})$, we have that $R$ is analytic across these contours.
What remains are jumps for $R$ on the contour $\Sigma_R$ shown in
Figure \ref{figure4} with orientation that is also shown in the
figure. Then, $R$ satisfies the following RH problem.

\paragraph{Riemann-Hilbert problem for $R$}
\begin{itemize}
\item $R : \mathbb{C} \setminus \Sigma_R \to \mathbb C^{2\times 2}$ is analytic.
\item $R_+(z) = R_-(z)\,v_R(z)$ for $z \in \Sigma_R$, where\\
\begin{equation} \label{vRdef}
v_R = \left\{
    \begin{array}{ll}
        P^{(\infty)}\,(P^{(0)})^{-1},&\text{ on } \partial U^{(0)},\\
        P^{(\infty)}\,(P^{(a)})^{-1},&\text{ on } \partial U^{(a)},\\
        P^{(\infty)}\,v_S\,(P^{(\infty)})^{-1},&\text{ on } \Sigma_R
            \setminus (\partial U^{(0)} \cup \partial U^{(a)}).
\end{array}
\right.
\end{equation}
\item $R(z) = I + O(1/z)$ as $z \to \infty$.
\end{itemize}

Now let $n \to \infty$, $t \to 1$ such that $n^{2/3}(t -1) = O(1)$.
It then follows from the construction of the parametrices (see in
particular the RH problems for $P^{(0)}$ and $P^{(a)}$) that
\begin{equation} \label{vRasymptotics1}
v_R = \left\{
\begin{array}{ll}
I + O(n^{-1/3}),&\text{ on } \partial U^{(0)},\\[5pt]
I + O(n^{-1}),&\text{ on } \partial U^{(a)}.
\end{array}
\right.
\end{equation}
Furthermore, by regularity of the eigenvalue density, there is a constant $c > 0$ such that
\begin{align*}
   & \Re \varphi_t(z) > c > 0,  \qquad \text{ for } z \in \Sigma_R \cap (0,\infty), \\
   & \Re \tilde{\varphi}_t(z) > c > 0,  \qquad \text{ for } z\in \Sigma_R \cap (-\infty,a), \\
   & \Re \varphi_t(z) < -c < 0,  \quad
    \text{ for } z \in \Sigma_R \setminus (\partial U^{(0)} \cup \partial U^{(a)} \cup \mathbb R).
\end{align*}
This implies (see the RH problem for $S$) that $v_S = I +
O(e^{-cn})$ uniformly on $\Sigma_R \setminus (\partial U^{(0)} \cup
\partial U^{(a)})$, so that by \eqref{vRdef}
\begin{equation} \label{vRasymptotics2}
    v_R = I + O(e^{-2c n}) \qquad \text{ on } \Sigma_R \setminus
    (\partial U^{(0)} \cup \partial U^{(a)}).
\end{equation}
The $O$-terms in \eqref{vRasymptotics1} and \eqref{vRasymptotics2}
are uniform on the indicated contours.
In addition, it follows from \eqref{g-identity1}, \eqref{g-identity2}, \eqref{gtdef}, and the
growth condition \eqref{Vgrowth} on $V$ that for any $C > 0$ there exists
$r = r(C) > 1$ such that $\varphi_t(x) \geq C \log x$ for $x \geq r$,
and $\tilde{\varphi}_t(x) \geq C \log |x|$ for $x \leq -r$.
Combined with \eqref{vRasymptotics2} this implies that
\begin{equation} \label{vRasymptotics3}
||v_R - I||_{L^2\left(\Sigma_R \setminus
    (\partial U^{(0)} \cup \partial U^{(a)})\right)} = O(e^{-2c n}), \qquad \text{ as } n \to \infty.
\end{equation}
Thus, by \eqref{vRasymptotics1}--\eqref{vRasymptotics3}, as $n \to \infty$
and $t \to 1$ such that $n^{2/3} (t-1) = O(1)$, the jump matrix for
$R$ is close to $I$ in both $L^{2}$ and $L^{\infty}$ norm on
$\Sigma_R$, indeed
\begin{equation}
||v_R - I||_{L^2(\Sigma_R) \cap L^{\infty}(\Sigma_R)} = O(n^{-1/3}).
\end{equation}
Standard estimates using $L^2$-boundedness of the
operators $C_{\pm}$ on $L^2(\Sigma_R)$ together with the correspondence between RH
problems and singular integral equations now imply that
\begin{equation}
\label{RclosetoI} R(z) = I + O(n^{-1/3}), \quad \text{uniformly for
} z \in \mathbb{C} \setminus \Sigma_R,
\end{equation}
as $n\to \infty$, $t \to 1$ such that $n^{2/3}(t -1) = O(1)$.
To get the uniform bound \eqref{RclosetoI} up to the contour one needs a contour deformation
argument. Again, see the presentation in \cite{Deift} for more details.

This completes the steepest descent analysis of the
RH problem for $Y$.

\subsection{Completion of the proofs of Theorem \ref{theorem1} and \ref{theorem2}}

Having completed the steepest descent analysis we are now ready
for the proofs of Theorem \ref{theorem1} and \ref{theorem2}.
We start by rewriting the kernel \eqref{CD-kernel} for $x,y \in U^{(0)} \cap \mathbb R$
according to the transformations $Y \mapsto T \mapsto S \mapsto R$ that
we did in the steepest descent analysis.
To state the result it is convenient to introduce $B = B_{n,N}$ as
\begin{equation} \label{Bdef}
    B(z) = R(z) E(z), \qquad \text{for } z \in U^{(0)},
\end{equation}
where $E$ and $R$ are defined in \eqref{Edef} and \eqref{Rdef}.
We also define for $x, s  \in \mathbb R$ the column vector
\begin{equation} \label{vecpsidef}
    \vec{\psi}_{\alpha}(x;s) = \begin{pmatrix} \psi_1(x;s) \\[5pt] \psi_2(x;s) \end{pmatrix}
        =   \left\{ \begin{array}{ll}
        \Psi_{\alpha,+}(x;s) \begin{pmatrix} 1 \\ 0 \end{pmatrix},
        & \text{ for } x > 0, \\
        \Psi_{\alpha,+}(x;s)  e^{-\alpha \pi i \sigma_3}
        \begin{pmatrix} 1 \\ 1 \end{pmatrix}, & \text{ for } x < 0,
        \end{array} \right.
\end{equation}
cf.\ \eqref{psi12def}.
We then have the following result.
\begin{lemma}
\label{CD-kernelFinite} Let $x, y \in  U^{(0)} \cap \mathbb{R} $.
Then,
\begin{alignat}{2} \nonumber
K_{n,N}(x,y) &= \frac{1}{2\pi i (x - y)}\,\left(\vec{\psi}_{\alpha}\left(n^{\frac{2}{3}}(f_t(y)-f_t(0)) ;
    n^{\frac{2}{3}}f_t(0)\right)\right)^{T}
\begin{pmatrix} 0 & 1\\ -1 & 0 \end{pmatrix}\\
&\qquad \times B^{-1}(y)
B(x)\, \vec{\psi}_{\alpha}\left(n^{\frac{2}{3}}(f_t(x)-f_t(0)) ; n^{\frac{2}{3}}\,f_t(0)\right).
\label{CD-kernelthroughB}
\end{alignat}
\end{lemma}

\begin{proof}
We start from the formula \eqref{CD-kernel} for the
eigenvalue correlation kernel. Using \eqref{Tdef} we obtain, for any
$x,y \in \mathbb{R}$,
\begin{alignat}{2}
\notag
K_{n,N}(x,y) &= |x|^{\alpha}\,e^{\frac{1}{2} n (2 g_{t,+}(x) - V_t(x)+ l_t)}
\,|y|^{\alpha}\,e^{\frac{1}{2} n (2 g_{t,+}(y) - V_t(y) + l_t)}\\
\label{CD-kernelthroughT}
&\qquad \times \frac{1}{2\pi i
(x-y)}\,\begin{pmatrix} 0 & 1
\end{pmatrix} T_+^{-1}(y)\,T_+(x) \begin{pmatrix} 1 \\
0\end{pmatrix}.
\end{alignat}
Using \eqref{g-identity1} and the fact that
$g_{t,+} = g_{t,-}$ on $(b_t, \infty)$, it follows that $2g_t - V_t
+l_t = -2\varphi_t$ on $(b_t, \infty)$. Then, by analytic
continuation, $2g_{t,+} - V_t +l_t = -2\varphi_{t,+}$ on all of
$\mathbb{R}$. Therefore we can rewrite \eqref{CD-kernelthroughT} as
\begin{alignat}{2}
\label{CD-kernelthroughT2}
K_{n,N}(x,y) &= |x|^{\alpha}\,e^{- n \varphi_{t,+}(x)}
\,|y|^{\alpha}\,e^{-n \varphi_{t,+}(y)}
\frac{1}{2\pi i
(x-y)}\,\begin{pmatrix} 0 & 1
\end{pmatrix} T_+^{-1}(y)\,T_+(x) \begin{pmatrix} 1 \\ 0\end{pmatrix}.
\end{alignat}

Now we analyze the effect of the transformations $T \mapsto S \mapsto R$
on the expression
$|x|^{\alpha}  e^{- n \varphi_{t,+}(x)} T_+(x)
\left(\begin{smallmatrix} 1 \\ 0\end{smallmatrix}\right)$
in case $x \in U^{(0)} \cap \mathbb R$.
The result is that for $x \in U^{(0)} \cap \mathbb R$,
\begin{align}
    |x|^{\alpha}  e^{- n \varphi_{t,+}(x)} T_+(x) \begin{pmatrix} 1 \\ 0\end{pmatrix}
    =\label{KnNrewrite-xpart1}
     B(x) \Psi_{\alpha,+}\left(n^{\frac{2}{3}}(f_t(x)-f_t(0));
    n^{\frac{2}{3}} f_t(0)\right)
        \begin{pmatrix} 1 \\ 0 \end{pmatrix}
\end{align}
in case $x > 0$, and
\begin{align}
    |x|^{\alpha}  e^{- n \varphi_{t,+}(x)} T_+(x) \begin{pmatrix} 1 \\ 0\end{pmatrix}
   = \label{KnNrewrite-xpart2}
     B(x) \Psi_{\alpha,+}\left(n^{\frac{2}{3}}(f_t(x)-f_t(0)) ;
    n^{\frac{2}{3}} f_t(0)\right)
        e^{-\alpha \pi i \sigma_3} \begin{pmatrix} 1 \\ 1 \end{pmatrix},
\end{align}
in case $x < 0$. Since the calculations for \eqref{KnNrewrite-xpart1} are easier,
we will only show how to obtain \eqref{KnNrewrite-xpart2}.
If $x \in U^{(0)} \cap \mathbb R$ and $x < 0$, then it follows
from \eqref{Sdef}
that
\begin{align} \nonumber
    |x|^{\alpha}  e^{- n \varphi_{t,+}(x)} T_+(x) \begin{pmatrix} 1 \\ 0\end{pmatrix}
    & = |x|^{\alpha}  e^{- n \varphi_{t,+}(x)} S_+(x)
        \begin{pmatrix} 1 \\ |x|^{-2\alpha} e^{2n \varphi_{t,+}(x)} \end{pmatrix} \\
    & = S_+(x) \left(|x|^{\alpha}  e^{- n \varphi_{t,+}(x)}\right)^{\sigma_3}
        \begin{pmatrix} 1 \\ 1 \end{pmatrix}.
        \label{KnNrewrite-xpart3}
    \end{align}
From \eqref{Rdef}, \eqref{Pdef}, \eqref{P0hat}, \eqref{Edef}, and \eqref{Bdef}, we find that
\[ S_+(x) = B(x) \Psi_{\alpha,+}\left(n^{\frac{2}{3}}(f_t(x)-f_t(0)) ;
    n^{\frac{2}{3}} f_t(0)\right)
        \left(e^{n \varphi_{t}(x)} x^{-\alpha}\right)^{\sigma_3}_+. \]
Inserting this into \eqref{KnNrewrite-xpart3} and
noting that $x^{-\alpha}_+ |x|^{\alpha} = e^{-\alpha \pi i}$ we indeed
obtain \eqref{KnNrewrite-xpart2}.

In a similar way, we find for $y \in U^{(0)} \cap \mathbb R$,
\begin{align}
|y|^{\alpha} e^{-n \varphi_{t,+}(y)}
\begin{pmatrix} 0 & 1
\end{pmatrix} T_+^{-1}(y)  =
\begin{pmatrix} 0 & 1 \end{pmatrix}
\Psi_{\alpha,+}^{-1}\left(n^{\frac{2}{3}}(f_t(y)-f_t(0)) ;
    n^{\frac{2}{3}} f_t(0)\right) B^{-1}(y),
    \label{KnNrewrite-ypart1}
\end{align}
in case $y > 0$, and
\begin{align}
|y|^{\alpha} e^{-n \varphi_{t,+}(y)}
\begin{pmatrix} 0 & 1
\end{pmatrix} T_+^{-1}(y)  =
\begin{pmatrix} -1 & 1 \end{pmatrix} e^{\alpha \pi i \sigma_3}
\Psi_{\alpha,+}^{-1}\left(n^{\frac{2}{3}}(f_t(y)-f_t(0)) ;
    n^{\frac{2}{3}} f_t(0)\right) B^{-1}(y),
    \label{KnNrewrite-ypart2}
\end{align}
in case $y < 0$. To rewrite \eqref{KnNrewrite-ypart1} and
\eqref{KnNrewrite-ypart2} we use the following fact, which is easy
to check. If $A$ is an invertible $2 \times 2$ matrix having
determinant $1$, then
\begin{equation}
\label{AinverseAndAtranspose} A^{-1} = \begin{pmatrix} 0 & -1 \\ 1 &
0 \end{pmatrix} A^{T}
\begin{pmatrix} 0 & 1 \\ -1 & 0 \end{pmatrix}.
\end{equation}
If we apply \eqref{AinverseAndAtranspose} to $\Psi_{\alpha,+}$
in \eqref{KnNrewrite-ypart1} and \eqref{KnNrewrite-ypart2}, then we get
\begin{align} \nonumber
|y|^{\alpha} e^{-n \varphi_{t,+}(y)}
\begin{pmatrix} 0 & 1
\end{pmatrix} T_+^{-1}(y) & =
\begin{pmatrix} 1 & 0 \end{pmatrix}
\Psi_{\alpha,+}^{T}\left(n^{\frac{2}{3}}(f_t(y)-f_t(0)) ;
    n^{\frac{2}{3}} f_t(0)\right) \\
    & \qquad \times \label{KnNrewrite-ypart3}
    \begin{pmatrix} 0 & 1 \\ -1 & 0 \end{pmatrix} B^{-1}(y),
\end{align}
in case $y > 0$, and
\begin{align} \nonumber
|y|^{\alpha} e^{-n \varphi_{t,+}(y)}
\begin{pmatrix} 0 & 1
\end{pmatrix} T_+^{-1}(y)
    & =  \begin{pmatrix} 1 & 1 \end{pmatrix} e^{-\alpha \pi i \sigma_3}
\Psi_{\alpha,+}^{T}\left(n^{\frac{2}{3}}(f_t(y)-f_t(0)) ;
    n^{\frac{2}{3}} f_t(0)\right) \\
    &\qquad \times \label{KnNrewrite-ypart4}
    \begin{pmatrix} 0 & 1 \\ -1 & 0 \end{pmatrix} B^{-1}(y),
\end{align}
in case $y < 0$.
Then \eqref{CD-kernelthroughB} follows if we insert
\eqref{KnNrewrite-xpart1}, \eqref{KnNrewrite-xpart2},
    \eqref{KnNrewrite-ypart3}, and \eqref{KnNrewrite-ypart4}
    into \eqref{CD-kernelthroughT2} and use the definition
\eqref{vecpsidef}.
\end{proof}

As in Theorem \ref{theorem1} we now fix $x,y \in \mathbb{R}$. We
define
\begin{equation} \label{xnyndef}
    x_n = \frac{x}{(c_1 n)^{2/3}}, \qquad \text{ and }
    \qquad
    y_n = \frac{y}{(c_1 n)^{2/3}}
\end{equation}
where $c_1$ is the constant from \eqref{c1def}.

In order to take the limit of $(c_1 n)^{-2/3} K_{n,N} (x_n,y_n)$
we need one more lemma. Recall that $B = RE$ is
defined in \eqref{Bdef}.
\begin{lemma} \label{Bprop}
    Let $n \to \infty$, $t \to 1$ such that $n^{2/3}(t-1) \to L$.
    Let $x, y \in \mathbb R$ and let $x_n$ and $y_n$ defined as in \eqref{xnyndef},
Then the following hold.
\begin{enumerate}
\item[\rm (a)] $n^{2/3} f_t(0) \to s$,
\item[\rm (b)] $n^{2/3}(f_t(x_n) - f_t(0)) \to x$
and $n^{2/3}(f_t(y_n) - f_t(0)) \to y$,
\item[\rm (c)] $ B^{-1}(y_n)\,B(x_n) = I + O\left(\frac{x-y}{n^{1/3}}\right)$
where the implied constant in the $O$-term is uniform with respect to $x$ and $y$.
\end{enumerate}
\end{lemma}
\begin{proof}
(a) This  follows from Lemma \ref{matchinglemma1}.

(b) By \eqref{c1def} and \eqref{htdef} we have
$c_1 = \frac{1}{2} \sqrt{-a} h_1(0)$, so that
$f_1'(0) = c_1^{2/3}$ by \eqref{ftprime}.
Taking note of the definitions \eqref{xnyndef}, we then
obtain part (b), since $f_t \to f_1$ uniformly in $U^{(0)}$.

(c)
We have
\begin{alignat}{2} \nonumber
    R^{-1}(y_n)\,R(x_n) &= I + R^{-1}(y_n)(R(x_n) - R(y_n)) \\
    \label{Rprop}
    &= I + R^{-1}(y_n)\,(x_n - y_n) \int_0^1 R'(t x_n + (1-t) y_n) dt.
\end{alignat}
Recall that $R$ is analytic in $U^{(0)}$, and that $R(z) = I + O(n^{-1/3})$ by \eqref{RclosetoI},
uniformly in $U^{(0)}$. Since $\det R \equiv 1$, we find that $R^{-1}(y_n)$
remains bounded as $n \to \infty$. It also follows from \eqref{RclosetoI} and Cauchy's theorem,
that $R'(z) = O(n^{-1/3})$  for $z$ in a
neighborhood of the origin. By \eqref{Rprop} we then obtain
\begin{equation} \label{Bpropeq1}
    R^{-1}(y_n)\,R(x_n) = I + O\left(\frac{x-y}{n}\right).
\end{equation}
Using analyticity of $E$ in a neighborhood of the origin with
$E(z) = O(n^{1/6})$, see \eqref{Edef}, and the
fact that $\det E \equiv 1$, we obtain in the same way
\begin{equation} \label{Bpropeq2}
    E^{-1}(y_n)\,E(x_n) = I + O\left(\frac{x-y}{n^{1/3}} \right).
\end{equation}
The implied constants in \eqref{Bpropeq1} and \eqref{Bpropeq2}
are independent of $x$ and $y$.

Using \eqref{Bpropeq1},  \eqref{Bpropeq2}, and the fact that
$E(x_n) = O(n^{1/6})$ and $E^{-1}(y_n) = O(n^{1/6})$, we obtain from \eqref{Bdef}
\begin{alignat*}{2}
B^{-1}(y_n)\,B(x_n) &= E^{-1}(y_n)\,\left(I + O\left(\frac{x-y}{n}\right)\right)\,E(x_n) \\
                    &= E^{-1}(y_n)\,E(x_n) +  O(n^{1/6}) \,O\left( \frac{x-y}{n} \right)\,O(n^{1/6}) \\
                    &= I + O\left(\frac{x-y}{n^{1/3}}\right).
\end{alignat*}
This completes the proof of part (c).
\end{proof}

\begin{varproof} \textbf{of Theorems \ref{theorem1} and \ref{theorem2}.}
We let $n,N \to \infty$, $t=n/N \to 1$, in such a way that $n^{2/3}(t - 1) \to L$.
Then by parts (a) and (b) of Lemma \ref{Bprop}, we have
\[ \vec{\psi}_{\alpha}(n^{2/3} (f_t(x_n) - f_t(0)); n^{2/3} f_t(0))
    \to \vec{\psi}_{\alpha}(x;s) \]
and similarly if we replace $x_n$ by $y_n$. The existence of the
limit \eqref{KnNlimit} then follows easily from Lemma
\ref{CD-kernelFinite} and part (c) of Lemma \ref{Bprop}, which
proves Theorem \ref{theorem1}.

We also find that the limiting kernel $K_{\alpha}^{edge}(x,y;s)$ is given by
\begin{alignat}{2} \nonumber
K_{\alpha}^{edge}(x,y;s) & =
    \frac{1}{2\pi i(x-y)}
        \vec{\psi}_{\alpha}(y;s)^T \begin{pmatrix} 0 & 1 \\ -1 & 0 \end{pmatrix}
        \vec{\psi}_{\alpha}(x;s)
\end{alignat}
and so \eqref{Kintform} follows because of \eqref{vecpsidef}. The
model RH problem is solvable for every $s \in \mathbb R$ by
Proposition \ref{ExistenceProp} and so we have also proved Theorem
\ref{theorem2}.
\end{varproof}

\section{Proof of Theorems \ref{theorem3} and \ref{theorem4}}
\label{section3}

We prove Theorem \ref{theorem3} and Theorem \ref{theorem4} by first establishing,
with the help of \cite{BBIK},
a connection between the model RH problem and the RH problem for
Painlev\'e II in the form due to Flaschka and Newell \cite{FN}.
We can then use known properties of the RH problem for Painlev\'e II
to prove the theorems.

\subsection{The Painlev\'e II RH problem} \label{FNRHP}

We review the RH problem for the Painlev\'e II equation
$q''(s) = sq + 2q^3 - \nu$, as first given by Flaschka and Newell \cite{FN},
see also \cite{FIKN} and \cite{FZ}.
We will assume that
\[ \nu > - 1/2. \]
The RH problem involves three complex constants $a_1$, $a_2$, $a_3$
satisfying
\begin{equation} \label{a1a2a3variety}
    a_1 + a_2 + a_3 + a_1 a_2 a_3 = - 2i \sin \nu \pi,
    \end{equation}
and certain connection matrices $E_j$.

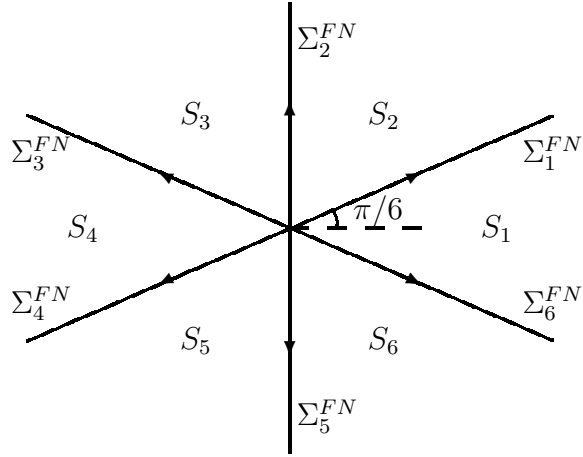
\begin{figure}[t]
\ifx\JPicScale\undefined\def\JPicScale{1}\fi
\unitlength \JPicScale mm
\begin{picture}(80,60)(-30,10)
\linethickness{0.3mm}
\multiput(10,55)(0.28,-0.12){125}{\line(1,0){0.28}}
\linethickness{0.3mm} \put(45,40){\line(0,1){30}}
\linethickness{0.3mm}
\multiput(45,40)(0.28,0.12){125}{\line(1,0){0.28}}
\linethickness{0.3mm} \put(45,10){\line(0,1){30}}
\linethickness{0.3mm}
\multiput(45,40)(0.28,-0.12){125}{\line(1,0){0.28}}
\linethickness{0.3mm}
\multiput(10,25)(0.28,0.12){125}{\line(1,0){0.28}}
\put(80,50){\makebox(0,0)[cc]{$\Sigma_1^{FN}$}}

\put(50,65){\makebox(0,0)[cc]{$\Sigma_2^{FN}$}}

\put(12,50){\makebox(0,0)[cc]{$\Sigma_3^{FN}$}}

\put(12,30){\makebox(0,0)[cc]{$\Sigma_4^{FN}$}}

\put(50,15){\makebox(0,0)[cc]{$\Sigma_5^{FN}$}}

\put(80,30){\makebox(0,0)[cc]{$\Sigma_6^{FN}$}}

\put(72.5,40){\makebox(0,0)[cc]{$S_1$}}

\put(57.5,55){\makebox(0,0)[cc]{$S_2$}}

\put(32.5,55){\makebox(0,0)[cc]{$S_3$}}

\put(17.5,40){\makebox(0,0)[cc]{$S_4$}}

\put(32.5,25){\makebox(0,0)[cc]{$S_5$}}

\put(57.5,25){\makebox(0,0)[cc]{$S_6$}}

\put(57.5,12.5){\makebox(0,0)[cc]{}}


\thicklines \put(45,23){\line(0,1){1}}
\put(45,23){\vector(0,-1){0.12}} \thicklines
\put(45,57){\line(0,1){1}} \put(45,57){\vector(0,1){0.12}}
\thicklines \multiput(27.5,47.5)(0.27,-0.12){16}{\line(1,0){0.27}}
\put(27.5,47.5){\vector(-2,1){0.12}} \thicklines
\multiput(27.5,32.5)(0.27,0.12){16}{\line(1,0){0.27}}
\put(27.5,32.5){\vector(-2,-1){0.12}} \thicklines
\multiput(58.12,45.62)(0.27,0.12){16}{\line(1,0){0.27}}
\put(62.5,47.5){\vector(2,1){0.12}} \thicklines
\multiput(58.12,34.38)(0.27,-0.12){16}{\line(1,0){0.27}}
\put(62.5,32.5){\vector(2,-1){0.12}}

\linethickness{0.3mm}
\multiput(58.12,34.38)(0.27,-0.12){16}{\line(1,0){0.27}}
\put(62.5,32.5){\vector(2,-1){0.12}} \linethickness{0.3mm}
\put(45,40){\line(1,0){2.5}} \linethickness{0.3mm}
\put(50,40){\line(1,0){2.5}} \linethickness{0.3mm}
\put(55,40){\line(1,0){2.5}} \linethickness{0.3mm}
\put(60,40){\line(1,0){2.5}} \linethickness{0.3mm}
\multiput(51.25,40)(0.07,0.53){1}{\line(0,1){0.53}}
\multiput(51.29,41.06)(0.03,-0.53){1}{\line(0,-1){0.53}}
\multiput(51.16,41.58)(0.13,-0.52){1}{\line(0,-1){0.52}}
\multiput(50.93,42.06)(0.11,-0.24){2}{\line(0,-1){0.24}}
\multiput(50.62,42.5)(0.1,-0.15){3}{\line(0,-1){0.15}}

\put(56.75,42.25){\makebox(0,0)[cc]{$\pi / 6$}}

\end{picture}
\caption{\label{figure5} Contour for the RH problem for $\Psi_{\nu}^{FN}$.}
\end{figure}

Let $S_j = \{w \in \mathbb{C} \mid \frac{2j-3}{6}\pi < \arg w <
\frac{2j-1}{6}\pi \}$ for $j = 1,\ldots,6$,
and let $\Sigma^{FN} = \mathbb C \setminus \bigcup_j S_j$.
Then $\Sigma^{FN}$ consists of six rays $\Sigma_j^{FN}$ for $j=1,\ldots,6$,
all chosen oriented towards infinity as in Figure \ref{figure5}.
The RH problem is the following.

\paragraph{Riemann-Hilbert problem for $\Psi_{\nu}^{FN}$}
\begin{itemize}
\item $\Psi_{\nu}^{FN} : \mathbb{C} \setminus \Sigma^{FN}  \to
    \mathbb C^{2\times 2}$ is analytic,
\item $\Psi_{\nu,+}^{FN} = \Psi_{\nu,-}^{FN}
    \begin{pmatrix} 1 & 0 \\ a_1 & 1 \end{pmatrix}$
    on $\Sigma_1^{FN}$,

    $\Psi_{\nu,+}^{FN} = \Psi_{\nu,-}^{FN}
    \begin{pmatrix} 1 & a_2 \\ 0 & 1 \end{pmatrix}$
    on $\Sigma_2^{FN}$,

    $\Psi_{\nu,+}^{FN} = \Psi_{\nu,-}^{FN}
    \begin{pmatrix} 1 & 0 \\ a_3 & 1 \end{pmatrix}$
    on $\Sigma_3^{FN}$,

    $\Psi_{\nu,+}^{FN} = \Psi_{\nu,-}^{FN}
    \begin{pmatrix} 1 & a_1 \\ 0 & 1 \end{pmatrix}$
    on $\Sigma_4^{FN}$,

    $\Psi_{\nu,+}^{FN} = \Psi_{\nu,-}^{FN}
    \begin{pmatrix} 1 & 0 \\ a_2 & 1 \end{pmatrix}$
    on $\Sigma_5^{FN}$,

    $\Psi_{\nu,+}^{FN} = \Psi_{\nu,-}^{FN}
    \begin{pmatrix} 1 & a_3 \\ 0 & 1 \end{pmatrix}$
    on $\Sigma_6^{FN}$.
\item $\Psi_{\nu}^{FN}(w) = (I + O(1/w)) e^{-i(\frac{4}{3} w^3 + s w) \sigma_3}$
    as $w \to \infty$.
\item If $\nu - \frac{1}{2} \not\in \mathbb N_0$, then
\begin{equation} \label{connection0}
    \Psi_{\nu}^{FN}(w) = B(w)
    \begin{pmatrix} w^{\nu} & 0 \\ 0 & w^{-\nu} \end{pmatrix} E_j,
    \quad \text{ for } w \in S_j,
\end{equation}
where $B$ is analytic. If $\nu \in \frac{1}{2} + \mathbb N_0$, then
there exists a constant $\kappa$ such that
\begin{equation} \label{connection1}
    \Psi_{\nu}^{FN}(w) = B(w)
    \begin{pmatrix} w^{\nu} & \kappa w^{\nu}\log w \\ 0 & w^{-\nu} \end{pmatrix} E_j,
    \quad \text{ for } w \in S_j,
\end{equation}
where $B$ is analytic.
\end{itemize}

The connection matrix $E_1$ is given explicitly in \cite[Chapter 5]{FIKN}.
It is determined (up to inessential left diagonal or upper triangular factors)
by $\nu$ and the Stokes multipliers
$a_1$, $a_2$, and $a_3$, except in the special case
\begin{equation} \label{specialcase}
    \nu = \frac{1}{2} + n, \quad a_1 = a_2 = a_3 = i (-1)^{n+1},
    \quad n \in \mathbb Z, \end{equation}
where an additional parameter $c \in \mathbb C \cup \{\infty\}$ is needed.
For example, for $\nu \not\in \frac{1}{2} + \mathbb N_0$ and $1+a_1a_2 \neq 0$,
we have
\begin{equation}  \label{E1generic0}
    E_1 =
        \begin{pmatrix} d & 0 \\[10pt]
        0 & d^{-1} \end{pmatrix}
        \begin{pmatrix}  1 & \ds \frac{e^{-\nu \pi i} - a_2}{1+a_1a_2}
    \\[10pt]
    -\ds \frac{1+a_1a_2}{2\cos \nu \pi} &
    \ds \frac{e^{\nu \pi i}+a_2}{2 \cos \nu \pi}
    \end{pmatrix}, \end{equation}
where $d \neq 0$ is arbitrary. In the special case \eqref{specialcase}, when
$E_1$ depends on the additional parameter $c \in \mathbb C \cup \{\infty\}$,
by \cite[Chapter 5, (5.0.21)]{FIKN} we may take $E_1$ as
\begin{equation} \label{E1special2}
    E_1  = \begin{pmatrix}
        1 & 0 \\ c & 1 \end{pmatrix}, \quad \mbox{ if } c \in \mathbb C,
        \qquad \mbox{ while} \quad
    E_1 = \begin{pmatrix} 0 & -1 \\ 1 & 0 \end{pmatrix} \quad \mbox{ if } c = \infty.
\end{equation}
Assuming that the branch cuts for the functions in \eqref{connection0}
and \eqref{connection1} are chosen along
$\arg w = -\pi/6$, we obtain the other connection matrices
from $E_1$ through the formula
\begin{equation} \label{Ej}
    E_{j+1} = E_j v_j^{FN}, \qquad j = 1, \ldots, 5,
    \end{equation}
where $v_j^{FN}$ is the jump matrix on $\Sigma_j^{FN}$.
We shall refer to the Stokes multipliers $a_1$, $a_2$, and $a_3$, and
in the special case \eqref{specialcase} also to the additional parameter $c$,
as the monodromy data for Painlev\'e II. We note that in the special case
\eqref{specialcase} we have $\kappa = 0$ in \eqref{connection1}.

The special case \eqref{specialcase} has geometric interpretation.
Indeed, \eqref{specialcase} describes the singular point of the algebraic
variety \eqref{a1a2a3variety}, that is, the point at which the (complex) gradient
of the left-hand side of \eqref{a1a2a3variety} vanishes. The singularity
may be removed by attaching a copy of the Riemann sphere (see also \cite{IK}).

The monodromy data does not depend on $s$. The RH problem is uniquely
solvable, except for a discrete set of $s$-values, and its solution $\Psi_{\nu}^{FN}$
depends on $s$ through the asymptotic condition at infinity. We write
$\Psi_{\nu}^{FN}(w;s)$ if we want to emphasize its dependence on $s$.
If we take
\begin{equation} \label{Recoverq}
    q(s) = 2i \lim_{w \to \infty}
w \left(\Psi_{\nu}^{FN}(w;s)\right)_{12} e^{i(\frac{4}{3} w^3 + s w) \sigma_3},
\end{equation}
then $q$ satisfies the Painlev\'e II equation $q'' = sq + 2q^3 - \nu$.
In addition $\Psi_{\nu}^{FN}$ satisfies the Lax pair for Painlev\'e II
\begin{equation}
\label{FlaschkaNewellSystem1} \frac{\partial}{\partial w} \Psi = L \Psi, \quad L =
    \begin{pmatrix} - 4iw^2 - i(s+2q^2) &
        4wq + 2ir + \frac{\nu}{w} \\[10pt]
        4wq - 2ir + \frac{\nu}{w} &  4iw^2 + i(s+2q^2)
    \end{pmatrix},
\end{equation}
\begin{equation}
\label{FlaschkaNewellSystem2} \frac{\partial}{\partial s} \Psi = P \Psi, \quad P =
    \begin{pmatrix} -iw & q \\ q & iw \end{pmatrix},
\end{equation}
where $q=q(s)$ and $r = r(s) = q'(s)$.
In this way there is a one-to-one correspondence between
monodromy data and solutions of Painlev\'e II.

We also need the more precise asymptotic behavior
\begin{equation} \label{Psinuasymp}
    \Psi_{\nu}^{FN}(w;s) = \left(I + \frac{1}{2i w} \begin{pmatrix} H(s) & q(s) \\ -q(s) & -H(s) \end{pmatrix}
    + O(1/w^2)\right) e^{-i(\frac{4}{3} w^3 +sw) \sigma_3}
  \end{equation}
as $ w \to \infty$, where
\begin{equation} \label{Hdef}
H(s) = (q'(s))^2 - s q^2(s) - q^4(s) + 2\nu q(s)
\end{equation}
is the Hamiltonian for Painlev\'e II. Note that $H' = -q^2$.

We finally note that $\Psi_{\nu}^{FN}$ satisfies the symmetry property
\begin{equation} \label{FNsym}
\Psi_{\nu}^{FN}(w; s) = \sigma_{1}\Psi_{\nu}^{FN}(-w; s)\sigma_{1},
\end{equation}
where $\sigma_1 = \left(\begin{smallmatrix} 0 & 1 \\ 1 & 0 \end{smallmatrix}\right)$.
Indeed, by a straightforward calculation
(see also \cite[Chapter 5]{FIKN}) we check that the function
$ \sigma_{1}\Psi_{\nu}^{FN}(-w; s)\sigma_{1}$
solves exactly the same RH problem as the function $\Psi_{\nu}^{FN}(w; s)$.
Unique solvability of the RH problem yields equation \eqref{FNsym}.

\subsection{Connection with $\Psi_{\alpha}$}
The Hastings-McLeod solution of Painlev\'e II corresponds to the
Stokes multipliers $a_1 = -e^{\nu \pi i}$, $a_2 =0$, and
$a_3 = e^{-\nu \pi i}$. This is not the solution that interests us here.
We use instead the solution corresponding to
\begin{equation} \label{Stokesmultipliers}
a_1 = e^{-\nu \pi i},
    \qquad a_2 = -i,
    \qquad a_3 = -e^{\nu \pi i}.
\end{equation}

For these Stokes multiplies \eqref{Stokesmultipliers}  we obtain from
\eqref{E1generic0} the following connection matrix $E_1$ in case
$\nu \not\in \frac{1}{2} + \mathbb N_0$ (where we take $d = (e^{\nu \pi i} -
i)/(2\cos \nu \pi)$)
\begin{equation} \label{E1generic1}
    E_1 = \begin{pmatrix}
    \ds \frac{e^{\nu \pi i}-i}{2 \cos \nu \pi} &
    \ds \frac{i e^{\nu \pi i}+1}{2 \cos \nu \pi} \\[10pt]
    \ds -e^{-\nu\pi i} &    1
    \end{pmatrix}.
\end{equation}
For $\nu \in \frac{3}{2} + 2 \mathbb N_0$, it follows from
\eqref{Stokesmultipliers} and the formulas in \cite[Chapter 5, (5.0.18)]{FIKN}
that we can take
\begin{equation} \label{E1special1}
    E_1 = \begin{pmatrix}
        1 & 0 \\
        -i & 1
        \end{pmatrix}.
\end{equation}
If $\nu \in \frac{1}{2} + 2 \mathbb N_0$, then
we are in the special case \eqref{specialcase}.
We then choose $c=i$, so that for $\nu \in \frac{1}{2} + 2 \mathbb N_0$
we have monodromy data
\begin{equation} \label{monodromydata}
    a_1 = e^{-\nu \pi i} = -i,
    \qquad a_2 = -i,
    \qquad a_3 = -e^{\nu \pi i} = -i,
    \qquad c = i.
\end{equation}

\begin{lemma} \label{E2E3lemma}
    For any $\nu > -1/2$, we have that
    \begin{equation} \label{E2E3entries}
        (E_2)_{21} = (E_3)_{21} = 0.
        \end{equation}
\end{lemma}
\begin{proof}
In all cases we may check from \eqref{Stokesmultipliers}, \eqref{E1generic1},
\eqref{E1special1}, \eqref{monodromydata}, and \eqref{E1special2} that the
second row of $E_1$ is given by $\begin{pmatrix} -a_1 & 1 \end{pmatrix}$.
So by \eqref{Ej} we have that
$E_2 = E_1 \left(\begin{smallmatrix} 1 & 0 \\ a_1 & 1 \end{smallmatrix}\right)$
is upper triangular. Then also
$E_3 = E_2 \left(\begin{smallmatrix} 1 & a_2 \\ 0 & 1 \end{smallmatrix}\right)$
is upper triangular and therefore \eqref{E2E3entries} holds.
\end{proof}

The following proposition holds for more general monodromy data, and it was established
in \cite{BBIK}, see also \cite{KH}. For the reader's convenience we present
a detailed proof for our particular case.

\begin{proposition} \label{Prop:PsiFNtoPsialpha} {\rm (\cite{BBIK})}
For $\alpha > -1/2$, let $\Psi_{2\alpha+1/2}^{FN}$ be the unique
solution of the RH problem for Painlev\'e II
with parameter $\nu = 2 \alpha + 1/2$ and monodromy data
\eqref{Stokesmultipliers} in case $\alpha \not\in \mathbb N_0$ (so that
$\nu \not\in \frac{1}{2} + 2 \mathbb N_0$), and
monodromy data \eqref{monodromydata} in the special case
$\alpha \in \mathbb N_0$. Then, for any $\eta = \eta(s)$, we
have that
\begin{equation}
\label{PsiFNtoPsialpha}
\Psi_{\alpha}(\zeta;s) = \begin{pmatrix} 1 & 0 \\ \eta(s) & 1 \end{pmatrix}
    \zeta^{-\sigma_3/4} \frac{1}{\sqrt{2}} \begin{pmatrix}
    1 & i \\ i & 1 \end{pmatrix} e^{\pi i \sigma_3/4}
    \Psi_{2\alpha + 1/2}^{FN}(w;-2^{1/3}s) e^{-\pi i \sigma_3/4}
\end{equation}
where $w = e^{\pi i/2} 2^{-1/3} \zeta^{1/2}$ with $\Im w > 0$,
is a solution of the model RH problem for $\Psi_{\alpha}$
given in Subsection \ref{modelRHP1}.
\end{proposition}
\begin{proof}
Because of Proposition \eqref{nonunique} we may take $\eta(s) =0$
without loss of generality.

Clearly $\Psi_{\alpha}$ is analytic on $\mathbb C \setminus \Sigma$.
The correct asymptotics as $\zeta \to \infty$ follows immediately,
as well as the correct jumps across $\Sigma_{1}$, $\Sigma_{2}$,
and $\Sigma_{4}$. A little bit more work is needed to
check the jump across $\Sigma_{3} = (-\infty,0)$ and the behavior
at $z=0$.

In order to analyze the jump across $\Sigma_{3}$, we suppose
that  $\zeta \in \Sigma_{3}$.  Then we have that
\[
w_{+} \equiv e^{\pi i/2} 2^{-1/3} \zeta^{1/2}_{+} = -w_{-}
\equiv  - e^{\pi i/2} 2^{-1/3} \zeta^{1/2}_{-} \quad ( < 0 ),
\]
and hence by \eqref{PsiFNtoPsialpha} and the symmetry property
\eqref{FNsym},
\begin{alignat*}{2}
\Psi_{\alpha, +}(\zeta; s) &=
\zeta^{-\sigma_3/4}_{+} \frac{1}{\sqrt{2}} \begin{pmatrix} 1 & i \\ i & 1 \end{pmatrix} e^{\pi i \sigma_3/4}
\Psi_{2\alpha + 1/2}^{FN}(w_{+};-2^{1/3}s) e^{-\pi i \sigma_3/4}\\
&= \zeta^{-\sigma_3/4}_{-}\begin{pmatrix} -i & 0 \\ 0 & i \end{pmatrix} \frac{1}{\sqrt{2}}
\begin{pmatrix} 1 & i \\ i & 1 \end{pmatrix} e^{\pi i \sigma_3/4}
\Psi_{2\alpha + 1/2}^{FN}(-w_{-};-2^{1/3}s) e^{-\pi i \sigma_3/4}\\
&=\zeta^{-\sigma_3/4}_{-}\
 \frac{1}{\sqrt{2}} \begin{pmatrix}
-i & 1 \\ -1 & i \end{pmatrix} e^{\pi i \sigma_3/4}
\sigma_{1}\Psi_{2\alpha + 1/2}^{FN}(w_{-};-2^{1/3}s)\sigma_{1} e^{-\pi i \sigma_3/4}\\
&=\zeta^{-\sigma_3/4}_{-}\
 \frac{1}{\sqrt{2}} \begin{pmatrix}
-i & 1 \\ 1 & -i \end{pmatrix}
e^{\pi i \sigma_3/4}
\Psi_{2\alpha + 1/2}^{FN}(w_{-};-2^{1/3}s) e^{-\pi i \sigma_3/4}
 \begin{pmatrix}
0 & i \\ -i & 0 \end{pmatrix}\\
&=\zeta^{-\sigma_3/4}_{-}\
 \frac{1}{\sqrt{2}} \begin{pmatrix}
1 & i \\ i & 1 \end{pmatrix}
e^{\pi i \sigma_3/4}
\Psi_{2\alpha + 1/2}^{FN}(w_{-};-2^{1/3}s) e^{-\pi i \sigma_3/4}
 \begin{pmatrix}
0 & 1\\ -1 & 0 \end{pmatrix}\\
&= \Psi_{\alpha, -}(\zeta; s)
 \begin{pmatrix}
0 & 1\\ -1 & 0 \end{pmatrix}.
\end{alignat*}
This shows that $\Psi_{\alpha}$ has the correct jump across $\Sigma_3$,
and it follows that $\Psi_{\alpha}$ satisfies the parts (a), (b), and (c)
of the model RH problem.

Consider now a neighborhood of the point $\zeta =0$. We recall that
\eqref{connection0} or \eqref{connection1} holds with $B(w)$ analytic at $0$.
A corollary  of the symmetry property \eqref{FNsym} is the
equation
\[
B(w)  = \left\{ \begin{array}{ll}
    \sigma_{1}B(-w)\sigma_{3},
    & \text{ if } \nu \not\in \frac{1}{2} + \mathbb N_0, \\[10pt]
     \sigma_1 B(-w)
    \begin{pmatrix} 1 & O(w^{2\nu}) \\ 0 & -1 \end{pmatrix} \text{ as } w \to 0,
    & \text{ if } \nu \in \frac{1}{2} + \mathbb N_0,
     \end{array} \right.
\]
which yields the formula (cf.\ \cite[Chapter 5]{FIKN})
\[
    B(0)\sigma_{3} = \sigma_{1}B(0).
\]
The last relation, together with $\det B(0) = 1$, in turn implies that $B(0)$
can be represented in the form
\[
B(0) = \frac{1}{\sqrt{2}}\begin{pmatrix}
1 & -1 \\ 1 & 1\end{pmatrix}
\begin{pmatrix}
b & 0 \\ 0 & b^{-1}\end{pmatrix}, \qquad b \neq 0.
\]
If $\zeta \in \Omega_j$ then $w \in S_{\pi(j)}$, $j = 1,2,3,4$,
where $\pi$ denotes the permutation
\[ \pi = \begin{pmatrix} 1 & 2 & 3 & 4 \\ 3 & 4 & 1 & 2 \end{pmatrix}. \]
Therefore, for the function $\Psi_{\alpha}(\zeta;s)$ defined by
equation \eqref{PsiFNtoPsialpha} with $\eta(s) = 0$,
we find that (assuming that $\alpha \not\in \frac{1}{2}\mathbb N_0$)
\begin{alignat}{2} \nonumber
\Psi_{\alpha}(\zeta;s) &=  \zeta^{-\sigma_3/4} \frac{1}{\sqrt{2}} \begin{pmatrix}
1 & i \\ i & 1 \end{pmatrix} e^{\pi i \sigma_3/4}B(0)
 e^{-\pi i \sigma_3/4} \left(I + O(\zeta^{1/2})\right) \zeta ^{\sigma_{3}/4}
 \zeta^{\alpha \sigma_{3}} \widetilde{E}_{\pi(j)}\\ \nonumber
&= \zeta^{-\sigma_3/4} \frac{1}{2}
\begin{pmatrix}
1 & i \\ i & 1 \end{pmatrix}
\begin{pmatrix}
1 & -i \\ -i &1\end{pmatrix}
\begin{pmatrix}
b & 0 \\ 0 & b^{-1}\end{pmatrix}
\left(I + O(\zeta^{1/2})\right) \zeta ^{\sigma_{3}/4}
 \zeta^{\alpha \sigma_{3}}\widetilde{E}_{\pi(j)}\\ \nonumber
&= \zeta^{-\sigma_3/4} \left(I + O(\zeta^{1/2})\right) \zeta ^{\sigma_{3}/4}
 \zeta^{\alpha \sigma_{3}}\begin{pmatrix}
b & 0 \\ 0 & b^{-1}\end{pmatrix}\widetilde{E}_{\pi(j)}\\
\label{Psialphaat0}
&=  O(1) \zeta^{\alpha \sigma_{3}}\begin{pmatrix}
b & 0 \\ 0 & b^{-1}\end{pmatrix}\widetilde{E}_{\pi(j)},
    \qquad \text{ as $\zeta \to 0$ in $\Omega_j$},
\end{alignat}
where we have introduced the notation
\begin{equation} \label{deftildeEj}
\widetilde{E}_{j} \equiv
e^{\pi i \sigma_3 / 4} \left(e^{\pi i/2}2^{-1/3}\right)^{(2\alpha + 1/2)\sigma_{3}}E_{j} e^{-\pi i \sigma_3 / 4}.
\end{equation}
From \eqref{Psialphaat0} it immediately follows that
$\Psi_{\alpha}(\zeta;s) = O(\zeta^{-|\alpha|})$ as $\zeta \to 0$,
which is the required behavior in the model RH problem  if $\alpha < 0$, or if $\alpha \geq 0$ and $j \in\{2,3\}$.
If $\alpha \geq 0$ and $j\in \{1,4\}$, then
$\pi(j) \in \{2,3\}$, and it follows from Lemma \ref{E2E3lemma}
and \eqref{deftildeEj} that
\[ \left(\widetilde{E}_{\pi(j)}\right)_{21} = \left(E_{\pi(j)}\right)_{21} = 0. \]
Then \eqref{Psialphaat0} also yields the required behavior of
$\Psi_{\alpha}(\zeta;s)$ as $\zeta \to 0$ in $\Omega_1 \cup \Omega_4$.

The calculation leading to \eqref{Psialphaat0} is valid for $\nu \not\in \frac{1}{2} + \mathbb N_0$,
or $\alpha \not \in \frac{1}{2} \mathbb N_0$. In fact it is also valid if $\alpha \in \mathbb N_0$,
since then we are in the special case \eqref{specialcase} where $\kappa =  0$ in \eqref{connection1}
and so no logarithmic terms appear.
Logarithmic terms only appear if $\alpha \in \frac{1}{2} + \mathbb N_0$, and then
a similar calculation leads to
\[ \Psi_{\alpha}(\zeta;s) = O(1) \zeta^{\alpha \sigma_3}
    \begin{pmatrix} b & 0 \\ 0 & b^{-1} \end{pmatrix}
    \begin{pmatrix} 1 & O(\log \zeta) \\ 0 & 1 \end{pmatrix}
    \widetilde{E}_{\pi(j)}, \]
with $\widetilde{E}_j$ again given by \eqref{deftildeEj}.
Since $\alpha > 0$, the required behavior as $\zeta \to 0$ then follows in a similar way.

This completes the proof of the proposition.
\end{proof}

\subsection{Differential equation}
Recall that $\Psi^{FN}_{\nu}$ has the Lax pair \eqref{FlaschkaNewellSystem1}
and \eqref{FlaschkaNewellSystem2}. Then $\Psi_{\alpha}$ defined by \eqref{PsiFNtoPsialpha}
also satisfies a system of differential equations.
It will involve the solution $q$ of the Painlev\'e II equation with parameter $\nu = 2\alpha + 1/2$
and monodromy data \eqref{Stokesmultipliers} or \eqref{monodromydata}.
We put $r = q'$ and
\begin{alignat}{2}
\label{Udef}
U(s) &= q^2(s) + r(s) + \frac{s}{2},\\
\label{Vdef}
V(s) &= q^2(s) - r(s) + \frac{s}{2}.
\end{alignat}
The functions $U$ and $V$ both satisfy the Painlev\'e XXXIV
equation in a form similar to \eqref{painleve34}, namely
(cf.\ \cite[Chapter 5]{FIKN}):
\begin{alignat}{2} \label{Uequation}
U''(s) &= \frac{(U'(s))^2}{2U(s)} +
2U^2(s) - sU(s) -
\frac{(2\alpha)^2}{2U(s)}, \\
V''(s) &= \label{Vequation}
\frac{(V'(s))^2}{2V(s)} +
2V^2(s) - sV(s) -
\frac{(2\alpha + 1)^2}{2V(s)}.
\end{alignat}

Then we obtain the following differential equations for $\Psi_{\alpha}$.

\begin{lemma} \label{LaxPairP34}
Let $\Psi_{\alpha}$ be given by \eqref{PsiFNtoPsialpha}.
\begin{enumerate}
\item[\rm (a)]
If $\eta \equiv 0$, then $\Psi_{\alpha}$ satisfies
\begin{alignat}{2} \label{DEPsiA}
\frac{\partial}{\partial \zeta} \Psi_{\alpha}(\zeta;s) &= A \Psi_{\alpha}(\zeta;s), \\
\label{DEPsiB}
\frac{\partial}{\partial s} \Psi_{\alpha}(\zeta;s) &= B \Psi_{\alpha}(\zeta;s),
\end{alignat}
where
\begin{alignat}{2} \label{Adef}
A &=
\begin{pmatrix} -2^{1/3} q(-2^{1/3}s) + \frac{\alpha}{\zeta}
& i - i 2^{-1/3} U(-2^{1/3}s)\frac{1}{\zeta} \\[5pt]
-i \zeta + i 2^{-1/3} V(-2^{1/3}s) & 2^{1/3} q(-2^{1/3}s) -
\frac{\alpha}{\zeta} \end{pmatrix}, \\[10pt]
B &= \begin{pmatrix} -2^{1/3} q(-2^{1/3}s) & i
\\[5pt] -i\zeta & 2^{1/3} q(-2^{1/3}s)\end{pmatrix}.
\end{alignat}
\item[\rm (b)] For general $\eta$ we have that $\Psi_{\alpha}$
satisfies
\begin{alignat}{2} \label{DEPsi2}
\frac{\partial}{\partial \zeta} \Psi_{\alpha}(\zeta;s) &=
\begin{pmatrix} 1 & 0 \\ \eta(s) & 1 \end{pmatrix} A
\begin{pmatrix} 1 & 0 \\ -\eta(s) & 1 \end{pmatrix} \Psi_{\alpha}(\zeta;s),
\end{alignat}
with $A$ given by \eqref{Adef}.
\end{enumerate}
\end{lemma}
\begin{proof}
This follows by straightforward calculations from \eqref{FlaschkaNewellSystem1},
\eqref{FlaschkaNewellSystem2}, and \eqref{PsiFNtoPsialpha}.
\end{proof}
The Lax pair \eqref{DEPsiA}--\eqref{DEPsiB}, after the replacement
$\zeta \mapsto \zeta - s$, becomes the Lax pair from \cite{BBIK,KH}.
Equations \eqref{Uequation}--\eqref{Vequation} can also be derived
directly from the compatibility conditions of the Lax pair \eqref{DEPsiA}--\eqref{DEPsiB}
in a usual way.

It is a fact \cite{kapaev:P2}, that the solution $q$ of the Painlev\'e II equation
(with parameter $\nu = 2\alpha + 1/2$ and monodromy data \eqref{Stokesmultipliers}
or \eqref{monodromydata}) has an infinite number of poles on the positive real line,
see also \eqref{P2atplusinfinity} below.
If $-2^{1/3}s$ is such a pole then $\Psi_{2\alpha+1/2}^{FN}(\cdot, -2^{1/3}s)$
does not exist. So to be precise, if we assume that $\eta$ is analytic on $\mathbb R$,
then \eqref{PsiFNtoPsialpha} does not define $\Psi_{\alpha}$ for values of $s \in \mathbb R$
which belong to the discrete set of values $s$ where $q(-2^{1/3}s)$ has poles.

The relation \eqref{PsiFNtoPsialpha} defines $\Psi_{\alpha}$ for all $s \in \mathbb R$
only if we are able to choose $\eta$ so that all the poles on the real line
of the right-hand side of \eqref{PsiFNtoPsialpha} cancel out. Such a choice of $\eta$
would require $\eta$ itself to have poles at the poles
of $q(-2^{1/3}s)$.

We will describe two special choices for $\eta$. The first choice
is such that \eqref{PsiFNtoPsialpha} is equal to the special solution
$\Psi_{\alpha}^{(spec)}$, which is characterized by the asymptotic
condition \eqref{Psispecial}. From Proposition \ref{ExistenceProp} we
know that $\Psi_{\alpha}^{(spec)}$ exists for all $s \in \mathbb R$, so that
we can already conclude that the special choice $\eta = \eta^{(spec)}$ will have
poles at the poles of $q(-2^{1/3}s)$, and that
the real poles of the right-hand side of \eqref{PsiFNtoPsialpha} will indeed cancel out.

The second choice of $\eta$ is made so that the differential
equation \eqref{DEPsi2} takes a nice form. It will lead to
the differential equation \eqref{psisystem} for $\psi_1$
and $\psi_2$. This $\eta$ is denoted $\eta_0$, and
it is defined by the simple formula
\begin{equation} \label{eta0def}
    \eta_0(s) = i 2^{1/3} q(-2^{1/3} s),
\end{equation}
from which it is already clear that it has poles at
the poles of $q(-2^{1/3}s)$. For the choice \eqref{eta0def}  we can already
check that the differential equation \eqref{DEPsi2} leads to
\begin{equation} \label{DEPsi3}
    \frac{\partial}{\partial \zeta} \Psi_{\alpha}(\zeta;s)
    = A_0 \Psi_{\alpha}(\zeta;s) \end{equation}
where
\begin{alignat}{2} \nonumber
    A_0 & =
\begin{pmatrix} 1 & 0 \\ \eta_0 & 1 \end{pmatrix} A
\begin{pmatrix} 1 & 0 \\ -\eta_0 & 1 \end{pmatrix} \\[10pt]
    & = \begin{pmatrix} (\alpha + i u \eta_0)/\zeta & i - iu/\zeta
          \\ - i \zeta + i (v + \eta_0^2) + \eta_0(2 \alpha + i u \eta_0)/\zeta
        & -(\alpha + i u \eta_0)/\zeta \end{pmatrix},
        \label{A0def}
\end{alignat}
and
\begin{align} \label{udef}
     u(s) & = 2^{-1/3} U(-2^{1/3} s), \\
     v(s) & = 2^{-1/3} V(-2^{1/3} s). \label{vdef}
    \end{align}

\subsection{Special choice $\eta^{(spec)}$}

\begin{lemma}
Let $H$ be the Hamiltonian for Painlev\'e II
as in \eqref{Hdef}, with parameter $\nu = 2 \alpha + 1/2$,
and let
\begin{equation} \label{etaspecdef}
    \eta^{(spec)}(s) = i2^{-2/3} \left(q(-2^{1/3} s) + H(-2^{1/3} s)\right).
\end{equation}
Then the choice $\eta = \eta^{(spec)}$ in \eqref{PsiFNtoPsialpha}
leads to the special solution $\Psi_{\alpha}^{(spec)}$ of
the model RH problem characterized by \eqref{Psispecial}.
\end{lemma}
\begin{proof}
It follows from \eqref{Psinuasymp} and \eqref{PsiFNtoPsialpha}
by straightforward computation, that
\begin{alignat}{2}
\nonumber
\Psi_{\alpha}(\zeta;s) &= \begin{pmatrix} 1 & 0 \\ \eta & 1 \end{pmatrix}
    \zeta^{-\sigma_3/4} \frac{1}{\sqrt{2}}
    \begin{pmatrix} 1 & i \\ i & 1 \end{pmatrix} e^{\pi i \sigma_3/4}
    \Psi_{2\alpha+1/2}^{FN}(w;-2^{1/3}s) e^{-\pi i \sigma_3/4} \\
    \nonumber
    & = \begin{pmatrix} 1 & 0 \\ \eta & 1 \end{pmatrix}
    \bigg[ \begin{pmatrix} 1 & 0 \\ -\eta^{(spec)} & 1 \end{pmatrix}
    + \frac{1}{\zeta} \begin{pmatrix} 0 & i2^{-2/3}(H-q)(-2^{1/3}s) \\ 0 & 0 \end{pmatrix}\\
    \nonumber
    & \qquad \qquad + \zeta^{-\sigma_3/4} \frac{1}{\sqrt{2}} \begin{pmatrix} 1 & i \\ i & 1 \end{pmatrix} e^{\pi i \sigma_3/4}
    O(1/\zeta) e^{-\pi i \sigma_3/4} \frac{1}{\sqrt{2}} \begin{pmatrix} 1 & -i \\ -i & 1 \end{pmatrix}
    \zeta^{\sigma_3/4}\bigg] \\
     \label{theorem3formula}
    & \qquad \qquad \qquad \times \zeta^{-\sigma_3/4} \frac{1}{\sqrt{2}} \begin{pmatrix} 1 & i \\ i & 1 \end{pmatrix}
    e^{-(\frac{2}{3}\zeta^{3/2} + s\zeta^{1/2})\sigma_3}
\end{alignat}
as $\zeta \to \infty$.
From \eqref{theorem3formula} it is clear that we need
to take $\eta = \eta^{(spec)}$ in order to be able to
obtain \eqref{Psispecial}. Thus the lemma follows.
\end{proof}

From the calculation \eqref{theorem3formula}
we also note that for any solution $\Psi_{\alpha}$ of
the model RH problem we have
\begin{equation} \label{theorem3formula2}
    \left(\Psi_{\alpha}(\zeta;s) e^{\frac{2}{3}(\zeta^{3/2} + s \zeta^{1/2})\sigma_3}
        \frac{1}{\sqrt{2}} \begin{pmatrix} 1 & - i \\ - i & 1 \end{pmatrix}
        \zeta^{\sigma_3/4} \right)_{12} =
  \frac{i2^{-2/3} (H-q)(-2^{1/3}s)}{\zeta} + O(\zeta^{-3/2})
\end{equation}
as $\zeta \to \infty$. This property  will be used later in the proof
of Theorem \ref{theorem3}.

Since the left-hand side of \eqref{theorem3formula2}
is analytic in $s$ for $s \in \mathbb R$, it also follows
from \eqref{theorem3formula2} that $H-q$ does not have poles on the real line.
This and similar properties are collected in the following lemma.
Recall that $U$ is given by \eqref{Udef}.
\begin{lemma} \label{lem:analytic}
The following hold.
\begin{enumerate}
\item[\rm (a)] $H-q$ has no poles on the real line.
\item[\rm (b)] $U$ has no poles on the real line.
\item[\rm (c)] $U$ has a zero at each of the real poles of $q$
and $Uq$ has no poles on the real line.
\item[\rm (d)] $Uq$ takes the value $\nu-1/2$ at each of
the real poles of $q$.
\end{enumerate}
\end{lemma}
\begin{proof}
(a) We noted already that part (a) follows from \eqref{theorem3formula2}.

(b) Since $H' = -q^2$, we have that
\begin{equation}
\label{poleidentity0}
    U(s) = q^2(s) + q'(s) + s/2 = -(H-q)'(s) + s/2,
\end{equation}
and so it follows from part (a) that $U$ has no poles on the real line either.

(c) Differentiating \eqref{Udef}, we obtain
\begin{equation} \label{poleidentity1}
U'  = 2qq' + q'' + \frac{1}{2}
    = 2q q' + sq + 2q^3 - \nu + \frac{1}{2}
    = 2 U q - \nu + \frac{1}{2}.
    \end{equation}
Thus also $Uq$ has no poles on the real line, which means that
$U$ has a zero at each of the real poles of $q$.

(d) Using \eqref{poleidentity1}, we get
\begin{equation} \label{poleidentity2}
    (Uq - \nu + \frac{1}{2})q = (U' - Uq)q =
    (Uq)' - U(q^2+q') =
    (Uq)' -  U(U - s/2). \end{equation}
Since the right-hand side of \eqref{poleidentity2} is analytic
on the real line by parts (b) and (c), we conclude that $Uq - \nu + \frac{1}{2}$ has
a zero at each of the real poles of $q$. This proves part (d).
\end{proof}

It is well-known and easy to check that
each pole of $q$ is simple and has residue $+1$ or $-1$. Indeed, the Laurent series for
$q$ at a pole $s_0$ has the form
\[ q(s) = \frac{q_{-1}}{s-s_0} + q_1(s-s_0) + \cdots, \]
where $q_{-1} \in \{-1,1\}$. Using this, one easily verifies
that either $q^2 + q'$ or $q^2-q$ is analytic
at $s_0$ (depending on the sign of the residue $q_{-1}$). Our
result that $U = q^2 + q' + s/2$ is analytic on $\mathbb R$
can then also be stated as follows.

\begin{corollary}
The solution $q$ of the Painlev\'e II equation
with parameter $\nu = 2\alpha +1/2$ and monodromy
data \eqref{Stokesmultipliers} or \eqref{monodromydata} has
only simple poles on the real line, with residue $+1$.
\end{corollary}

\subsection{Special choice $\eta_0$}

As already announced we will also use the special
choice $\eta = \eta_0$ given by \eqref{eta0def}.

By \eqref{eta0def} and \eqref{etaspecdef} we have that
\[ \eta_0(s) - \eta^{(spec)}(s)
    = i2^{-2/3} \left(q(-2^{1/3} s) - H(-2^{1/3} s)\right),
\]
and so it follows from part (a) of Lemma \ref{lem:analytic}
that $\eta_0 - \eta^{(spec)}$ is analytic on the real line.
Since $\Psi_{\alpha}^{(spec)}$ exists for all $s \in \mathbb R$,
it follows that the solution of the model RH problem associated with $\eta_0$
exists for all $s \in \mathbb R$ as well, and it is analytic
in $s$.

The differential equation for $\Psi_{\alpha}$ with $\eta = \eta_0$
is given by \eqref{DEPsi3} with $A_0$ as in \eqref{A0def}.
It then follows that $A_0$ is analytic on the real line,
and we will explicitly verify this by rewriting its entries in terms of the
function $u$ from \eqref{udef}
\[ u(s) = 2^{-1/3} U(-2^{1/3}s). \]
The analyticity of $u$ is immediate from \eqref{udef}
and part (b) of Lemma \ref{lem:analytic}. The  analyticity
of $u \eta_0$ follows from \eqref{udef}, \eqref{eta0def}
and part (c) of Lemma \ref{lem:analytic}.
Using also \eqref{poleidentity1} we get
\begin{equation} \label{uprimeequation}
    u' = 2iu\eta_0 + \nu - 1/2 = 2 iu \eta_0 + 2\alpha.
    \end{equation}

Next, it follows from \eqref{Udef}, \eqref{Vdef}, \eqref{udef}, \eqref{vdef},
and \eqref{eta0def} that
\begin{equation} \label{vequation}
    v(s) + \eta_0(s)^2 = - u(s) - s.
    \end{equation}
We can use \eqref{uprimeequation} and \eqref{vequation} to
eliminate $\eta_0$ and $v$ from the entries in $A_0$,
and we get from \eqref{A0def} that
\begin{equation} \label{A0def2}
    A_0 = \begin{pmatrix} u'/(2\zeta) & i - iu/\zeta
          \\ - i \zeta - i (u +s) - i ((u')^2 - (2\alpha)^2)/(4u \zeta)
        & -u'/(2\zeta) \end{pmatrix}.
\end{equation}

\subsection{Proof of Theorem \ref{theorem3} and \ref{theorem4}}

After these preparations the proofs of Theorems \ref{theorem3}
and \ref{theorem4} are short.

\begin{varproof} \textbf{of Theorem \ref{theorem3}.}
From \eqref{Uequation} and \eqref{udef} it follows that $u$
satisfies the Painlev\'e XXXIV equation
in the form \eqref{painleve34}.

From \eqref{theorem3formula2} it follows that
\[ \lim_{\zeta \to \infty}
    \left[ \zeta
    \left(\Psi_{\alpha}(\zeta;s) e^{\frac{2}{3}(\zeta^{3/2} + s \zeta^{1/2})\sigma_3}
        \frac{1}{\sqrt{2}} \begin{pmatrix} 1 & - i \\ - i & 1 \end{pmatrix}
        \zeta^{\sigma_3/4} \right)_{12} \right]
        = i 2^{-2/3} (H-q)(-2^{1/3}s)
        \]
which in view of \eqref{poleidentity0} and \eqref{udef}
leads to \eqref{usolution}.
This proves Theorem \ref{theorem3}.
\end{varproof}

\begin{varproof} \textbf{of Theorem \ref{theorem4}.}
Let $\Psi_{\alpha}$ be the solution of the model RH problem
given by \eqref{PsiFNtoPsialpha} with $\eta = \eta_0$ as in \eqref{eta0def}.
Then
\begin{equation} \label{DEPsi4}
    \frac{\partial}{\partial \zeta} \Psi_{\alpha}(\zeta;s)
    = A_0 \Psi_{\alpha}(\zeta;s), \end{equation}
with $A_0$ given by \eqref{A0def2}.
The differential equation \eqref{DEPsi4} is valid for
$\zeta \in \mathbb C \setminus \Sigma$.
We can take the limit $\zeta \to x$ with $x \in \mathbb R \setminus \{0\}$
to obtain a differential equation for $\Psi_{\alpha,+}(x;s)$,
with the same matrix $A_0$ (but with $\zeta$ replaced by $x$).
Using \eqref{psi12def}, we obtain the differential equation
\eqref{psisystem} for $\psi_1$ and $\psi_2$.
This completes the proof of Theorem \ref{theorem4}.
\end{varproof}

\section{Concluding remarks}
\label{section4}

\subsection{The case $\alpha = 0$} \label{conclusion1}
The case $\alpha = 0$ is classical and well understood. We know that
$K_0^{edge}(x,y;s)$ is the (shifted) Airy kernel, see \eqref{K0}.
We will show here how this follows from the calculations from
the previous section.

In the special case $\alpha = 0$, we have $\nu = 1/2$, and then
the Painlev\'e II equation has special solutions built
out of Airy functions. To be precise if $\Ai$ and $\Bi$
are the standard Airy functions, then for any
$C_1$ and $C_2$, not both zero, we have that
\begin{equation} \label{qnuishalf}
    q(s) = \frac{d}{ds} \log \left(C_1 \Ai(-2^{-1/3}s) + C_2 \Bi(-2^{-1/3}s) \right)
\end{equation}
is a solution of $q'' = sq + 2q^3 - \frac{1}{2}$.
These are exactly the solutions that correspond to the special
Stokes multipliers $a_1 = a_2 = a_3 = -i$. The corresponding solutions
to the RH problem were given by Flaschka and Newell \cite[Section 3F(iv)]{FN}.
For example, for $w$ in sector $S_1$ we have (see also \cite[Chapter 11]{FIKN})
\begin{equation} \label{psiFNnuishalf}
    \Psi_{1/2}^{FN}(w;s) = \frac{\alpha_0}{2} w^{1/2}
    \begin{pmatrix}
        1 - iq(s)/w & - 2^{-1/3} i/w \\
        1 + iq(s)/w & 2^{-1/3} i/w
        \end{pmatrix}
        \begin{pmatrix} \Ai(z) & \Bi(z) \\
        \Ai'(z) & \Bi'(z) \end{pmatrix}
        \begin{pmatrix} -i & 1 \\ 1 & - i
        \end{pmatrix} \end{equation}
with $z = -2^{2/3} w^2 - 2^{-1/3} s$ and $\alpha_0 = 2^{1/6} \sqrt{\pi} e^{i\pi/4}$.
The expressions for $\Psi_{\nu}^{FN}(w;s)$ in the other sectors follow by multiplying
\eqref{psiFNnuishalf} by the appropriate jump matrices.

It follows from \eqref{qnuishalf} and \eqref{psiFNnuishalf} that
the extra parameter $c$ in the monodromy data for \eqref{qnuishalf} is
\begin{equation} \label{cnuishalf}
    c = \frac{iC_1 - C_2}{C_1 - i C_2}.
    \end{equation}
So if we take $c=i$ as in \eqref{monodromydata} then $C_2 =0$
and the corresponding solution \eqref{qnuishalf} is
\begin{equation} \label{qnuishalf2}
    q(s) = \frac{d}{ds} \log \Ai(-2^{-1/3} s) = -2^{-1/3} \frac{\Ai'(-2^{-1/3}s)}{\Ai(-2^{-1/3}s)}.
\end{equation}
Note that the solution \eqref{qnuishalf2} is special among all solutions \eqref{qnuishalf}
in its behavior for $s \to -\infty$. Indeed, from the asymptotic behavior for the
Airy functions it follows that for \eqref{qnuishalf2} we have
\[ q(s) \sim \frac{1}{2} \sqrt{2} (-s)^{1/2} \qquad \mbox{ as } s \to - \infty, \]
while for the other solutions \eqref{qnuishalf} we have
\[ q(s) \sim - \frac{1}{2} \sqrt{2} (-s)^{1/2} \qquad \mbox{ as } s \to -\infty. \]

So according to Proposition \ref{Prop:PsiFNtoPsialpha} we should
be using $q$ given by \eqref{qnuishalf2} and then define $\Psi_0$
as in \eqref{PsiFNtoPsialpha}. If $\zeta$ is in sector
$\Omega_3$, then $w = e^{i\pi/2} 2^{-1/3} \zeta^{1/2}$ is in
sector $S_1$, so that by  \eqref{psiFNnuishalf}
\[ \Psi_{1/2}^{FN}(w; - 2^{1/3} s)
    = \frac{\sqrt{\pi} i}{2} \zeta^{1/4}
        \begin{pmatrix} 1 + i\eta_0(s)\zeta^{-1/2} & - \zeta^{-1/2}  \\
            1 - i \eta_0(s) \zeta^{-1/2} & \zeta^{-1/2} \end{pmatrix}
        \begin{pmatrix} \Ai(\zeta+s) & \Bi(\zeta+s) \\
      \Ai'(\zeta+s) & \Bi'(\zeta+s) \end{pmatrix}
      \begin{pmatrix} - i & 1 \\ 1 & -i \end{pmatrix}
\]
where $\eta_0(s) = i 2^{1/3} q(-2^{1/3})$ as in \eqref{eta0def}.
Then \eqref{PsiFNtoPsialpha} with $\eta = \eta_0$ yields for $\zeta \in \Omega_3$,
\begin{align*} \Psi_0(\zeta;s) & =
    \frac{\sqrt{\pi}}{2\sqrt{2}}
        \begin{pmatrix} 1 & 0 \\ \eta_0(s) & 1 \end{pmatrix} \begin{pmatrix} 1 & 0 \\ 0 & \zeta^{1/2} \end{pmatrix}
        \begin{pmatrix} 1 & i \\ i & 1 \end{pmatrix}
        e^{\pi i\sigma_3/4} \\
        & \qquad \times
        \begin{pmatrix} 1 + i\eta_0(s)\zeta^{-1/2} & - \zeta^{-1/2}  \\
            1 - i \eta_0(s) \zeta^{-1/2} & \zeta^{-1/2} \end{pmatrix}
        \begin{pmatrix} \Ai(\zeta+s) & \Bi(\zeta+s) \\
      \Ai'(\zeta+s) & \Bi'(\zeta+s) \end{pmatrix}
      \begin{pmatrix} 1 & i \\ i & 1 \end{pmatrix}
      e^{-\pi i\sigma_3/4} \\
     & = \frac{\sqrt{\pi}}{\sqrt{2}}
        e^{\pi i\sigma_3/4}
        \begin{pmatrix} \Ai(\zeta+s) & \Bi(\zeta+s) \\
      \Ai'(\zeta+s) & \Bi'(\zeta+s) \end{pmatrix}
      \begin{pmatrix} 1 & i \\ i & 1 \end{pmatrix}
      e^{-\pi i\sigma_3/4} \\
    & =
    \frac{\sqrt{\pi}}{\sqrt{2}}
        \begin{pmatrix} \Ai(\zeta+s) + i \Bi(\zeta+s) & - (\Ai(\zeta+s) - i \Bi(\zeta+s)) \\
        -i (\Ai'(\zeta+s) + i \Bi'(\zeta+s)) & i (\Ai'(\zeta+s) - i \Bi'(\zeta+s))
        \end{pmatrix}.
\end{align*}
Since (see e.g.\ formula (10.4.9) in \cite{AS})
\[ \Ai(z) \pm i \Bi(z) = 2 e^{\pm \pi i/3} \Ai(e^{\mp 2\pi i/3} z) \]
we can write $\Psi_0$ in the more familiar form
\begin{align} \label{Psi0inOmega3}
    \Psi_0(\zeta;s) & =
    \sqrt{2\pi}
        \begin{pmatrix} e^{\pi i/3} \Ai(e^{-2\pi i/3}(\zeta+s)) & - e^{-\pi i/3} \Ai(e^{2\pi i/3}(\zeta+s)) \\
        -i e^{-\pi i/3} \Ai'(e^{-2\pi i/3}(\zeta+s))  & i e^{\pi i/3} \Ai'(e^{2\pi i/3}(\zeta+s))
        \end{pmatrix},
            \quad \mbox{for } \zeta \in \Omega_3.
\end{align}
For $\zeta \in \Omega_1$ we find in a similar way
(or by multiplying \eqref{Psi0inOmega3} on the right by
$\left(\begin{smallmatrix} 1 & 1 \\ -1 & 0 \end{smallmatrix} \right)$), that
\begin{align} \label{Psi0inOmega1}
    \Psi_0(\zeta;s)  & =
    \sqrt{2\pi}
        \begin{pmatrix} \Ai(\zeta+s) &  e^{\pi i/3} \Ai(e^{-2\pi i/3}(\zeta+s)) \\
        -i \Ai'(\zeta+s)  & -i e^{-\pi i/3} \Ai'(e^{-2\pi i/3}(\zeta+s))
        \end{pmatrix},
        \qquad \mbox{for } \zeta \in \Omega_1.
\end{align}
Then it follows from \eqref{psi12def} and \eqref{Psi0inOmega1} that for $x > 0$,
\begin{equation} \label{psi12foralpha=0}
    \psi_1(x;s) = \sqrt{2\pi} \Ai(x+s), \qquad \psi_2(x;s) = - \sqrt{2\pi} i \Ai'(x+s),
    \end{equation}
and a similar calculation shows that \eqref{psi12foralpha=0} also
holds for $x < 0$. Therefore, by \eqref{Kintform}
\begin{align} \nonumber
    K_0^{edge}(x,y;s) & = \frac{\psi_2(x;s) \psi_1(y;s) - \psi_1(x;s) \psi_2(y;s)}{2\pi i(x-y)} \\
    & = \frac{\Ai(x+s) \Ai'(y+s) - \Ai'(x+s) \Ai(y+s)}{x-y},
    \label{K0edge}
    \end{align}
which is indeed the (shifted) Airy kernel.

\subsection{The case $\alpha = 1$} \label{conclusion2}

The case $\alpha = 1$ can be solved explicitly in terms of Airy functions as well.
Let $\Psi_0$ be a solution of the model RH problem with parameter $\alpha = 0$.
Then for any matrix $X = X(s)$, it is easy to check that
\begin{equation} \label{Psi1def}
 \Psi_1(\zeta;s) = (I - \frac{1}{\zeta}X(s)) \Psi_0(\zeta;s)
\end{equation}
satisfies the conditions (a), (b), and (c) of the model RH problem
for $\alpha = 1$. For a special choice of $X$ we will have that
the condition (d) is also satisfied.

Let's take $\Psi_0$ given by \eqref{Psi0inOmega1} for $\zeta \in \Omega_1$.
Then the condition (d) of the model RH problem yields the following
condition on $X$
\begin{equation} \label{Xcondition}
    (I - \frac{1}{\zeta} X(s)) \begin{pmatrix} \Ai(\zeta+s) \\ - i \Ai'(\zeta+s) \end{pmatrix}
    = O(\zeta), \qquad \mbox{ as } \zeta \to 0.
    \end{equation}
The condition \eqref{Xcondition} is satisfied if and only if we take
 \begin{equation} \label{Xdef} X(s) = \frac{1}{\Ai'(s)^2 - s\Ai(s)^2}
    \begin{pmatrix} \Ai(s) \\ -i \Ai'(s) \end{pmatrix}
    \begin{pmatrix} \Ai'(s) & -i \Ai(s) \end{pmatrix}. \end{equation}
Note that the denominator in \eqref{Xdef} cannot be zero for $s \in \mathbb R$.
Indeed, its derivative is $-\Ai(s)^2$, so that it is decreasing for $s\in \mathbb R$,
and since the limit for $s \to +\infty$ is equal to $0$, it follows
that $\Ai'(s) - s \Ai(s)^2 > 0$ for all $s \in \mathbb R$.
Note also that if we take the limit $x,y \to 0$ in \eqref{K0edge}, then
\begin{equation} \label{K0edgeat0}
    K_0^{edge}(0,0;s) = \Ai'(s)^2 - s \Ai(s)^2.
    \end{equation}

Using \eqref{psi12def}, \eqref{Psi0inOmega1}, \eqref{Psi1def}, \eqref{Xdef},
and \eqref{K0edgeat0}, we obtain that
\begin{align*}
    \psi_1(x;s) & =
    \sqrt{2\pi} \Ai(x+s)
    - \sqrt{2\pi}  \frac{\Ai(x+s)\Ai'(s) - \Ai(s) \Ai'(x+s)}{x (\Ai'(s)^2 - s\Ai(s)^2)}
    \Ai(s) \\
    & = \sqrt{2\pi} \left( \Ai(x+s) - \frac{K_0^{edge}(x,0;s)}{K_0^{edge}(0,0;s)} \Ai(s)\right), \\
 \psi_2(x;s) & =
    -\sqrt{2\pi}i \Ai'(x+s)
    + \sqrt{2\pi}i  \frac{\Ai(x+s)\Ai'(s) - \Ai(s) \Ai'(x+s)}{x (\Ai'(s)^2 - s\Ai(s)^2)}
    \Ai'(s) \\
    & = - \sqrt{2\pi}  i \left( \Ai'(x+s) - \frac{K_0^{edge}(x,0;s)}{K_0^{edge}(0,0;s)} \Ai'(s)\right).
        \end{align*}
Thus
\begin{align} \nonumber
    K_1^{edge}(x,y;s) & = \frac{\psi_2(x;s) \psi_1(y;s) - \psi_1(x;s) \psi_2(y;s)}{2\pi i(x-y)} \\
    & = K_0^{edge}(x,y;s) - \frac{K_0^{edge}(x,0;s) K_0^{edge}(y,0;s)}{K_0^{edge}(0,0;s)}.
    \label{K1edge}
    \end{align}

To compute the relevant solution $u$ of the Painlev\'e XXXIV equation for $\alpha = 1$,
we may assume that we have taken $\Psi_0^{spec}$ in \eqref{Psi1def}, and then
use \eqref{usolution}, \eqref{Psi1def}, \eqref{Psispecial}, and the fact that $u\equiv 0$
for $\alpha=0$, to obtain that $u(s) = i X_{12}'(s)$,
which by \eqref{Xdef} leads to
\begin{equation} \label{udefalpha=1}
    u(s) = \frac{d}{ds} \left( \frac{\Ai(s)^2}{\Ai'(s)^2-s\Ai(s)^2} \right)
        = - \frac{d^2}{ds^2} \log K_0^{edge}(0,0;s).
\end{equation}
Its graph is shown in Figure \ref{figure6}.

\begin{figure}[t]
\centerline{
\includegraphics[scale=0.5,angle=-90]{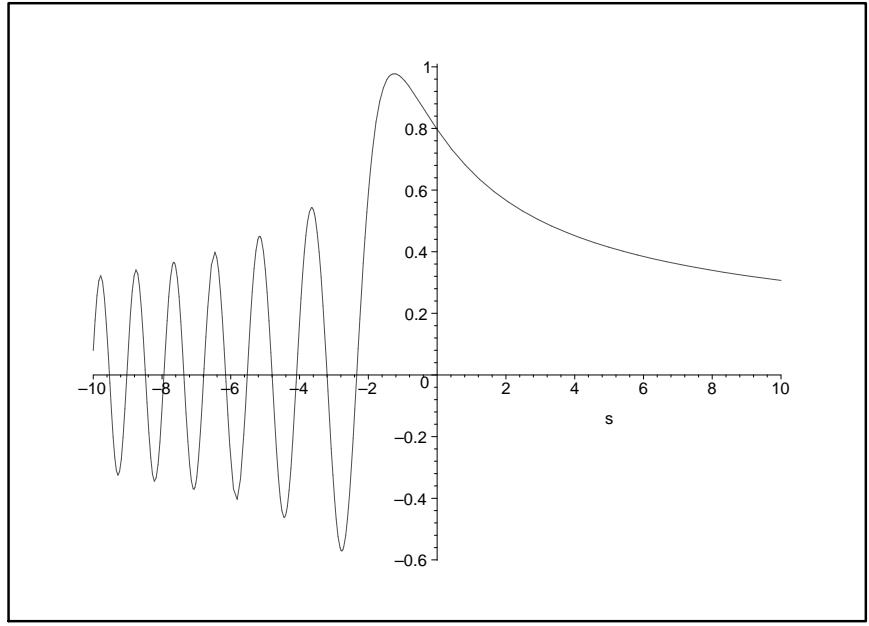}}
\caption{\label{figure6} The solution of the Painlev\'e XXXIV equation
for $\alpha = 1$.}
\end{figure}

One can verify from the explicitly known asymptotic formulas
for $\Ai$ that
\begin{equation} \label{uasympalpha=1}
    u(s) = \frac{1}{\sqrt{s}} - \frac{1}{s^2} + O(s^{-7/2})
        \qquad \mbox{ as } s \to +\infty.
        \end{equation}
On the negative real axis, $u$ has an infinite number of
zeros. These are the zeros of the Airy function $\Ai$,
and an infinite number of additional zeros that interlace
with the zeros of $\Ai$.

Equations \eqref{Psi1def} and \eqref{udefalpha=1} constitute the
Schlesinger and (induced by it) B\"acklund
transformations, respectively, for the case of Painlev\'e XXXIV and
applied to its zero vacuum solution
(for the general theory of Schlesinger transformations see \cite{JMU};
see also \cite[Chapter 6]{FIKN}).

\subsection{Asymptotic characterization of the Painlev\'e function $u(s)$}
\label{conclusion3}

We finally want to characterize the solution $u$ of the Painlev\'e XXXIV
equation by its asymptotic properties. Recall that $u$ is connected to
the solution of the Painlev\'e II equation $q'' = sq + 2q^3 - \nu$ with
$\nu = 2 \alpha + 1/2$ by the formulas
\begin{equation} \label{uqrelation}
    u(s) = 2^{-1/3} U(-2^{1/3} s), \qquad  U(s) = q^2(s) + q'(s) + \frac{s}{2}.
    \end{equation}

Assume that $\nu > - 1/2$.
It is shown in \cite{kapaev:HM} (see also \cite[Chapters 5, 11]{FIKN})
that the solution $q(s)$ of the Painlev\'e II equation corresponding to the Stokes multipliers
\eqref{Stokesmultipliers} exhibits the following asymptotic behavior in the sector
$\arg s \in \bigl[\frac{2\pi}{3}, \frac{4\pi}{3}\bigr]$
\begin{alignat}{2}
\nonumber
q(s) &= \sqrt{-\frac{s}{2}}\sum_{n=0}^{[\nu +1/2]}b_n(-s)^{-3n/2} + O\bigl(s^{-3[\nu + 1/2]/2 -1}\bigr)\\
\nonumber
& \qquad + c_{+}(-s)^{-\frac{3}{2}\nu-\frac{1}{4}} e^{-\frac{2\sqrt2}{3}(-s)^{3/2}} \bigl(1+{O}(s^{-1/4})\bigr),\\
\label{trong+4}
& \qquad \qquad \text{ as } s \to \infty, \quad
\arg s \in \bigl(\frac{2\pi}{3}, \frac{4\pi}{3}\bigr],
\quad \arg(-s) \in \bigl(-\frac{\pi}{3}, \frac{\pi}{3}\bigr],
\end{alignat}
or
\begin{alignat}{2}
\nonumber
q(s) &= \sqrt{-\frac{s}{2}}\sum_{n=0}^{[\nu +1/2]}b_n(-s)^{-3n/2} + O\bigl(s^{-3[\nu +1/2]/2}\bigr)\\
\nonumber
& \qquad + c_{-}(-s)^{-\frac{3}{2}\nu-\frac{1}{4}} e^{-\frac{2\sqrt2}{3}(-s)^{3/2}} \bigl(1+{O}(s^{-1/4})\bigr),\\
\label{trong+5}
& \qquad \qquad \text{ as } s \to \infty , \quad
\arg s \in \bigl[\frac{2\pi}{3}, \frac{4\pi}{3}\bigr),
\quad \arg(-s) \in \bigl[-\frac{\pi}{3}, \frac{\pi}{3}\bigr),
\end{alignat}
where we have used the notation $[r]$ for the integer part of the
positive number $r$, i.e.
\[
[r] \in {\mathbb N}_{0}, \qquad [r] \leq r < [r]+1.
\]
The coefficients $c_{+}$ and $c_{-}$  of the exponential terms, which
oscillate on the respective  boundaries of the sector
$\bigl(\frac{2\pi}{3}, \frac{4\pi}{3}\bigr)$, are given by the formulae
\begin{equation}\label{c0}
c_{+} = -\frac{e^{\pi(\nu+\frac{1}{2})i}+1}{\pi}
2^{-\frac{5}{2}\nu-\frac{7}{4}}
\Gamma(\tfrac{1}{2}+\nu),
\end{equation}
and
\begin{equation}\label{c1}
c_{-} = -\frac{e^{-\pi(\nu+\frac{1}{2})i}+1}{\pi}
2^{-\frac{5}{2}\nu-\frac{7}{4}}
\Gamma(\tfrac{1}{2}+\nu),
\end{equation}
where $\Gamma$ denotes the Gamma function.
Moreover, either the relations \eqref{trong+4}, \eqref{c0}
or the relations \eqref{trong+5}, \eqref{c1} can be taken
as a characterization of the solution $q(s)$.

Alternatively, the solution $q(s)$ can be characterized by its
comparison to the Boutroux tri-tronqu\'ee solution
$q^{(tri-tronq)}(s)$ of the Painlev\'ee II equation, which is
defined as the unique solution satisfying the asymptotic condition
\begin{alignat}{2}
\nonumber
q^{(tri-tronq)}(s) & \sim \sqrt{-\frac{s}{2}}\sum_{n=0}^{\infty}b_n(-s)^{-3n/2},\\
\label{tritrong}
& \qquad \text{ as } s \to \infty, \quad \arg s \equiv \pi + \arg(-s)  \in \bigl(0, \frac{4\pi}{3}\bigr).
\end{alignat}
The solution $q(s)$ we are working with is the one whose asymptotic
behavior as $s \to -\infty$ is given by the equation
\begin{alignat}{2}
\nonumber
q(s) - q^{(tri-tronq)}(s) & = -\frac{e^{-\pi(\nu+\frac{1}{2})i}+1}{\pi} 2^{-\frac{5}{2}\nu-\frac{7}{4}}
\Gamma(\tfrac{1}{2}+\nu)\\
\label{qtritronq}
& \qquad \times |s|^{-\frac{3}{2}\nu-\frac{1}{4}}e^{-\frac{2\sqrt2}{3}|s|^{3/2}}
\bigl(1+{O}(s^{-1/4})\bigr), \quad \text { as } s \to -\infty.
\end{alignat}

The coefficients $b_{n}$ of the asymptotic series in \eqref{trong+4},
\eqref{trong+5}, and \eqref{tritrong} are determined by substitution
into the Painlev\'e II equation. Indeed, the following recurrence relation
takes place
\begin{equation}\label{bn_recurrence}
\left\{ \begin{array}{l}
\ds b_0 = 1,\quad b_1 = \frac{\nu}{\sqrt2},
\\[10pt]
\ds b_{n+2} = \frac{9n^2-1}{8}b_{n} - \sum\limits_{m=1}^{n+1}b_mb_{n+2-m} -
\frac{1}{2}\sum\limits_{l=1}^{n+1}\sum\limits_{m=1}^{n+2-l}b_lb_mb_{n+2-l-m}.
\end{array}
\right.
\end{equation}

Using relation \eqref{uqrelation} between the Painlev\'e II and
Painlev\'e XXXIV functions
we arrive at the asymptotic characterization of the 
function $u(s)$ of Theorem \ref{theorem4}.
\begin{proposition} \label{uasymptotics}
The solution $u(s)$ of the Painlev\'e XXXIV
equation which appears in Theorem \ref{theorem4} is uniquely
characterized by one of the following asymptotic conditions

\begin{alignat}{2}
\nonumber
u(s) & = \frac{\alpha}{\sqrt{s}} + \sum_{n=1}^{[2\alpha +1]}a_n s^{-\frac{3n+1}{2}}
+O\bigl(s^{-3[2\alpha + 1]/2 -1}\bigr)\\
\nonumber
& \qquad + d_{+}s^{-3\alpha+\frac{1}{2}} e^{-\frac{4}{3}s^{3/2}} \bigl(1+{O}(s^{-1/4})\bigr),\\
\label{trong+44}
& \qquad \qquad \text{ as } s \to \infty, \quad
\arg s  \in \bigl(-\frac{\pi}{3}, \frac{\pi}{3}\bigr],
\end{alignat}
or
\begin{alignat}{2}
\nonumber
u(s) & = \frac{\alpha}{\sqrt{s}} + \sum_{n=1}^{[2\alpha +1]}a_n s^{-\frac{3n+1}{2}}
+O\bigl(s^{-3[2\alpha + 1]/2 -1}\bigr)
\\
\nonumber
& \qquad + d_{-}s^{-3\alpha+\frac{1}{2}} e^{-\frac{4}{3}s^{3/2}} \bigl(1+{O}(s^{-1/4})\bigr),\\
\label{trong+55}
& \qquad \qquad  \text{ as } s \to \infty, \quad \arg s  \in \bigl[-\frac{\pi}{3}, \frac{\pi}{3}\bigr),
\end{alignat}
where
\begin{equation}\label{dpm}
d_{\pm} = \frac{e^{\pm 2\alpha\pi i}-1}{\pi} 2^{-6\alpha-\frac{5}{3}} \Gamma(1+2\alpha).
\end{equation}
Alternatively, the solution $u(s)$ can be characterized
by the asymptotic relation
\begin{alignat}{2}
\nonumber
u(s) - u^{(tri-tronq)}(s) & = -\frac{e^{-2\alpha\pi i}-1}{\pi} 2^{-6\alpha-\frac{5}{3}} \Gamma(1+2\alpha)\\
\label{qtritronqP34}
& \qquad \times s^{-3\alpha+\frac{1}{2}}e^{-\frac{4}{3}s^{3/2}} \bigl(1+{O}(s^{-1/4})\bigr),
\qquad \text{ as } s \to +\infty.
\end{alignat}
The Painlev\'e XXXIV {\it tri-tronqu\'ee} solution $u^{(tri-tronq)}(s)$ is
determined by the asymptotic condition
\begin{equation}\label{tritrongP34}
u^{(tri-tronq)}(s) \sim \frac{\alpha}{\sqrt{s}} + \sum_{n=1}^{\infty}a_n s^{-\frac{3n+1}{2}},
\quad \text{ as } s \to \infty, \quad  \arg s  \in \bigl(-\pi, \frac{\pi}{3}\bigr).
\end{equation}
Finally, the coefficients $a_{n}$ of the asymptotic series above can be
expressed in terms of the coefficients $b_{n}$ defined in \eqref{bn_recurrence},
with $\nu$ replaced by $2\alpha + 1/2$:
\[
2^{\frac{n+1}{2}}a_{n} = b_{n+1} -\frac{3n-2}{2\sqrt{2}}b_{n} +\frac{1}{2}\sum_{k,m \geq 1;k+m = n+1} b_{k}b_{m}.
\]
\end{proposition}

\begin{remark} \label{rem:uasymptotics}
The leading asymptotics of the Painlev\'e II function $q(s)$ as $s\to +\infty$
is known (see \cite{kapaev:P2}; see also \cite[Chapter 10]{FIKN}). Unfortunately,
the leading term is not enough to derive the corresponding asymptotics as $s \to -\infty$
of the Painlev\'e XXXIV function $u(s)$. Indeed, the leading asymptotics of
$q(s)$ as $s \to +\infty$ is of the form
\begin{equation}
\label{P2atplusinfinity}
q(s) \sim \sqrt{\frac{s}{2}} \cot \left(\frac{\sqrt{2}}{3}s^{3/2} + \chi \right),
\end{equation}
(the phase $\chi$ is known) and it cancels out in the right-hand side of equation
\eqref{Udef}. The better way to study the large negative $s$ asymptotics
of the function $u(s)$ is via the direct analysis of the model RH problem for $\Psi_{\alpha}$.
The case $\alpha = 1$ shows that we might expect oscillating behavior as $s \to -\infty$
(see Figure \ref{figure6}) and indeed, assuming that $\alpha-1/2 \not\in \mathbb N_0$,
we are able to show that
\begin{equation} \label{uatminusinfinity}
u(s) = \frac{\alpha}{\sqrt{-s}}\cos \left(\frac{4}{3}(-s)^{3/2} -\alpha
\pi\right)
+ O(1/s^2), \qquad
\mbox{as } s \to -\infty.
\end{equation}
The proof of \eqref{uatminusinfinity} will be given in a future publication.
Moreover, we conjecture that asymptotics \eqref{uatminusinfinity} determines
the solution $u(s)$ uniquely.
\end{remark}

\section*{Acknowledgements}

Alexander Its was supported in part by NSF grant \#DMS-0401009.
Arno Kuijlaars is supported by FWO-Flanders project G.0455.04,
by K.U. Leuven research grant OT/04/21, by the Belgian Interuniversity
Attraction Pole P06/02, by the  European Science Foundation Program MISGAM,
and by a grant from the Ministry of Education and
Science of Spain, project code MTM2005-08648-C02-01.
J\"{o}rgen \"{O}stensson is supported by K.U. Leuven research
grant OT/04/24.


\begin{thebibliography}{99}
\bibitem{AS} M. Abramowitz and I. Stegun,
    Handbook of Mathematical Functions,
    Dover Publications, New York, 1992. Reprint of the 1972 edition.
\bibitem{BBIK} P. Bleher, A. Bolibruch, A. Its, and A. Kapaev,
    Linearization of the P34 equation of Painlev\'e-Gambier,
    unpublished manuscript.
\bibitem{BI1}
    P. Bleher and A. Its,
    Semiclassical asymptotics of orthogonal polynomials,
    Riemann-Hilbert problem, and universality in the matrix model,
    Ann. Math. 150 (1999), 185--266.
\bibitem{BI2}
    P. Bleher and A. Its,
    Double scaling limit in the random matrix model: the
    Riemann-Hilbert approach,
    Comm. Pure Appl. Math. 56 (2003), 433--516.
\bibitem{BD} A. Borodin and P. Deift,
    Fredholm determinants, Jimbo-Miwa-Ueno $\tau$-functions, and representation theory,
    Comm. Pure Appl. Math. 55 (2002), 1160--1230.
\bibitem{BB} M.J. Bowick and E. Br\'ezin,
    Universal scaling of the tail of the density of eigenvalues in random matrix models,
    Phys. Lett. B  268 (1991), 21--28.
\bibitem{Carlson}
    F. Carlson, Sur une classe de s\'eries de Taylor,
    Dissertation, Uppsala, Sweden, 1914.
\bibitem{CK}
    T. Claeys and A.B.J. Kuijlaars,
    Universality of the double scaling limit in random matrix models,
    Comm. Pure Appl. Math. 59 (2006), 1573--1603.
\bibitem{CKV}
    T. Claeys, A.B.J. Kuijlaars and M. Vanlessen,
    Multi-critical unitary random matrix ensembles and the general
    Painlev\'e II equation,
    to appear in Annals of Mathematics.
\bibitem{CV}
    T. Claeys and M. Vanlessen,
    Universality of a double scaling limit near singular
    edge points in random matrix models,
    arxiv: math-ph/0607043,
    to appear in Comm. Math. Phys.
\bibitem{Deift}
    P. Deift,
    Orthogonal Polynomials and Random Matrices: A  Riemann-Hilbert Approach,
    Courant Lecture Notes 3, New York University, 1999.
\bibitem{DG}
    P. Deift and D. Gioev,
    Universality at the edge of the spectrum for unitary, orthogonal
    and symplectic ensembles of random matrices,
    to appear in Comm. Pure Appl. Math.
\bibitem{DKM}
    P. Deift, T. Kriecherbauer, and K.T-R McLaughlin,
    New results on the equilibrium measure for logarithmic potentials
    in the presence of an external field,
    J. Approx. Theory 95 (1998), 388--475.
\bibitem{DKMVZ2}
    P. Deift, T. Kriecherbauer, K.T-R McLaughlin, S. Venakides, and X. Zhou,
    Uniform asymptotics for polynomials orthogonal with respect to
    varying exponential weights and applications to universality
    questions in random matrix theory,
    Comm. Pure Appl. Math. 52 (1999), 1335--1425.
\bibitem{DKMVZ1}
    P. Deift, T. Kriecherbauer, K.T-R McLaughlin, S. Venakides,
    and X. Zhou,
    Strong asymptotics of orthogonal polynomials with respect to
    exponential weights,
    Comm. Pure Appl. Math. 52 (1999), 1491--1552.
\bibitem{DeiftZhou}
    P. Deift and X. Zhou,
    A steepest descent method for oscillatory Riemann-Hilbert problems.
    Asymptotics for the MKdV equation,
    Ann. Math. 137 (1993), 295--368.
\bibitem{DeiftZhou2}
    P. Deift and X. Zhou,
    Long-time asymptotics for solutions of the NLS equation with
    initial data in a weighted Sobolev space,
    Comm. Pure Appl. Math. 56 (2003), 1029--1077.
\bibitem{DeiftZhou3}
    P. Deift and X. Zhou,
    Perturbation theory for infinite-dimensional integrable systems
    on line. A case study,
    Acta Math. 188 (2002), 163--262.
\bibitem{DeiftZhou4}
    P. Deift and X. Zhou,
    A priori $L^p$-estimates for solutions of Riemann-Hilbert problems,
    Int. Math. Research Notices  2002 (2002), 2121--2154.
\bibitem{DK}
    M. Duits and A.B.J. Kuijlaars,
    Painlev\'e I asymptotics for orthogonal polynomials with respect
    to a varying quadratic weight, Nonlinearity 19 (2006), 2211--2245.
\bibitem{Dyson}
    F.J. Dyson, Correlation between the eigenvalues of a random matrix,
    Comm. Math. Phys. 19 (1970), 235--250.
\bibitem{FN}
    H. Flaschka and A.C. Newell,
    Monodromy and spectrum-preserving deformations I,
    Comm. Math. Phys. 76 (1980), 65--116.
\bibitem{FIK}
    A.S. Fokas, A.R. Its, and A.V. Kitaev,
    The isomonodromy approach to matrix models in 2D quantum gravity,
    Comm. Math. Phys. 147 (1992), 395--430.
\bibitem{FIKN}
    A.S. Fokas, A.R. Its, A.A. Kapaev and V.Yu. Novokshenov,
    Painlev\'e Transcendents, the Riemann-Hilbert approach,
    Math. Surveys and Monogr. 128, Amer. Math. Soc., Providence RI, 2006
\bibitem{FZ}
    A.S. Fokas and X. Zhou,
    On the solvability of Painlev\'e II and IV,
    Commun. Math. Phys. 144 (1992), 601--622.
\bibitem{For}
    P.J. Forrester, The spectrum edge of random matrix ensembles,
    Nucl. Phys. B 402 (1993), 709--728.
\bibitem{HI}
    J. Harnad and A.R. Its,
    Integrable Fredholm operators and dual isomonodromic deformations,
    Comm. Math. Phys. 226 (2002), 497--530.
\bibitem{HMcL}
    S.P. Hastings and J.B. McLeod,
    A boundary value problem associated with the second Painlev\'e transcendent
    and the Korteweg-de Vries equation,
    Arch. Rational Mech. Anal. 73 (1980), 31--51.
\bibitem{Ince}
    E.L. Ince, Ordinary Differential Equations, Dover, New York, 1956.
\bibitem{IK}
    A.R. Its and A.A. Kapaev,
    The irreducibility of the second Painlev\'e equation and the isomonodromy method.
    In: Toward the exact WKB analysis of differential equations, linear or non-linear,
    C.J. Howls, T. Kawai, and Y. Takei, eds.,
    Kyoto Univ. Press, 2000, pp. 209--222.
\bibitem{JMU}
    M. Jimbo, T. Miwa, and K. Ueno,
    Monodromy preserving deformation of linear ordinary differential equations
    with rational coefficients,
    Physica D 2 (1981), 306--352.
\bibitem{KMM}
    S. Kamvissis, K.D.T-R McLaughlin, and P.D. Miller,
    Semiclassical Soliton Ensembles for the focusing Nonlinear
    Schr\"odinger Equation,
    Ann. Math. Studies 154, Princeton Univ. Press,
    Princeton, 2003.
\bibitem{kapaev:P2}
    A.A. Kapaev, Global asymptotics of the second Painlev\'e
    transcendent,  Phys.\ Lett.\ A, 167 (1992) 356--362.
\bibitem{kapaev:HM}
    A.A. Kapaev, Quasi-linear Stokes phenomenon for the Hastings-McLeod
    solution of the second Painlev\'e equation,
    arXiv: nlin.SI/0410009
\bibitem{KH}
    A.A. Kapaev and E. Hubert, A note on the Lax pairs for Painlev\'e equations,
    J. Phys. A: Math. Gen. 32 (1999), 8145--8156.
\bibitem{KM}
    A.B.J. Kuijlaars and K.T-R McLaughlin,
    Generic behavior of the density of states in random matrix theory
    and equilibrium problems in the presence of real analytic external
    fields,
    Comm. Pure Appl. Math. 53 (2000), 736--785.
\bibitem{KV}
    A.B.J. Kuijlaars and M. Vanlessen,
    Universality for eigenvalue correlations at the origin of the spectrum,
    Comm. Math. Phys. 243 (2003), 163--191.
\bibitem{Mehta}
    M.L. Mehta,
    Random Matrices, 2nd. ed.
    Academic Press, Boston, 1991.
\bibitem{Moore}
    G. Moore, Matrix models of 2D gravity and isomonodromic deformations,
    Progr. Theor. Phys. Suppl. 102 (1990), 255--285.
\bibitem{SaTo}
    E.B. Saff and V. Totik,
    Logarithmic Potentials with External Fields,
    Springer-Verlag, New-York, 1997.
\bibitem{ReedSimon}
     M. Reed and B. Simon,
     Methods of Modern Mathematical Physics IV,
     Academic Press, New York-London, 1978.
\bibitem{TW}
    C. Tracy and H. Widom,
    Level spacing distributions and the Airy kernel,
    Comm. Math. Phys.  159  (1994), 151--174.
\bibitem{TW2}
    C. Tracy and H. Widom,
    Airy kernel and Painlev\'e II. In:
    Isomonodromic Deformations and Applications in Physics,
    (J. Harnad and A. Its, eds),
    CRM Proc. Lecture Notes, 31, Amer. Math. Soc., Providence, RI, 2002.
    pp.~85--96.
\bibitem{Zhou}
    X. Zhou,
    The Riemann-Hilbert problem and inverse scattering,
    SIAM J. Math. Anal. 20 (1989), 966--986.
\end{thebibliography}
\end{document}